\documentclass[10pt,reqno,a4paper]{article} 

\usepackage{anysize}
\marginsize{3.0cm}{3.0cm}{2.2cm}{2.2cm}

\usepackage[english]{babel}
\usepackage[centertags]{amsmath}
\usepackage{amsfonts,amssymb,amsthm}
\usepackage[svgnames]{xcolor}
\usepackage{comment}
\usepackage{hyperref} 
\usepackage{mathrsfs}
\usepackage[sans]{dsfont}
\usepackage{accents}

\let\OLDthebibliography\thebibliography
\renewcommand\thebibliography[1]{
  \OLDthebibliography{#1}
  \setlength{\parskip}{0pt}
  \setlength{\itemsep}{0pt plus 0.3ex} }

\numberwithin{equation}{section}

\theoremstyle{plain}
\newtheorem{theorem}{Theorem}[section]
\newtheorem{proposition}[theorem]{Proposition}
\newtheorem{lemma}[theorem]{Lemma}

\theoremstyle{definition}
\newtheorem{remarkbase}[theorem]{Remark}
\newenvironment{remark}{\pushQED{\qed} \remarkbase}{\popQED\endremarkbase}

\renewcommand{\Im}{\mathrm{Im}\,}
\renewcommand{\Re}{\mathrm{Re}\,}

\newcommand{\N}{{\mathbb N}}
\newcommand{\R}{{\mathbb R}}
\newcommand{\C}{{\mathbb C}}
\newcommand{\Z}{\mathbb Z}

\newcommand{\T}{{\mathbb T}}
\renewcommand{\S}{{\mathbb S}}

\newcommand{\mE}{\mathcal{E}}
\newcommand{\mF}{\mathcal{F}}

\newcommand{\mH}{\mathcal{H}}

\newcommand{\mJ}{\mathcal{J}}

\newcommand{\mL}{\mathcal{L}}
\newcommand{\mM}{\mathcal{M}}

\newcommand{\mT}{\mathcal{T}}

\renewcommand{\a}{\alpha}

\newcommand{\g}{\gamma}
\renewcommand{\d}{\delta}

\newcommand{\e}{\varepsilon}
\newcommand{\ph}{\varphi}
\newcommand{\lm}{\lambda}

\newcommand{\Om}{\Omega}
\newcommand{\om}{\omega}

\newcommand{\s}{\sigma}

\renewcommand{\th}{\vartheta}

\newcommand{\la}{\langle}
\newcommand{\ra}{\rangle}
\newcommand{\pa}{\partial}
\renewcommand{\div}{\mathrm{div}\,}
\newcommand{\grad}{\nabla}

\newcommand{\pois}{\mM}
\newcommand{\res}{\mathrm{res}}

\title{Liquid drop with capillarity and rotating traveling waves}
 
\author{\normalsize{Pietro Baldi, Vesa Julin, Domenico Angelo La Manna}}
\date{} 

\begin{document}

\maketitle

\noindent
\textbf{Abstract.} 
We consider the free boundary problem for a 3-dimensional, incompressible, irrotational liquid drop 
of nearly spherical shape with capillarity. 
We study the problem from the beginning, 
extending some classical results from the flat case (capillary water waves) to the spherical geometry: 
the reduction to a problem on the boundary, its Hamiltonian structure, 
the analyticity and tame estimates for the Dirichlet-Neumann operator in Sobolev class, 
and a linearization formula for it, 
both with the method of the good unknown of Alinhac
and by a geometric approach.
Then, also thanks to the analyticity of the operators involved, 
we prove the bifurcation of traveling waves, 
which are nontrivial (i.e., nonspherical) fixed profiles 
rotating with constant angular velocity. 
To the best of our knowledge, this is the first example 
of global-in-time nontrivial solutions 
of the free boundary problem for the capillary liquid drop.

\tableofcontents

\section{Introduction and main results} 
\label{sec:intro}

We consider the free boundary problem for a liquid drop with capillarity, 
a problem already considered by Lord Rayleigh \cite{Rayleigh}.  
We do not consider gravity forces; in fact, a liquid drop with capillarity
falling in the vacuum under gravity 
is described by the same system, 
as gravity can be removed from the equations 
by considering a reference frame that falls together with the drop.

We assume that the liquid drop occupies the time-dependent, open bounded region 
$\Om_t \subset \R^3$ with smooth boundary $\pa \Om_t$ for some time interval $t \in (0,T)$,  
and that the velocity vector field $u$ and the pressure $p$ are defined in $\Om_t$. 
Since the only effecting force is the surface tension, then  $\Om_t$, $u$ and $p$ satisfy 
\begin{align}
\pa_t u + u \cdot \grad u + \grad p 
& = 0 \quad \text{in} \ \Om_t, 
\label{dyn.eq.01} \\
\div u 
& = 0 \quad \text{in} \ \Om_t,
\label{div.eq.01}
\\
p & = \sigma_0 H_{\Om_t} \quad \text{on} \ \pa \Om_t, 
\label{pressure.eq.01}
\\
V_t & = \la u , \nu_{\Om_t} \ra \quad \ \text{on} \ \pa \Om_t,
\label{kin.eq.01}
\end{align}
where $\sigma_0$ is the capillarity coefficient, 
$H_{\Om_t}$ is the mean curvature of the boundary $\pa \Om_t$, 
$V_t$ is the normal velocity of the boundary, 
$\nu_{\Om_t}$ is the unit outer normal of the boundary, 
and $\la \cdot , \cdot \ra$ denotes the scalar product of vectors in $\R^3$. 
Equations \eqref{dyn.eq.01}, \eqref{div.eq.01} are 
the Euler equations of incompressible fluid mechanics, 
\eqref{pressure.eq.01} gives the pressure at the boundary 
in terms of capillarity, and \eqref{kin.eq.01} is the assumption 
that the movement of the boundary $\pa \Om_t$ 
in its normal direction is due to the movement of  the liquid particles on $\pa \Om_t$, that is, 
the velocity of $\pa \Om_t$ and the vector field $u$ must have 
the same normal component at the boundary $\pa \Om_t$.  
We call \eqref{dyn.eq.01} the dynamics equation, 
\eqref{div.eq.01} the incompressibility condition,
\eqref{pressure.eq.01} the condition for the pressure at the boundary,
and \eqref{kin.eq.01} the kinematic condition. 
All four together form the free boundary problem for the motion 
of a drop of incompressible fluid with capillarity. 
The unknowns are the domain $\Om_t$, 
the velocity vector field $u$, 
and the pressure $p$.

An important property of system \eqref{dyn.eq.01}-\eqref{kin.eq.01} 
is the conservation of the total energy 
\begin{equation}\label{eq:energy}
E(t)=\frac12 \int_{\Omega_t }|u|^2\, dx + \sigma_0 \mathrm{Area}(\pa \Omega_t), 
\end{equation}
which means that $E(t) = E(0)$ for all $t \in (0,T)$; 
this follows from a straightforward calculation. 
The total fluid mass, i.e., the volume of $\Omega_t$, 
and the velocity of the fluid barycenter are also conserved quantities. 

If the velocity $u$ is zero, then by \eqref{dyn.eq.01} and \eqref{pressure.eq.01} the mean curvature $ H_{\Om_t}$ is constant, 
and therefore by the Alexandrov theorem \cite{Alexandrov} the drop is a ball. 
Moreover, if the velocity is small (in $C^1$-sense) and the drop is uniformly $C^2$-regular,
then by \cite{CM} the drop is nearly spherical, i.e., it is a small perturbation of the ball. 
Hence, if we study solutions with small velocity, we may reduce to nearly spherical geometry. 
In particular, the drop is star-shaped.

In this paper we study the case when the vorticity is zero and 
$\Om_t$ is star-shaped with respect to the origin, i.e., 
the boundary $ \pa \Om_t$ is the graph of a radial function over the sphere. 
We thus always assume that the domain $\Om_t$, or $\Omega$ when the time plays no role, is of the form 
\begin{equation} \label{def:omega}
\pa \Om_t  =\{ (1+h(t,x)) x : x \in \S^2\} \quad \text{or} \quad 
\pa \Omega  =\{ (1+h(x))x : x \in \S^2\}, 
\end{equation}
where $h(t, \cdot) : \S^2 \to (-1,\infty)$, or $h : \S^2 \to (-1,\infty)$ when time plays no role, is the elevation function. 
In particular, $\pa \Omega$ is diffeomorphic to the sphere with diffeomorphism $\gamma : \S^2 \to \pa \Omega$,
\begin{equation} \label{def:gamma}
\gamma(x) = (1+ h(x)) x .
\end{equation}
For a time-dependent boundary, we denote  $\gamma_t : \S^2 \to \pa \Om_t$,
\begin{equation} \label{def:gamma-t}
\gamma_t(x)  = (1+ h(t,x)) x.
\end{equation}
Since the vorticity is zero and $\Om_t$ is diffeomorphic to the ball, 
there exists a velocity potential $\Phi : \Om_t \to \R$ such that 
\begin{equation}  \label{u.grad.Phi}
u = \grad \Phi \quad \text{in} \ \Om_t .
\end{equation}
Like in the case of water wave equations, we may use these assumptions to reduce 
system  \eqref{dyn.eq.01}-\eqref{kin.eq.01} into a system of two equations written in terms 
of the elevation function $h$ in \eqref{def:omega} and the function $\psi(t,\cdot) : \S^2  \to \R$ 
defined as the pullback by $\gamma_t$ of the velocity potential at the boundary,
\begin{equation}  \label{def:psi}
\psi(t,x) :=  \Phi(t, \gamma_t(x)).
\end{equation}

In order to write these two equations we need to introduce further notation. Assuming sufficient regularity to exchange partial derivatives, 
one has $\pa_t u = \grad \pa_t \Phi$ and $u \cdot \grad u = \grad (\frac12 |\grad \Phi|^2)$. 
Hence the dynamics equation \eqref{dyn.eq.01} becomes 
\begin{equation} \label{dyn.eq.02}
\pa_t \Phi + \frac12 |\grad \Phi|^2 + p  \sim 0 \quad \text{in } \Om_t,
\end{equation}
which is an equation for equivalence classes with respect to the 
equivalence relation $f \sim g$ iff $f- g$ is independent of $x$.
The incompressibility condition \eqref{div.eq.01} becomes the Laplace equation 
\begin{equation}\label{Lap.eq.02}
\Delta \Phi = 0 \quad \text{in } \Om_t.
\end{equation}
We define the Dirichlet-Neumann operator $G(h)\psi$, 
where $h$ is the elevation function as in \eqref{def:omega} 
and $\psi :\S^2 \to \R$ is a generic Dirichlet datum, as 
\begin{equation} \label{def:Dirichlet-Neumann}
G(h) \psi(x) := \langle (\nabla \Phi)(\gamma(x)), \nu_{\Omega}(\gamma(x)) \rangle
\end{equation}
at all points $\gamma(x) \in \pa \Om$, i.e., for all $x \in \S^2$,
where $\Phi:\Omega \to \R$ is the solution in $H^1(\Om)$ 
of the boundary value problem
\begin{equation}
    \label{def:Dirichlet-Neumann2}
    \Delta \Phi = 0 \,\, \text{in } \, \Omega \quad \text{and} \quad 
		\Phi(\gamma(x)) = \psi(x)  \,\, \text{for all } \, \g(x) \in \pa \Omega, 
		\ \text{i.e., for all } \, x \in \S^2. 
\end{equation}
We underline that \eqref{def:Dirichlet-Neumann} defines $G(h)\psi$ 
as the Neumann datum $\la \grad \Phi , \nu_\Om \ra|_{\pa \Om}$ 
of the harmonic extension $\Phi$ of the Dirichlet datum $\psi$, 
while, in the water wave literature, the definition of the Dirichlet-Neumann operator 
often includes a normalizing factor that makes it a self-adjoint operator. 
On the other hand, \eqref{def:Dirichlet-Neumann} 
is more natural from a geometric point of view. 
Finally, we use the notation
\begin{equation}
    \label{def:mean-curv-h}
    H(h)(x) := H_{\Omega}(\gamma(x)) \qquad \text{for } \, x \in \S^2,
\end{equation}
where $\gamma :\S^2 \to \pa \Om $ is defined in \eqref{def:gamma}. 
We derive the  formula of the mean curvature $H(h)$ in terms of the elevation function $h$  
in  Lemma \ref{lem:meancurvature} as the precise explicit formula is difficult to find in literature.  

Under the assumptions of zero vorticity and star-shapedness of $\Om_t$, 
the free boundary problem \eqref{dyn.eq.01}-\eqref{kin.eq.01} 
can be formulated as the system of two equations on $\S^2$ 
\begin{align} \label{kin.eq.08}
& \pa_t h = \frac{\sqrt{(1+h)^2 + |\grad_{\S^2} h|^2}}{1 + h} \, G(h)\psi,
\\
\label{dyn.eq.08-1}
& \pa_t \psi 
- \frac12 \Big( G(h)\psi 
+ \frac{\la \grad_{\S^2} \psi , \grad_{\S^2} h \ra}{(1+h) \sqrt{(1+h)^2 + |\grad_{\S^2} h|^2}} \Big)^2
+ \frac{|\grad_{\S^2} \psi|^2}{2(1+h)^2} 
  + \s_0 H(h) \sim 0,
\end{align} 
where $\grad_{\S^2}$ denotes the tangential gradient on $\S^2$ and, 
moreover, system \eqref{kin.eq.08}, \eqref{dyn.eq.08-1} is a Hamiltonian system.
In the flat case $x \in \R^2$ or $x \in \T^2$, this is an old and well-known observation  
\cite{Zakharov.1968}, \cite{Craig.Sulem}.  In  the spherical case $x \in \S^2$ the equivalence of system  \eqref{dyn.eq.01}-\eqref{kin.eq.01}  and system \eqref{kin.eq.08}, \eqref{dyn.eq.08} 
is proved   in \cite{Beyer.Gunther, Shao.Longtime} but we give the argument in order to be self-contained, 
see Proposition \ref{prop:derivation-system}. 
The Hamiltonian  structure of  \eqref{kin.eq.08}, \eqref{dyn.eq.08}  is proved in  \cite{Beyer.Gunther}; 
we also prove it with a slightly more general argument, 
see Proposition \ref{prop:quasi.Ham} and Lemma \ref{lemma:Darboux}. 
The equation \eqref{dyn.eq.08-1} holds  for equivalence classes, 
because $\Phi(\cdot, t)$ and $\psi(\cdot, t)$ are defined only up to a constant.  
In practice, we may always choose a representative $\tilde \psi  \sim \psi$ 
for which \eqref{dyn.eq.08} is an equation with any fixed constant on the right-hand-side. 
The most convenient choice is
\begin{equation}
\label{dyn.eq.08}
\pa_t \psi 
- \frac12 \Big( G(h)\psi 
+ \frac{\la \grad_{\S^2} \psi , \grad_{\S^2} h \ra}{(1+h) \sqrt{(1+h)^2 + |\grad_{\S^2} h|^2}} \Big)^2
+ \frac{|\grad_{\S^2} \psi|^2}{2(1+h)^2} 
  + \s_0 H(h) = 2 \s_0,
\end{equation}
so that $(h,\psi) = (0,0)$ is a solution of \eqref{kin.eq.08} and \eqref{dyn.eq.08}. 

Our aim is to study the motion of a drop when $\Om_t$ is nearly spherical, 
i.e., when the elevation function in \eqref{def:omega} is small. 
To this aim, we first study the linearization of equations \eqref{kin.eq.08}, \eqref{dyn.eq.08}. 
The first main result of this paper is Theorem \ref{thm:shape.der.DN.op}, 
which gives an explicit formula of the shape derivative of the Dirichlet-Neumann operator 
$G(h)\psi$ defined in \eqref{def:Dirichlet-Neumann} and \eqref{def:Dirichlet-Neumann2}, 
i.e., given a function $\eta :\S^2 \to \R$, we compute a formula for 
\begin{equation} \label{eq:deri:diri-neumann}
G'(h)[\eta]\psi = \frac{d}{d\e } \Big|_{\e = 0}G(h+ \e \eta)\psi.
\end{equation}
The calculations for \eqref{eq:deri:diri-neumann} are rather heavy, 
and we give two different methods to derive the desired formula. 
First, we consider the method of the \emph{good unknown of Alinhac} 
from the theory of water waves \cite{Alazard.Metivier}, \cite{Lannes.book}
and adapt it to the nearly spherical geometry. 
Interestingly, the method does not generalize from water waves to the drop trivially, 
because the simplest extension of the diffeomorphism $\gamma$ in \eqref{def:gamma}, 
which is the homogeneous one,
generates a singularity at the origin, which we remove by introducing a smooth cut-off function. 
The introduction of this cut-off makes the adaptation of the proof in \cite{Lannes.book} 
highly nontrivial. For this reason, to prove Lemma \ref{lemma:alpha.immediately}, 
we do not follow the proof 
in \cite{Lannes.book} directly, but, instead, we employ a new elementary argument 
based on harmonic functions and transformations of the domain. 

After this we also give a completely different argument for \eqref{eq:deri:diri-neumann}, 
which relies only on  geometry. This geometric method is more direct and has the advantage that  
it can be adopted to more general setting as the method itself does not rely on the spherical symmetry. 
Then again, the fact that the reference manifold is the sphere and not the plane, as in the flat case, 
makes the calculations technically more challenging as the derivatives do not commute. 
On the other hand, the advantage of using the good unknown of Alinhac is that it does not require 
any knowledge in differential geometry and therefore we choose to give both arguments.     

The second main result of the paper is Theorem \ref{thm:tame.est.DN.op},
where we prove that the Dirichlet-Neumann operator in \eqref{def:Dirichlet-Neumann} 
depends analytically on the elevation function and prove tame estimates for it in Sobolev class. 

The key achievement of the paper is Theorem \ref{thm:bif}, where we prove the existence of traveling waves, 
i.e., non-trivial solutions of system \eqref{kin.eq.08}, \eqref{dyn.eq.08} 
corresponding to a fixed profile that is rotating around the $x_3$-axis with constant angular velocity
(clearly, by the rotation invariance of the equations, 
any other straight line through the origin could also be taken as the symmetry axis 
of the rotating profiles). 
To the best of our knowledge, this is the first existence result 
for three-dimensional capillary traveling waves on $\S^2$. 
An important consequence of Theorem \ref{thm:bif} is that it provides the first example 
of global-in-time solutions for system \eqref{kin.eq.08}, \eqref{dyn.eq.08}, 
for which, for general initial data, only local existence 
\cite{Beyer.Gunther, Coutand.Shkoller, Shao.On.the.Cauchy}, 
continuation criteria and a priori estimates \cite{Julin.La.Manna, Shatah.Zeng} 
have been proved.

We prove Theorem \ref{thm:bif} by applying the classical bifurcation theorem due to 
Crandall-Rabi\-no\-witz \cite{Crandall-Rabinowitz} for simple eigenvalues of the linearized operator. The set-up of the problem differs from the flat case and also the linearized operator is different due to the curvature of the sphere. This
leads us to analyze the solutions of certain Diophantine equation 
specific of the spherical geometry, see \eqref{dioph.eq.travelling.in.the.thm}.  
By a careful arithmetical argument using prime number factorization, 
we are able to find infinitely many choices of angular velocity producing a one-dimensional kernel of the linearized operator.  

Summarizing, the main results of the paper are: 
\begin{itemize}
\item 
Proposition \ref{prop:derivation-system} 
(reduction to system \eqref{kin.eq.08}, \eqref{dyn.eq.08}); 
 
\item 
Proposition \ref{prop:quasi.Ham} and Lemma \ref{lemma:Darboux} 
(Hamiltonian structure); 

\item
Theorem \ref{thm:shape.der.DN.op}  
(formula of the shape derivative of the Dirichlet-Neumann operator);

\item
Theorem \ref{thm:tame.est.DN.op}
(analytic dependence and tame estimate for the Dirichlet-Neumann operator);

\item
Theorem \ref{thm:bif} 
(existence of rotating traveling waves);
 
\item
Proposition \ref{prop:infinity-family} 
(existence of infinitely many simple bifurcation points).
\end{itemize}

\medskip

\emph{Related literature}. 
The Hamiltonian structure of the water wave equations 
has been proved to hold also with constant vorticity in \cite{Wahlen}.
With the method of the good unknown of Alinhac of \cite{Alazard.Metivier},
a formula for the shape derivative of the Dirichlet-Neumann operator 
is proved in \cite{Lannes.book} in the flat case.
See also the recent work \cite{Huang.Karakhanyan} in conical domains.
A corresponding paralinearization formula is in \cite{Alazard.Metivier} 
for the flat case, and in \cite{Shao.On.the.Cauchy} for $\S^2$.
The analyticity of the Dirichlet-Neumann operator as a function of the elevation function
is proved in the flat case in dimension 2 and 3 
by many authors, see, for example, 
the classical works \cite{Calderon, Coifman.Meyer}, 
the works \cite{Craig.Nicholls, Craig.Schanz.Sulem, Lannes.book}, 
and the recent papers \cite{Berti.Maspero.Ventura, Groves.Nilsson.Pasquali.Wahlen}. 
See also the related paper \cite{Alazard.Burq.Zuily1}. 
Tame estimates for the Dirichlet-Neumann operator are proved in \cite{Alazard.Delort, Lannes.book} 
for the flat case; for related estimates on $\S^2$ see \cite{Beyer.Gunther}.
The literature about traveling waves in the flat case is extensive, 
for both pure gravity and gravity-capillary case, 
for both periodic profiles and solitary waves,  
both finite and infinite depth,
with and without viscosity.  
For a comprehensive review of the existing literature we refer 
to the recent survey \cite{Haziot.et.al}; 
here we just mention the pioneering works \cite{Levicivita, Nekrasov}, 
the papers \cite{Craig.Nicholls, Craig.Schanz.Sulem, Groves.Sun, Wahlen.steady, Walsh},  
and the recent works \cite{Berti.Franzoi.Maspero, Leoni.Tice}. 
We also mention \cite{Moon.Wu}, where rotating travelling waves are obtained for a 2-dimensional drop,
and \cite{Garcia.Hassainia.Roulley} concerning another bifurcation problem in fluid dynamics on $\S^2$.
We finally report that, after the first submission of the present paper, the question of the bifurcation 
from multiple eigenvalues for the present problem has been studied in \cite{BLL}.

\bigskip

\textbf{Acknowledgements.} 
We thank Alberto Maspero for some interesting discussions. 
We also thank the referees for their careful reading 
and valuable comments that helped to improve the presentation of our results. 
This work is supported by Italian GNAMPA, 
by Italian PRIN 2022E9CF89 \emph{Geometric Evolution Problems and Shape Optimization},
by Italian PRIN 2020XB3EFL \emph{Hamiltonian and dispersive PDEs} and the Academy of Finland grant 314227.

\section{Notations and parametrization of the geometric objects} \label{sec:preliminary}

Throughout the  paper, 
LHS means ``left-hand side'', 
RHS means ``right-hand side'', 
$\N = \{ 1,2,\ldots \}$, 
and $\N_0 = \{ 0, 1, \ldots \}$.

\subsection{Notations for differential operators on surfaces}

We denote the fluid domain by $\Omega_t \subset \R^3$, 
or simply by $\Omega$ when time does not play any role, 
and assume it is star-shaped with respect to the origin. 
We may thus parametrize the boundary by the elevation (or height) function $h : \S^2 \to \R$ 
as in \eqref{def:omega}, 
and denote the associated diffeomorphism by $\gamma : \S^2 \to \pa \Omega$ as in \eqref{def:gamma}. 
We always assume without mentioning that the elevation function satisfies $h(x) >-1$ for all $x \in \S^2$ in order to \eqref{def:omega} make sense. If we assume that the Lipschitz norm of the elevation function is small, i.e., $\|h\|_{W^{1,\infty}(\S^2)}\leq c$ for some $c <1$, then we call the domain   $\Omega$ \emph{nearly spherical}.

We assume that  $\Omega$  is $C^2$-regular, which for us means that 
the elevation function $h$ is of class $C^2(\S^2)$, 
and we denote its outer unit normal by $\nu_{\Om}$, 
the mean curvature by $H_{\Om}$ and the second fundamental form by $B_{\Om}$. 
We use orientation for which $H_\Om$ is nonnegative for convex sets. 
We denote the tangent plane at $x \in \pa \Omega$ by $T_x(\pa \Omega)$, 
which we may identify with the plane in $\R^3$ 
\[
T_x(\pa \Omega) = \{ y \in \R^3 : \la y , \nu_{\Om}(x) \ra  = 0\}.
\]
We define the projection on $T_x(\pa \Omega)$ as
\begin{equation} \label{def.proj.tangent}
\Pi_{T_x(\pa \Omega)} =  I - \nu_{\Om}(x) \otimes \nu_{\Om}(x),
\end{equation}
where $I$ is the identity matrix and $\otimes$ denotes the tensor product 
(given $a,b \in \R^3$, $a \otimes b$ is the $3 \times 3$ matrix 
with entry $a_i b_j$ in the row $i$ and column $j$, 
i.e., $a \otimes b$ is the matrix product $a b^T$ of the column matrix $a$ with the row matrix $b^T$).
Note that in \eqref{def.proj.tangent} we identify the projection map 
with the matrix representing it. 
We also remark that \eqref{def.proj.tangent} is symmetric. 
We may split a given vector $a \in \R^3$ into normal and tangential components 
with respect to a fixed tangent plane $T_x(\pa \Omega)$ for $x \in \pa \Om$ as 
\begin{equation}
\label{def:normal-tangential}
a_\nu = a_{\nu_{\Om}(x)} = \la a, \nu_{\Om}(x) \ra \, \nu_{\Om}(x) \quad \text{and} \quad 
a_{\pa \Om} =  a - a_\nu  =  \Pi_{T_x(\pa \Omega)}a . 
\end{equation}
Then $|a|^2 = |a_\nu|^2 + |a_{\pa \Om}|^2$. 
For a vector field $F :\R^3 \to \R^3$, 
$F_\nu$ and $F_{\pa \Omega}$ denote the normal and tangential vector fields on $\pa \Om$. 
For a scalar function $f:\R^3 \to \R$, we define 
the normal and tangential gradient fields as
\begin{equation} \label{def:normal-tangential-2}
\nabla_\nu f(x) = \la \nabla f(x), \nu_{\Om}(x) \ra \, \nu_{\Om}(x) 
\quad \text{and } \quad 
\nabla_{\pa \Om} f(x) = \Pi_{T_x(\pa \Omega)} \nabla f(x),
\end{equation}
and we denote 
\begin{equation} \label{def.D.pa.Omega.scalar.f}
D_{\pa \Omega} f(x) = [\grad_{\pa \Omega} f(x)]^T
= [\grad f(x)]^T \Pi_{T_x(\pa \Omega)}
\end{equation} 
the transpose of the vector $\grad_{\pa \Omega} f(x)$; 
thus, $\grad_{\pa \Omega} f(x)$ is a vector of $\R^3$, i.e., a column, 
while $D_{\pa \Omega} f(x)$ is a $1 \times 3$ matrix, i.e., a row. 
Since we assume that the domain $\Omega$ is at least $C^2$-regular, 
we may extend any regular vector field $F : \pa \Omega \to \R^3$ to $\R^3$ 
and define the tangential differential at $x \in \pa \Omega$ as 
\begin{equation} \label{def:tangdiff}
D_{\pa \Omega} F(x) := DF(x) - DF(x)\nu_{\Omega}(x) \otimes \nu_{\Omega}(x) = DF(x)\Pi_{T_x(\pa \Omega)} ,
\end{equation}
where $DF(x)$ denotes the (Euclidian) differential (i.e., the Jacobian matrix) of the extension. 
Thus, $D_{\pa \Omega} F(x)$ is a $3 \times 3$ matrix, 
whose $k$-th row is $D_{\pa \Omega} F_k(x)$, 
the transpose of the tangential gradient of the $k$-th component $F_k$ 
of the vector field $F$. 
It is easy to see that definition \eqref{def:tangdiff} is independent of the chosen extension. 
We remark that this definition, 
unlike the covariant derivative,  
does not generalize to tensors, 
but, since we only deal with first order derivatives of vector fields, 
definition \eqref{def:tangdiff} is enough for us. 
For a scalar function $f : \pa \Omega \to \R$, 
using any extension of it, the tangential gradient is defined in \eqref{def:normal-tangential-2},
and we define the tangential Hessian as 
\begin{equation} \label{def:tangHess}
D_{\pa \Omega}^2 f := D_{\pa \Omega}(\nabla_{\pa \Omega} f). 
\end{equation}
We also define the Laplace-Beltrami on $\pa \Omega$ as the trace of the tangential Hessian 
\begin{equation} \label{def:Lap-Beltrami}
\Delta_{\pa \Omega} f(x) = \mathrm{Tr} \, D_{\pa \Omega}^2 f(x). 
\end{equation}
For a vector field $F : \pa \Omega \to \R^3$ 
we define the tangential divergence as the trace of the tangential differential
\begin{equation} \label{def:tangdiv}
\div_{\pa \Omega} F = \text{Tr}(D_{\pa \Omega} F)
= \div F - \la (DF) \nu_\Omega , \nu_\Omega \ra,
\end{equation}
and note that it holds 
\begin{equation} \label{Lap.is.div.grad}
\Delta_{\pa \Omega} f = \div_{\pa \Omega}  (\nabla_{\pa \Omega} f )
\end{equation}
and 
\begin{equation} \label{def:H.Om}
H_{\Om} = \div_{\pa \Omega} \nu_\Om.
\end{equation}
If $f$ is defined in a neighborhood of $\pa \Omega$, it holds 
\begin{equation} \label{eq:laplace-extension}
\Delta_{\pa \Omega} f = \Delta f - \la D^2 f \,  \nu_{\Omega}, \nu_{\Om} \ra - H_{\Omega} \la \nabla f, \nu_{\Om} \ra.   
\end{equation} 
Finally, the divergence theorem for hypersurfaces states that 
\begin{equation}  \label{div.thm.manifold}
\int_{\pa \Omega} \div_{\pa \Omega} F \, d \mathcal{H}^2 
= \int_{\pa \Omega} H_{\Om} \la F , \nu_{\Om} \ra \, d \mathcal{H}^2,
\end{equation}
where $d \mH^2$ is the 2-dimensional Hausdorff measure 
(see, e.g., \cite{Maggi, Ambrosio.Fusco.Pallara}).

When $\Om$ is the unit ball $B_1$, its boundary is the unit sphere $\S^2$, 
and we write $\grad_{\S^2}$, $\div_{\S^2}$, etc. 
instead of $\grad_{\pa B_1}$, $\div_{\pa B_1}$, etc. 
For the unit sphere, it is natural to extend a given scalar function $f: \S^2 \to \R$ 
to $\R^3 \setminus \{0\}$ in a homogeneous way as 
\begin{equation}  \label{def.mE.0.mE.1.sigma}
\mE_0 f(x) := f(\s(x)),  \quad 
\mE_1 f(x) := |x| f(\s(x)), \quad \text{where} \quad\ 
\s(x) := \frac{x}{|x|},
\end{equation}
for all $ x \in \R^3 \setminus \{ 0 \}$,
and similarly for vector fields $F : \S^2 \to \R^3$. 
Note that the gradient and Hessian of $0$-homogeneous extensions satisfy 
\begin{equation} \label{0.homog.tool}
\la \grad(\mE_0 f)(x) , x \ra = 0, \quad 
D^2 (\mE_0 f)(x) x + \grad (\mE_0 f)(x) = 0 
\end{equation}
for all $x \in \R^3 \setminus \{ 0 \}$ 
(differentiate the identity $\mE_0 f(\lm x) = \mE_0 f(x)$ 
with respect to $\lm$, and then with respect to $x_k$).  
For the unit sphere, 
the link between the tangential differential operators defined above 
and the corresponding classical differential operators for $0$-homogeneous extensions 
become particularly simple: one has 
\begin{alignat}{2}
\grad_{\S^2} f(x) & = \grad (\mE_0 f)(x), \quad &
D_{\S^2} F(x) & = D (\mE_0 F)(x), \quad 
D_{\S^2}^2 f(x) = D^2( \mE_0 f)(x) + \grad (\mE_0 f)(x) \otimes x, 
\notag \\
\Delta_{\S^2} f(x) & = \Delta (\mE_0 f)(x), \quad &
\div_{\S^2} F(x) & = \div (\mE_0 f)(x)
\label{plants.01}
\end{alignat}
for all $x \in \S^2$.
From \eqref{0.homog.tool} and \eqref{plants.01} it follows, in particular, that 
\begin{align} 
\la \nabla_{\S^2} f(x) ,  x  \ra 
& = 0, 
\label{S2.gradient.is.orthogonal.to.x}
\\
\la D^2_{\S^2} f(x) v , x \ra 
& = - \la \grad_{\S^2} f(x), v \ra
\label{plants.02}
\end{align}
for all $x \in \S^2$, all $v \in T_x(\S^2)$.

\subsection{Parametrization of the geometric objects}

Now we write the normal unit vector and the mean curvature in terms of the elevation function.
Let us first consider the diffeomorphism $\gamma:\S^2 \to \pa \Om$ in \eqref{def:gamma}, 
and take its 1-homogeneous extension 
 \begin{equation}  \label{def.tilde.h}
\mE_1 \g(x) =  x \, (1 + \mE_0 h(x)) ,
\end{equation}
defined on $\R^3 \setminus \{ 0 \}$, 
where $\mE_0, \mE_1$ are defined in \eqref{def.mE.0.mE.1.sigma}. 
Its Jacobian matrix is 
\begin{equation}  \label{Jac.tilde.gamma}
D (\mE_1 \g)  = (1 + \mE_0 h) I + x \otimes \nabla \mE_0 h.
\end{equation}
The advantage of the extension in \eqref{def.tilde.h} is that 
its Jacobian matrix $D (\mE_1 \gamma)$ is invertible, 
and its inverse can be immediately calculated by observing that, 
by \eqref{S2.gradient.is.orthogonal.to.x},  
$M = x \otimes \nabla (\mE_0 h)$ is a nilpotent matrix satisfying $M^2 = 0$. 
Thus
\begin{equation} \label{D.tilde.gamma.inv}
[D (\mE_1 \g)]^{-1} = \frac{I}{1+ \mE_0 h}  
- \frac{x \otimes \nabla \mE_0 h}{(1 + \mE_0 h)^2}.
\end{equation}
We also calculate the determinant 
\begin{equation} \label{det.D.tilde.gamma}
\det D(\mE_1 \g) = (1 + \mE_0 h)^3 
\end{equation}
and the transpose of the inverse matrix 
\begin{equation} \label{D.tilde.gamma.inv.T}
[D (\mE_1 \g)]^{-T} 
= \frac{I}{1+ \mE_0 h} 
- \frac{(\grad \mE_0 h) \otimes x}{(1 + \mE_0 h)^2} .
\end{equation}
Note that on $\S^2$ one has 
\begin{equation} \label{eq:Dgamma-on-S2}
D (\mE_1 \g) = (1 + h) I + x \otimes \nabla_{\S^2} h 
\quad \text{on } \, \S^2
\end{equation}
and
\begin{equation} \label{eq:Dgamma-T-on-S2}
[D (\mE_1 \g)]^{-T} 
= \frac{I}{1+ h} 
- \frac{(\nabla_{\S^2} h) \otimes x}{(1 + h)^2} \quad \text{on } \, \S^2.
 \end{equation}

The tangent plane $T_{\g(x)}(\g(\S^2))$ to the surface $\g(\S^2)$ 
at the point $\g(x) \in \g(\S^2)$ is
\begin{equation} \label{tangent.pa.Om}
T_{\g(x)}(\g(\S^2)) 
= \{ D_{\S^2} \g(x) v : v \in T_x(\S^2) \},
\quad  \text{for all } x \in \S^2.
\end{equation}
By \eqref{def:tangdiff}, one has 
$D_{\S^2} \gamma(x) v 
= D \tilde \gamma (x) (I - x \otimes x) v
= D \tilde \gamma (x) v$ for all $v \in T_x(\S^2)$, 
for any extension $\tilde \gamma$ of $\gamma$;
in particular, this holds for $\tilde \gamma = \mE_1 \gamma$. 
For all $v \in T_x(\S^2)$ one has 
\[
0 = \la x, v \ra 
= \la x, [D\tilde\g(x)]^{-1} D\tilde\g(x) v \ra 
= \la [D\tilde\g(x)]^{-T} x, D\tilde\g(x) v \ra, 
\quad \ \tilde \g = \mE_1 \g,
\]
where $[D\tilde\g(x)]^{-T}$ is in \eqref{eq:Dgamma-T-on-S2}. 
Hence, for $x \in \S^2$, the vector 
\begin{equation}
\label{nu.h.2}
N(\g)(x) := [D (\mE_1 \g) (x)]^{-T} x
\end{equation} 
satisfies $\la N(\g)(x) , w \ra = 0 $  for all $w \in T_{\g(x)} (\g(\S^2))$, i.e., 
it is orthogonal to the tangent plane \eqref{tangent.pa.Om}. 
It is also easy to see that it points outside the domain $\Omega$ 
(note that, for $h=0$ one has $N(\gamma)(x) = x$). 
Therefore the outward unit normal to the surface $\g(\S^2)$ 
at $\g(x)$ is, by \eqref{eq:Dgamma-T-on-S2} and \eqref{nu.h.2},
\begin{equation}  \label{nu.h}
\nu_{\Om}(\gamma(x)) 
= \frac{N(\g)(x)}{|N(\g)(x)|} = \frac{(1 + h(x)) x - \grad_{\S^2} h(x)}{\sqrt{(1 + h(x))^2 + |\grad_{\S^2} h(x)|^2}}
\end{equation}
for $x \in \S^2$. For future purpose it is convenient to introduce the notation
\begin{equation}  \label{def:J}
J = J(h) = \sqrt{(1 + h)^2 + |\grad_{\S^2} h|^2}.
\end{equation}

In order to deal with the forthcoming computation, 
in the next lemma we give some general differentiation rules. 

\begin{lemma}
\label{lem:formulas}
Assume that the domain $\Omega$ and  
the elevation function $h:\S^2 \to \R$ are as in \eqref{def:omega},  
and let $\gamma$ be as in \eqref{def:gamma}.
For any scalar function $f : \pa \Omega \to \R$,
let $\tilde f$ be its pullback by $\gamma$, namely $\tilde f(x) = f(\gamma(x))$. 
Then 
\begin{equation}
\begin{split}
(\nabla_{\pa \Omega} f)(\gamma(x)) &= \frac{\nabla_{\S^2} \tilde f(x)}{1+h} + \frac{\la \nabla_{\S^2}  \tilde f, \nabla_{\S^2} h \ra}{(1+h)J}\nu_\Om( \gamma(x) ) \\
&=  \frac{\nabla_{\S^2}  \tilde f(x)}{1+h}  - \frac{\la \nabla_{\S^2}  \tilde f, \nabla_{\S^2} h \ra}{(1+h)J^2} \nabla_{\S^2} h + \frac{\la \nabla_{\S^2}  \tilde f, \nabla_{\S^2} h \ra}{J^2} x,
\end{split}
\label{eq.1.in.lem:formulas}
\end{equation}
where $J$ is defined in \eqref{def:J}. 
For a vector field $F : \pa \Omega \to \R^3$ we denote similarly $\tilde F(x) = F(\gamma(x))$.
Then 
\begin{align}
(D_{\pa \Omega} F)(\gamma(x)) 
& = \frac{D_{\S^2} \tilde  F(x)}{1+h} 
- \frac{ (D_{\S^2} \tilde F) \, \nabla_{\S^2} h }{(1+h)J^2} \otimes \nabla_{\S^2} h 
+ \frac{ (D_{\S^2} \tilde  F) \, \nabla_{\S^2} h }{J^2} \otimes x,
\label{eq.2.in.lem:formulas}
\\
(\div_{\pa \Omega} F)(\gamma(x)) 
& = \frac{\div_{\S^2} \tilde F(x)}{1+h} 
- \frac{ \la (D_{\S^2} \tilde F) \, \nabla_{\S^2} h, \nabla_{\S^2} h \ra  }{(1+h)J^2} 
+ \frac{\la (D_{\S^2} \tilde F) \, \nabla_{\S^2} h , x \ra }{J^2}.
\label{eq.3.in.lem:formulas}
\end{align}
\end{lemma}

\begin{proof}
We only need to prove \eqref{eq.1.in.lem:formulas}, 
because \eqref{eq.2.in.lem:formulas} 
follows by applying \eqref{eq.1.in.lem:formulas} to the components of $F$, 
and \eqref{eq.3.in.lem:formulas} follows from \eqref{eq.2.in.lem:formulas}
by applying the trace, since $\div_{\pa \Omega} F = \mathrm{Tr}(D_{\pa \Omega} F)$. 

To prove \eqref{eq.1.in.lem:formulas}, 
we extend $f: \pa \Omega \to \R$ 
by taking its $0$-homogeneous extension, which we denote $f_0$, 
namely we define $f_0(x) = f( \gamma (x/|x|) )$ for all $x \in \R^3 \setminus \{ 0 \}$.  
We also define $h_0 : \R^3 \setminus \{0\} \to \R$, $h_0(x) = (\mE_0 h)(x)= h(x /|x|)$. 
Note that $\la \nabla f_0 , x \ra = \la \nabla h_0 , x \ra =0$ on $\R^3 \setminus \{ 0 \}$ 
and $\nabla h_0 = \nabla_{\S^2} h$ on $\S^2$. 
We also extend $\gamma$ as $\mE_1 \g$ defined in \eqref{def.tilde.h}. 

For $\lambda>0$ it holds  $(f_0 \circ \mE_1 \g)(\lambda x) = (f_0 \circ \mE_1 \g)(x)$. Therefore for $x \in \S^2$ it holds
\[
\nabla (f_0 \circ \mE_1 \g)(x) = \nabla_{\S^2}(f_0 \circ \mE_1 \g)(x)  = \nabla_{\S^2} \tilde f(x). 
\] 
On the other hand, one has 
\[
\nabla (f_0 \circ \mE_1 \g)(x) = [D(\mE_1 \g)(x)]^T  \,  (\nabla f_0)(\mE_1 \g(x)). 
\]
Then, by \eqref{eq:Dgamma-T-on-S2} 
and the orthogonality $\la \nabla_{\S^2} \tilde f, x \ra = 0$, 
we have, for all $x \in \S^2$,
\[
(\nabla f_0)(\gamma(x)) = [D(\mE_1 \g)(x)]^{-T}\nabla_{\S^2} \tilde f(x) 
= \frac{ \nabla_{\S^2} \tilde f(x)}{1+h(x)}. 
\]
Then \eqref{eq.1.in.lem:formulas} follows from the identity 
\[
(\nabla_{\pa \Omega} f)(\gamma(x)) 
= (\nabla f_0)(\gamma(x)) - \la (\nabla f_0)(\gamma(x)), \nu_\Om(\gamma(x)) \ra \nu_\Om(\gamma(x))
\]
(see definitions \eqref{def.proj.tangent}, \eqref{def:normal-tangential-2})
and formula \eqref{nu.h}. 
\end{proof}

Note that in equation \eqref{dyn.eq.08}, as well as in \eqref{pressure.eq.01}, 
we have the mean curvature, which has an explicit formula in terms of the elevation function $h$. 
Such a formula follows from a classical calculation, 
which is, however, hard to find in literature in its complete version; 
see the recent paper \cite{Fusco.LaManna}. 
For the sake of completeness, we give the calculations in the next lemma. 

\begin{lemma}
\label{lem:meancurvature}
Assume that the domain $\Omega$ and the elevation function $h:\S^2 \to \R$ are as in \eqref{def:omega}. 
Then the parametrization of the mean curvature is
\begin{equation}\label{eq:meancurvature0}
H(h)(x) = H_{\Om } (\gamma(x)) 
= - \frac{\Delta_{\S^2} h}{(1+h)J} 
+ \frac{2}{J} 
+ \frac{ \la (D_{\S^2}^2 h) \nabla_{\S^2} h, \nabla_{\S^2} h \ra}{(1+h)J^3} 
+ \frac{ |\nabla_{\S^2} h|^2}{J^3},
\end{equation}
where $J$ is defined in \eqref{def:J}.
\end{lemma}

\begin{proof}
Let us recall formula \eqref{nu.h} for the outer normal $\nu_{\Omega}(\gamma(x))$. 
We define the vector field $n : \pa \Omega \to \R^3$ such that, for $x \in \S^2$, 
\begin{equation} \label{def:fieldN}
n(\gamma(x)) =  \tilde n(x) = (1+h(x))x - \nabla_{\S^2} h(x). 
\end{equation}
Then $\nu_{\Omega}(y) = \ph(y) n(y)$ for $y \in \pa \Omega$, where $\ph = 1 / |n|$, and 
\[
H_{\Om} = \div_{\! \pa \Omega} \nu_{\Om} 
= \ph \, \div_{\! \pa \Omega} n + \la \nabla_{\pa \Omega} \ph, n \ra.
\]
Since $n(y)$ is in the direction of $\nu_\Om(y)$, the last term is zero. 
Also, $|n(\gamma(x))|=J$. It remains to calculate $(\div_{\! \pa \Om} n)(\gamma(x))$.  
To this aim, we apply formula \eqref{eq.3.in.lem:formulas} to the vector fields $n, \tilde n$. 
We first calculate $D_{\S^2}\tilde n$. 
By \eqref{def:fieldN}, 
\eqref{def:tangdiff}, 
\eqref{def:tangHess},
\eqref{S2.gradient.is.orthogonal.to.x},
we have 
\begin{equation} \label{eq:meancurv2}
D_{\S^2}\tilde n(x) = - D_{\S^2}^2 h + (1+h) I_{\S^2} + x \otimes  \nabla_{\S^2} h,
\end{equation}
where $I_{\S^2} = I - x \otimes x$ 
(i.e., $I_{\S^2}$ is the matrix $\Pi_{T_x(\S^2)}$ in \eqref{def.proj.tangent})
By \eqref{eq:meancurv2}, 
\eqref{def:tangdiv}, 
\eqref{def:Lap-Beltrami},
\eqref{S2.gradient.is.orthogonal.to.x},
we obtain  
\begin{align*}
\div_{\S^2} \tilde n(x) 
& = - \Delta_{\S^2}^2 h + 2(1+h),
\\
\la (D_{\S^2} \tilde n) \nabla_{\S^2} h, \nabla_{\S^2} h \ra 
& = - \la (D_{\S^2}^2 h) \nabla_{\S^2} h, \nabla_{\S^2} h \ra + (1+h) |\nabla_{\S^2} h|^2.
\end{align*}
Moreover, by \eqref{eq:meancurv2} and \eqref{plants.02}, 
\[
\la (D_{\S^2} \tilde n) \nabla_{\S^2} h, x \ra 
= - \la (D^2_{\S^2} h) \grad_{\S^2} h , x \ra + |\nabla_{\S^2} h|^2
= 2 |\nabla_{\S^2} h|^2.
\]
We use \eqref{eq.3.in.lem:formulas} and the above calculations to deduce that 
\[
\begin{split}
(\div_{\! \pa \Omega} n) (\gamma(x)) 
& = \frac{\div_{\S^2} \tilde n(x)}{1+h} 
- \frac{ \la (D_{\S^2} \tilde n) \nabla_{\S^2} h, \nabla_{\S^2} h \ra }{(1+h)J^2}  
+ \frac{\la (D_{\S^2} \tilde n) \nabla_{\S^2} h , x \ra }{J^2} \\
& = - \frac{\Delta_{\S^2}^2 h}{1+h} + 2 
+ \frac{ \la (D_{\S^2}^2 h) \nabla_{\S^2} h, \nabla_{\S^2} h \ra}{(1+h)J^2} 
+ \frac{ |\nabla_{\S^2} h|^2}{J^2}.
\end{split}
\]
Since $|n(\g(x))| = J$, the lemma is proved. 
\end{proof}

To conclude this subsection we parametrize  the Laplace-Beltrami operator on $\pa \Om$. 
This is needed in subsection \ref{subsec:proof.via.geometric.argument}.

 \begin{lemma}
\label{lem:LapBel}
Assume that the domain $\Omega$ and the  elevation function $h:\S^2 \to \R$ are as in \eqref{def:omega}. Let $f: \pa \Om \to \R$ and  denote  $\tilde f(x) = f(\gamma(x))$. The parametrization of the Laplace-Beltrami operator is
\[
(\Delta_{\pa \Om } f) (\gamma(x)) = \frac{\Delta_{\S^2}\tilde f}{(1+h)^2} 
- \frac{\la (D_{\S^2}^2 \tilde f) \nabla_{\S^2} h, \nabla_{\S^2} h\ra}{(1+h)^2J^2} 
- 2 \frac{\la \nabla_{\S^2} \tilde f, \nabla_{\S^2} h\ra}{(1+h) J^2} 
+ H(h) \frac{\la \nabla_{\S^2} \tilde f, \nabla_{\S^2} h\ra}{(1+h) J}, 
\]
where $J$ is defined in \eqref{def:J} 
and the mean curvature $H(h) = H_{\Om } (\gamma(x))$ is calculated in Lemma \ref{lem:meancurvature}. 
\end{lemma}

\begin{proof}
By \eqref{eq.1.in.lem:formulas} and \eqref{nu.h}, it holds 
\[
(\nabla_{\pa \Omega} f) (\gamma(x)) = \underbrace{\frac{\nabla_{\S^2} \tilde f(x)}{1+h} }_{= \tilde F_1(x)}+  \underbrace{\frac{\la \nabla_{\S^2}  \tilde f, \nabla_{\S^2} h \ra}{(1+h)J}\nu_\Om( \gamma(x) ) }_{= \tilde F_2(x)}.
\]
Recalling \eqref{Lap.is.div.grad}, 
we have to calculate 
$\Delta_{\pa \Om} f =  \div_{\pa \Om}(\nabla_{\pa \Omega} f) = \div_{\pa \Om}(F_1 + F_2)$, 
where $F_1, F_2$ are defined by the equality 
$F_i(\gamma(x)) = \tilde F_i(x)$, $i=1,2$, for $x \in \S^2$. 
We immediately notice that, since $F_2$ is of the form $\varphi \nu_{\Om}$ 
for a scalar function $\varphi$, then, by \eqref{def:tangdiv}, \eqref{def:H.Om}, 
we have $\div_{\pa \Om}(\varphi \nu_{\Om}) = \varphi H_{\Om}$, 
and therefore
\begin{equation}
\label{eq:LapBel-1}
(\div_{\pa \Om}F_2)(\gamma(x)) = H(h) \frac{\la \nabla_{\S^2} \tilde f, \nabla_{\S^2} h\ra}{(1+h) J},
\end{equation}
where $H(h)$ is defined in \eqref{def:mean-curv-h}. 
In order to parametrize $\div_{\pa \Om}F_1$, 
first, by \eqref{def:tangdiff}, \eqref{def:tangHess}, we calculate 
\[
D_{\S^2} \tilde F_1 = \frac{D_{\S^2}^2 \tilde f}{1+h} - \frac{\nabla_{\S^2} \tilde f \otimes \nabla_{\S^2} h}{(1+h)^2}.
\]
By \eqref{def:tangdiv}, \eqref{def:Lap-Beltrami}, \eqref{plants.02}, this yields 
\begin{align*}
\div_{\S^2}  \tilde F_1 & = \frac{\Delta_{\S^2} \tilde f}{1+h} 
- \frac{\la \nabla_{\S^2} \tilde f , \nabla_{\S^2} h\ra}{(1+h)^2},
\\
\la (D_{\S^2} \tilde F_1) \nabla_{\S^2} h, \nabla_{\S^2} h \ra 
& = \frac{\la (D_{\S^2}^2 \tilde f) \nabla_{\S^2} h, \nabla_{\S^2} h \ra}{1+h} 
- \frac{|\nabla_{\S^2} h|^2}{(1+h)^2} \la \nabla_{\S^2} \tilde f , \nabla_{\S^2} h\ra,
\\
\la (D_{\S^2} \tilde F_1) \nabla_{\S^2} h, x \ra 
& =  \frac{\la (D_{\S^2}^2 \tilde f) \nabla_{\S^2} h, x \ra}{1+h} 
=  - \frac{\la \nabla_{\S^2} \tilde f , \nabla_{\S^2} h\ra}{1+h}.
\end{align*}
We have then, by \eqref{eq.3.in.lem:formulas} and \eqref{def:J}, that 
\[
\begin{split}
& (\div_{\pa \Om} F_1) (\gamma(x)) \\
&= \frac{\Delta_{\S^2} \tilde f}{(1+h)^2} 
- \frac{\la \nabla_{\S^2} \tilde f , \nabla_{\S^2} h\ra}{(1+h)^3} 
- \frac{\la (D_{\S^2}^2 \tilde f) \nabla_{\S^2} h, \nabla_{\S^2} h \ra}{(1+h)^2J^2} 
+ \frac{|\nabla_{\S^2} h|^2  \la \nabla_{\S^2} \tilde f , \nabla_{\S^2} h\ra }{(1+h)^3J^2} 
- \frac{\la \nabla_{\S^2} \tilde f , \nabla_{\S^2} h\ra}{(1+h)J^2}\\
&=\frac{\Delta_{\S^2} \tilde f}{(1+h)^2} 
- \frac{\la (D_{\S^2}^2 \tilde f) \nabla_{\S^2} h, \nabla_{\S^2} h \ra}{(1+h)^2J^2} 
- \frac{ 2 \la \nabla_{\S^2} \tilde f , \nabla_{\S^2} h \ra }{ (1+h) J^2 }.
\end{split}
\]
This and \eqref{eq:LapBel-1} yield the claim. 
\end{proof}

We notice that we may define the elliptic operator $\mathcal L_h: C^2(\S^2 ) \to C(\S^2)$, 
\begin{equation}
\label{def:elliptic-operator}
\mathcal L_h[\varphi] = \Delta_{\S^2} \varphi 
- \frac{\la (D_{\S^2}^2 \varphi) \nabla_{\S^2} h, \nabla_{\S^2} h \ra}{J^2},
\end{equation}
and the Laplace-Beltrami operator in Lemma \ref{lem:LapBel} can be written as
\begin{equation}
\label{eq:Lap-Bel-2}
(\Delta_{\pa \Om} f) (\gamma(x)) 
= \frac{\mathcal L_h[\tilde f]}{(1+h)^2} 
+ \left( - \frac{\mathcal L_h[h]}{(1+h)^2J^2} + \frac{|\nabla_{\S^2} h|^2}{(1+h)J^4} \right) 
\la \nabla_{\S^2} \tilde f , \nabla_{\S^2} h\ra.
\end{equation}

\section{The water wave equations on the unit sphere} \label{sec:equation.on.sphere}

In this section we prove that  for the star-shaped liquid drop with zero vorticity the system 
of equations  \eqref{dyn.eq.01}-\eqref{kin.eq.01} is equivalent to  \eqref{kin.eq.08}, \eqref{dyn.eq.08}. 
We then show that the equations  \eqref{kin.eq.08}, \eqref{dyn.eq.08} have Hamiltonian structure. 
These results are classical for the water wave equations in the nearly flat case.

\subsection{Reduction to an equivalent problem on the unit sphere}
\label{subsec:reduction.to.S2}

In this subsection we show that the free boundary problem for the capillary liquid drop 
can be formulated as system \eqref{kin.eq.08}, \eqref{dyn.eq.08}. 
In the flat case where $x \in \R^2$ or $x \in \T^2$, this is an old and well-known observation;
we show here the analogue observation for the spherical case $x \in \S^2$.

\begin{proposition} \label{prop:derivation-system}
Under the condition of zero vorticity $\mathrm{curl} \, u = 0$ and star-shapedness of $\Om_t$, 
system \eqref{dyn.eq.01}-\eqref{kin.eq.01} is equivalent to system \eqref{kin.eq.08}, \eqref{dyn.eq.08}, 
where $\psi$ is defined in\eqref{def:psi} and the elevation function $h$ is in \eqref{def:omega}.
\end{proposition}

\begin{proof}
First we show that \eqref{dyn.eq.01}-\eqref{kin.eq.01} imply \eqref{kin.eq.08}, \eqref{dyn.eq.08}.  
We begin by observing that the parametrization of equation \eqref{kin.eq.01} is given by \eqref{kin.eq.08}:  
the boundary $\pa \Om_t$ in \eqref{def:omega} is described by $\gamma_t$ in \eqref{def:gamma-t},
and the normal velocity $V_t$ appearing in the left hand side of \eqref{kin.eq.01} is, by definition, 
the normal component at $\gamma_t(x) \in \pa \Om_t$ of the time derivative $\pa_t \gamma$, namely   
\begin{equation} \label{V.t.formula}
V_t(\gamma_t(x)) = \la \pa_t\gamma_t(x), \nu_{\Om_t}(\gamma_t(x)) \ra 
= \frac{(1+h)}{J} \pa_t h  \quad \text{on } \, \S^2,
\end{equation}
where we have used formula \eqref{nu.h} for the unit normal,  
the orthogonality property \eqref{S2.gradient.is.orthogonal.to.x}
and the definition \eqref{def:J} of $J$. 
On the other hand, the term $\la u , \nu_\Om \ra$ in the right hand side of \eqref{kin.eq.01} 
can be written by using assumption \eqref{u.grad.Phi} 
and the definition \eqref{def:Dirichlet-Neumann} of the Dirichlet-Neumann operator, 
i.e., 
\begin{equation} \label{u.nu.formula}
\la u(t, \gamma_t(x)) , \nu_{\Om_t}(\gamma_t(x)) \ra 
= \la (\grad \Phi)(t, \gamma_t(x)) , \nu_{\Om_t}(\gamma_t(x)) \ra 
= G(h) \psi(x)
\end{equation}
for all $x \in \S^2$. 
Thus \eqref{kin.eq.01} becomes \eqref{kin.eq.08}.

We proceed to derive \eqref{dyn.eq.08}. 
By continuity, equation \eqref{dyn.eq.02} also holds on the boundary $\pa \Om_t$ 
and, under the pressure condition \eqref{pressure.eq.01} and by choosing a suitable representative, we have
\begin{equation} \label{matita.07.bis}
\pa_t \Phi + \frac12 |\grad \Phi|^2 + \s_0 H_{\Om_t} = 2 \s_0
\quad \text{on } \pa \Om_t. 
\end{equation} 
We recall that $\Phi, \psi$ satisfy identity \eqref{def:psi}. 
By differentiating \eqref{def:psi} with respect to time, we have  
\[ 
(\pa_t \Phi)(t, \gamma_t(x)) + \la (\grad \Phi)(t,\gamma_t(x)) , \pa_t \gamma_t(x) \ra = \pa_t \psi(t,x)
\] 
for all $ x \in \S^2$. Therefore   \eqref{matita.07.bis} becomes
\[
\pa_t \psi(t,x) - \la (\grad \Phi)(t,\gamma_t(x)) , \pa_t \gamma_t(x) \ra
+ \frac12 |(\grad \Phi)(t,\gamma_t(x))|^2 + \s_0 H(h)(x) = 2 \s_0 
\]
for all $x \in \S^2$, 
where $H(h)(x)$ is defined in \eqref{def:mean-curv-h}. 
We then split the gradient $\grad \Phi$ into its normal $\nabla_\nu \Phi$ 
and tangential $\nabla_{\pa \Om} \Phi$ part defined in \eqref{def:normal-tangential-2}. 
Using the definition of the Dirichlet-Neumann operator \eqref{def:Dirichlet-Neumann}, we have
\begin{equation} \label{eq:normal-part-Phi}
(\nabla_\nu \Phi)(t,\gamma_t(x)) = G(h)\psi(x) \nu_{\Om}(\gamma_t(x)).
\end{equation} 
Therefore we deduce that
\[
\begin{split}
    \pa_t \psi(t,x) -  G(h)\psi \la \pa_t \gamma_t,  \nu_{\Om_t}(\gamma_t(x)) \ra  &- \la \grad_{\pa \Om_t} \Phi, \pa_t \gamma_t \ra\\
    &+ \frac12 ( G(h)(\psi) )^2 +\frac12 | \grad_{\pa \Om_t} \Phi |^2 + \s_0 H(h)(x) =  2 \s_0. 
\end{split}
\]
Now $\la \pa_t \gamma_t,  \nu_{\Om_t}(\gamma_t) \ra = G(h)\psi$ 
by \eqref{V.t.formula}, \eqref{u.nu.formula} and \eqref{kin.eq.01}.  
Hence 
\begin{equation} \label{matita.09}
\pa_t \psi - \frac12 ( G(h)\psi )^2   - \la \grad_{\pa \Om_t} \Phi, \pa_t \gamma_t \ra
+\frac12 | \grad_{\pa \Om_t} \Phi |^2 + \s_0 H(h) =  2 \s_0 . 
\end{equation}

We write the tangential part $(\grad_{\pa \Om_t} \Phi)(t, \gamma_t(x))$ 
using \eqref{eq.1.in.lem:formulas} and \eqref{def:psi}, and obtain
\begin{equation} \label{eq:tangential-part-Phi}
(\nabla_{\pa \Omega} \Phi)(t, \gamma_t(x)) 
= \frac{\nabla_{\S^2} \psi}{1+h} 
- \frac{\la \nabla_{\S^2} \psi, \nabla_{\S^2} h \ra}{(1+h)J^2} \nabla_{\S^2} h 
+ \frac{\la \nabla_{\S^2} \psi, \nabla_{\S^2} h \ra}{J^2} x,
\end{equation}
with $x \in \S^2$. By \eqref{eq:tangential-part-Phi}, we calculate
\begin{equation} \label{eq:2.2-1}
| (\nabla_{\pa \Omega} \Phi)(t, \gamma_t(x)) |^2 
= \frac{|\nabla_{\S^2}  \psi|^2}{(1+h)^2} - \frac{\la \nabla_{\S^2} \psi, \nabla_{\S^2} h \ra^2}{(1+h)^2J^2}.
\end{equation} 
Moreover, since  $\pa_t \gamma_t = x \pa_t h $, we have 
\begin{equation} \label{eq:2.2-2}
\la (\nabla_{\pa \Omega} \Phi)(t, \gamma_t(x)) , \pa_t \gamma_t(x) \ra 
= \frac{\la \nabla_{\S^2}  \psi, \nabla_{\S^2} h \ra}{J^2}  \, \pa_t h
= \frac{\la \nabla_{\S^2}  \psi, \nabla_{\S^2} h \ra}{(1+h)J} G(h)\psi,
\end{equation} 
where in the last equality we have used \eqref{kin.eq.08}. 
Combining \eqref{matita.09} with \eqref{eq:2.2-1} and \eqref{eq:2.2-2} yields 
\[
\pa_t \psi - \frac12 ( G(h)\psi )^2 
- \frac{\la \nabla_{\S^2} \psi, \nabla_{\S^2} h \ra}{(1+h)J} G(h) \psi  
- \frac12 \frac{\la \nabla_{\S^2} \psi, \nabla_{\S^2} h \ra^2}{(1+h)^2J^2} 
+ \frac12 \frac{|\nabla_{\S^2} \psi|^2}{(1+h)^2} + \s_0 H(h) = 2 \s_0. 
\]
Then \eqref{dyn.eq.08} follows by noticing that 
the three terms with minus sign form a square. 

Now we prove that \eqref{kin.eq.08}, \eqref{dyn.eq.08} imply \eqref{dyn.eq.01}-\eqref{kin.eq.01} 
with $\mathrm{curl}\, u = 0$. 
Suppose that two functions $h$ and $\psi$, defined on  $(0,T) \times \S^2$, satisfy 
the kinematic equation \eqref{kin.eq.08} and the dynamics equation \eqref{dyn.eq.08}. 
Define the set $\Om_t$ as in \eqref{def:omega} and $\Phi(t,\cdot)$ in $\Om_t$ 
as the solution of the Laplace problem \eqref{def:Dirichlet-Neumann2}.
Then $\Phi(t,\cdot)$ satisfies the incompressibility condition $\Delta \Phi=0$ in $\Om_t$. 
By \eqref{kin.eq.08} and \eqref{def:Dirichlet-Neumann}, 
equation \eqref{kin.eq.01} is also satisfied. 
From \eqref{dyn.eq.08}, using \eqref{kin.eq.08}, we obtain  \eqref{matita.07.bis}. 
Now we define $p$ on the closure of $\Om_t$ as 
\begin{equation} \label{def.tilde.p.using.dyn.eq.02}
p := - \pa_t \Phi - \frac12 |\grad \Phi|^2   + 2\s_0
\quad \ \text{in } \overline{\Om}_t = \Om_t \cup \pa \Om_t.
\end{equation}
Then the dynamics equation \eqref{dyn.eq.02} in the open domain $\Om_t$ trivially holds. 
From \eqref{def.tilde.p.using.dyn.eq.02} at the boundary $\pa \Om_t$ 
and \eqref{matita.07.bis} (which is an identity for points of the boundary $\pa \Om_t$)
we deduce that $p = \s_0 H_{\Om_t}$ on $\pa \Om_t$, i.e., \eqref{pressure.eq.01}. 
\end{proof}

\subsection{Hamiltonian structure}

In this subsection we prove that equations \eqref{kin.eq.08} and \eqref{dyn.eq.08}  
form a Hamiltonian system. 
Similarly as for Proposition \ref{prop:derivation-system}, 
also the Hamiltonian structure of the water wave system 
is an old and well-known result in the flat case $x \in \R^2$ or $x \in \T^2$, 
which goes back to \cite{Zakharov.1968, Craig.Sulem}.
In this subsection we prove the analogue result for the spherical case $x \in \S^2$. 

We remark that, concerning the Hamiltonian structure, 
with the spherical geometry there is a difference with respect to the flat case:
while on $\R^2$ or $\T^2$ the elevation $h$ and the value $\psi$ of the velocity potential at the boundary
are Darboux coordinates of the system, on $\S^2$ this is not true.  
However, $(h,\psi)$ fail to be Darboux coordinates 
only because of a ``wrong'' multiplicative factor $(1 + h)^2$
(see Lemma \ref{prop:quasi.Ham} below), 
and it is not difficult to obtain Darboux coordinates with a simple change of coordinate
(Lemma \ref{lemma:Darboux}). 

\medskip

We start with writing the energy \eqref{eq:energy} in terms of $h, \psi$. 
For $\pa \Omega= \gamma(\S^2)$, where $\gamma$ is in \eqref{def:gamma}, 
the area formula gives 
\begin{equation} \label{def.U.h}
\mathrm{Area}(\pa \Om) = \int_{\S^2} (1+h)\sqrt{(1+h)^2+|\nabla_{\S^2} h|^2}\, d \sigma =: U(h),
\end{equation}
where the last identity defines $U(h)$. 
Recall that, given $\psi :\S^2 \to \R$, we denote $\Phi: \Omega \to \R$ the solution of 
problem \eqref{def:Dirichlet-Neumann2} 
The function $\gamma$ is defined on $\S^2$, 
and its 1-homogeneous extension $\mE_1 \gamma$ in \eqref{def.tilde.h} 
is defined in $\R^3 \setminus \{ 0\}$.
If $h$ is smooth, then $\mE_1 \g$ is smooth in $\R^3 \setminus \{ 0\}$, 
and it can be extended to a Lipschitz continuous map in the whole $\R^3$, 
mapping $0$ to itself and $B_1$ onto $\Om$ bijectively. 
However, if $h$ is not constant, then $\mE_1 \gamma$ is only Lipschitz around the origin.
Therefore we introduce a radial smooth cut-off function around the origin.  
Let 
\begin{equation} \label{def.chi.rho.cut.off}
\chi(x) = \rho(|x|) \ \ \text{for } x \in \R^3, \quad 
\rho \in C^\infty(\R), \quad 
\rho = 1 \ \ \text{in } [\tfrac{1}{2}, \infty), \quad  
\rho = 0 \ \ \text{in } (-\infty, \tfrac{1}{4}].
\end{equation}
To deal with quantitative bounds, it is convenient to fix $\rho$ such that, say, 
\begin{equation} \label{est.rho.cut.off}
0 \leq \rho \leq 1, \quad 
0 \leq \rho' \leq 8.
\end{equation} 
We define 
\begin{equation} \label{def.gamma.B}
\gamma_{B} : \R^3 \to \R^3, \quad 
\gamma_{B}(x) := \big( 1 + \chi(x) \mE_0 h(x) \big) x. 
\end{equation}
Thus, $\gamma_{B} = \mE_1 \gamma$ for $|x| \geq 1/2$, 
and $\gamma_B(x) = x$ in the ball $|x| < 1/4$. 
We remark that $\gamma_B$ maps $B_1$ onto $\Om$. 
Therefore the function 
\begin{equation} \label{def:Phi-punctured}
\tilde \Phi(x) = \Phi(\g_B(x)) \quad \text{for }\, x \in B_1
\end{equation}
satisfies 
\begin{equation} \label{grad.rule} 
\grad \tilde \Phi(x) = [D \gamma_B (x)]^T (\grad \Phi)(\gamma_B(x))  
\quad \text{in } B_1,
\end{equation}
and hence, by arguing as in \cite{EVANS} (section 6.3, page 320),  
it is a weak solution of the Dirichlet problem
\begin{equation} \label{punct.modif.Lap.06}
\div ( P \grad \tilde \Phi) = 0 \ \ \text{in } B_1, 
\quad \ 
\tilde \Phi = \psi \ \ \text{on } \S^2,
\end{equation}
where the matrix $P(x)$ is 
\begin{equation} \label{def.P} 
P(x) := \det( D \g_B (x)) [D \g_B (x)]^{-1} [D \gamma_B(x)]^{-T} 
\end{equation}
in $\R^3$. 
Adapting the calculations in 
\eqref{D.tilde.gamma.inv}, \eqref{det.D.tilde.gamma} and \eqref{D.tilde.gamma.inv.T}, 
and using the orthogonality property $\la \grad \chi, \grad h_0 \ra = 0$, 
where $h_0 := \mE_0 h$,  
one has 
\begin{align}
D \gamma_B & = (1 + \chi h_0) I + \chi x \otimes \grad h_0 + h_0 x \otimes \grad \chi,
\notag \\
(D \gamma_B)^{-1} & = \frac{I}{1 + \chi h_0} 
- \frac{\chi x \otimes \grad h_0}{(1+\chi h_0)(1+\chi_1 h_0)} 
- \frac{h_0 x \otimes \grad \chi}{(1+\chi h_0)(1+\chi_1 h_0)},
\notag \\
\chi_1 & := \chi + \la x , \grad \chi \ra.
\label{def.chi.1}
\end{align} 
From the general formula $\det(I + a \otimes b) = 1 + \la a, b \ra$, we obtain 
\[
\det( D \gamma_B ) = (1 + \chi h_0)^2 (1 + \chi_1 h_0).
\]
Thus
\begin{equation} \label{punct.P.0}
P = (1+\chi_1 h_0) I 
- \grad (\chi h_0) \otimes x - x \otimes \grad (\chi h_0)
+ \frac{|\grad (\chi h_0)|^2 x \otimes x }{ 1 + \chi_1 h_0} 
\end{equation}
for $ x \in \R^3$. 
If $h$ is Lipschitz continuous with a norm $\|h\|_{W^{1,\infty}(\S^2)} \leq C$, 
then it is straightforward to see that the matrix $P$ is uniformly elliptic on $B_1$, 
i.e., $c_0 |\xi|^2 \leq \la P(x) \xi, \xi\ra \leq C_0|\xi|^2$ 
for all $\xi \in \R^3$, all $|x| \leq 1$, 
for some constants $0<c_0 <C_0$ independent of $x, \xi$. 
Therefore the differential operator in \eqref{punct.modif.Lap.06} is uniformly elliptic, 
and the weak solution of problem \eqref{punct.modif.Lap.06} is unique, i.e., 
$\tilde \Phi$ is its unique weak solution.

By the change of variable $y = \gamma_B(x)$, 
the divergence theorem and \eqref{punct.modif.Lap.06}, we have 
\begin{equation}
\label{eq:hamiltonian-0}
\int_{\Omega} |\nabla \Phi|^2\,dx
= \int_{B_1} \langle P \nabla\tilde  \Phi, \nabla\tilde  \Phi\rangle\, dx
= \int_{\S^2} \psi \langle P \nabla \tilde \Phi  ,  x\rangle\, d\sigma. 
\end{equation}
By the definition \eqref{def:Dirichlet-Neumann} of the Dirichlet-Neumann, 
the differentiation rule \eqref{grad.rule}, 
the first identity in \eqref{nu.h} for the unit normal $\nu_\Om$, 
formula \eqref{nu.h.2} for the normal vector $N$, 
and definition \eqref{def.P} of the matrix $P$, 
we get 
\begin{align} \label{eq:hamiltonian-1}
G(h) \psi (x) 
& = \la [D (\mE_1 \g)(x)]^{-T} \grad \tilde \Phi(x) , 
\frac{ [D (\mE_1 \g)(x)]^{-T} x }{ | [D (\mE_1 \g)(x)]^{-T} x | } \ra
\notag \\ & 
= \frac{ \la P(x) \grad \tilde \Phi(x) , x \ra }
{ | [D (\mE_1 \g)(x)]^{-T} x | \det [D (\mE_1 \g)(x)]}
\end{align}
for $x \in \S^2$. 
The matrix $[D (\mE_1 \g)(x)]^{-T}$ on $\S^2$ is given by \eqref{eq:Dgamma-T-on-S2}, and therefore 
\begin{equation} \label{D.tilde.gamma.inv.T.x}
[D (\mE_1 \g)]^{-T} x 
= \frac{x}{1 + h} - \frac{\grad_{\S^2} h}{(1 + h)^2} 
\quad \text{and} \quad 
|[D (\mE_1 \g)(x)]^{-T} x|= \frac{J}{(1+h)^2} 
\end{equation}
on $\S^2$, where $J$ is in \eqref{def:J}. 
Thus, by \eqref{eq:hamiltonian-1}, \eqref{D.tilde.gamma.inv.T.x} and \eqref{det.D.tilde.gamma}, 
we deduce that 
\begin{align}
G(h) \psi = \frac{ \la P \grad \tilde \Phi , x \ra }{ (1+h)J }  
= \frac{ J \la \grad \tilde \Phi , x \ra }{ (1 + h)^2 }  
- \frac{ \la \grad_{\S^2} \psi , \grad_{\S^2} h \ra }{ (1 + h) J }
\label{formula.G.P} 
\end{align}
on $\S^2$, where $J$ is in \eqref{def:J}, 
$P$ is in \eqref{def.P}, 
and $\tilde \Phi$ is the solution of \eqref{punct.modif.Lap.06}. 
The second identity in \eqref{formula.G.P} is obtained by using \eqref{punct.P.0} 
and the identity 
$\la \grad \tilde \Phi , \grad_{\S^2} h \ra = \la \grad_{\S^2} \psi , \grad_{\S^2} h \ra$. 

By \eqref{eq:hamiltonian-0} and \eqref{formula.G.P} we have that
\begin{equation} \label{kin.en.G}
\frac12 \int_{\Omega} |\nabla \Phi|^2\,dx
= \frac12 \int_{\S^2}  \psi \, (G(h)  \psi)  \,(1+h) \sqrt{(1+h)^2+|\nabla_{\S^2} h|^2} \, d\sigma 
=: K(h,\psi), 
\end{equation}
where the last identity defines $K$. 
Hence the energy \eqref{eq:energy} written in terms of $h, \psi$ is
\begin{align} 
\mathcal H (h, \psi) = K(h,\psi) + \sigma_0 U(h),
\label{nat.cand.Ham}
\end{align}
with $K$ in \eqref{kin.en.G} and $U$ in \eqref{def.U.h}. 
With the same calculations above, 
given $h, \psi_1, \psi_2$, with corresponding $\Phi_i, \tilde \Phi_i$, $i=1,2$, 
one has  
\[
\int_{\Omega} \la \nabla \Phi_1 , \nabla \Phi_2 \ra \,dx
= \int_{B_1} \langle P \nabla \tilde \Phi_1, \nabla \tilde \Phi_2 \rangle\, dx
= \int_{\S^2} \psi_2 \langle P \nabla \tilde \Phi_1 , x \rangle\, d\sigma
= \int_{\S^2} \psi_2 G(h) \psi_1 (1+h) J \, d\sigma,
\]
and therefore $G(h)$ satisfies 
\begin{equation} \label{G.self.adj}
\int_{\S^2} \psi_1 G(h) \psi_2 \, d\mu_h 
= \int_{\S^2} \psi_2 G(h) \psi_1 \, d\mu_h, 
\quad d\mu_h = (1+h) \sqrt{(1+h)^2+|\nabla_{\S^2} h|^2}\,d\sigma,
\end{equation}
for all $\psi_1, \psi_2 \in H^{\frac12}(\S^2)$.

We also define a modified Hamiltonian  as
\begin{equation}
\label{def:modifiedHam}
\mathcal H_{2\s_0}(h,\psi)  := \mathcal H(h,\psi) - 2\s_0 \text{Vol}(\Omega_t) =  \mathcal H(h,\psi) - \frac{2 \s_0}{3} \int_{\S^2}(1+h)^3  \, d\sigma.
\end{equation}

\begin{proposition} \label{prop:quasi.Ham}
System \eqref{kin.eq.08}, \eqref{dyn.eq.08-1} is system
\begin{equation}  \label{quasi.Ham.syst}
\pa_t h = \frac{ \pa_\psi \mathcal H(h,\psi) }{ (1+h)^2 } \,,
\quad \ 
\pa_t \psi \sim - \frac{ \pa_h \mathcal H(h,\psi) }{ (1+h)^2 } \,,
\end{equation}
where $\pa_h \mH$, $\pa_\psi \mH$ are the gradients of $\mH$ 
with respect to the $L^2(\S^2)$ scalar product. System \eqref{kin.eq.08}, \eqref{dyn.eq.08} is system
\begin{equation}  \label{quasi.Ham.syst-2}
\pa_t h = \frac{ \pa_\psi \mathcal H_{2\s_0}(h,\psi) }{ (1+h)^2 } \,,
\quad \ 
\pa_t \psi=  - \frac{ \pa_h \mathcal H_{2\s_0}(h,\psi) }{ (1+h)^2 }\,.
\end{equation}
\end{proposition}

We provide two different proofs of Proposition \ref{prop:quasi.Ham}. 
The first proof is here and uses Hadamard's formula; 
the second proof is in subsection \ref{subsec:Ham.by.der.G}, 
and uses the formula \eqref{G.der.02} of the shape derivative of the Dirichlet-Neumann operator.

\begin{proof}[First proof of Proposition \ref{prop:quasi.Ham}]
By linearity and \eqref{G.self.adj}, we have 
\begin{equation}  \label{der.K.psi}
\pa_\psi \mH(h, \psi) = \pa_\psi K(h,\psi) = (1+h) J G(h) \psi,
\end{equation}
with $J$ in \eqref{def:J}. 
Hence \eqref{kin.eq.08} is the first equation in \eqref{quasi.Ham.syst}.
To calculate $\pa_h \mH$, we consider $\pa_h U$ and $\pa_h K$ separately. 
Concerning the potential energy $U$ in \eqref{def.U.h}, 
its derivative with respect to $h$ in direction $\eta$ is 
\[
\begin{split}
\pa_{h}U(h)[\eta]
&= \int_{\S^2} \eta \sqrt{(1+h)^2+|\nabla_{\S^2} h|^2}\, d\sigma
+ \int_{\S^2}(1+h) \frac{(1+h) \eta + \langle\nabla_{\S^2} h ,\nabla_{\S^2} \eta \rangle}
{\sqrt{(1+h)^2+|\nabla_{\S^2} h|^2}} \, d \sigma
\\
&= \int_{\S^2} \eta \, \frac{2(1+h)^2 +|\nabla_{\S^2} h|^2}{\sqrt{(1+h)^2+|\nabla_{\S^2} h|^2}}\, d\sigma
- \int_{\S^2} \eta \, \div_{\S^2} \bigg( \frac{ (1+h) \nabla_{\S^2} h }{ \sqrt{(1+h)^2+|\nabla h|^2}\,} \bigg) \, d\sigma
\\
&=
\int_{\S^2} \eta \, \frac{2(1+h)^2 }{\sqrt{(1+h)^2+|\nabla_{\S^2} h|^2}}\, d\sigma
- \int_{\S^2} \eta (1+h) \div_{\S^2} \bigg( \frac{ \nabla_{\S^2} h }{\sqrt{(1+h)^2+|\nabla h|^2}\,} \bigg) \, d\sigma
\\
&=
\int_{\S^2} \eta (1+h) \div_{\S^2} \bigg( \frac{ (1+h) x - \nabla_{\S^2} h }{\sqrt{(1+h)^2+|\nabla_{\S^2} h|^2}} \bigg) \, d\sigma,
\end{split}
\]
where we have used the divergence theorem 
\eqref{div.thm.manifold} on $\S^2$   
with $F = (1+h) J^{-1} h_1 \grad_{\S^2} h$ where $J$ is defined in \eqref{def:J},
\eqref{S2.gradient.is.orthogonal.to.x}, 
and $\div_{\S^2}(x) = 2$. 
Hence
\begin{equation} \label{der.U.h} 
\pa_{h} U(h) 
= (1+h) \div_{\S^2} \bigg( \frac{ (1+h) x - \nabla_{\S^2} h }{\sqrt{(1+h)^2+|\nabla_{\S^2} h|^2}} \bigg)
= (1+h)^2 H(h)
\end{equation}
on $\S^2$, where $H(h)$ is given by \eqref{eq:meancurvature0}. 
To prove the second identity in \eqref{der.U.h}, apply formula 
$\div_{\S^2} (\ph F) = \ph \, \div_{\S^2} F + \la \grad_{\S^2} \ph , F \ra$  
to $\ph = (1+h) J^{-1}$, $F = x$ 
and to $\ph = J^{-1}$, $F = \grad_{\S^2} h$, 
where $J$ is in \eqref{def:J}, 
and use the fact that, for any extension of $h$, one has 
$\grad(J^{-1}) = - J^{-2} \grad J$ 
and $\grad J = J^{-1} \{ (1+h) \grad h + (D^2 h) \grad h \}$.

To compute the derivative of $K(h,\psi)$ with respect to $h$ in direction $\eta$,  
we consider $K(h+\e \eta, \psi)$, given by \eqref{kin.en.G} 
with $h$ replaced by $h+\e \eta$. To this aim, we define 
\begin{equation} \label{def.gamma.epsilon}
\gamma_\e(x) = (1 + h(x) + \e \eta(x) ) x
\end{equation}
for $x \in \S^2$, 
we extend $h, \eta, \gamma_\e$ 
as $\mathcal E_0 h, \mathcal E_0 \eta, \gamma_{\e,B}$, 
with $\mE_0, \mE_1$ in \eqref{def.tilde.h}, 
and $\gamma_{\e,B}$ defined like in \eqref{def.gamma.B} 
but with $h$ replaced by $h + \e \eta$. 
We denote $\Om_\e = \g_{\e,B}(B_1)$,   
and we denote $\Phi_\e$ the solution of problem 
\begin{equation} \label{Lap.problem.epsilon}
\Delta \Phi_\e = 0  
\ \ \text{in } \Omega_\e,
\qquad 
\Phi_\e = \psi \circ \gamma_\e^{-1} 
\ \ \text{on } \pa \Om_\e, 
\quad \text{i.e., } \quad  
\Phi_\e (\gamma_\e(x)) = \psi(x) \quad \forall x \in \S^2.
\end{equation} 
For $\e=0$, this is problem \eqref{def:Dirichlet-Neumann2}, 
and we write $\Phi, \gamma, \Omega$ instead of $\Phi_0, \gamma_0, \Omega_0$. 
Denote $\dot{\Phi} = \pa_\e|_{\e=0} \Phi_\e$. 
Thus, by \eqref{kin.en.G}, 
\begin{equation} \label{def.K.h.epsilon.eta}
K(h +\e \eta , \psi) = \frac12 \int_{\Om_\e} |\grad \Phi_\e|^2 \, dx.
\end{equation}
To differentiate \eqref{def.K.h.epsilon.eta} with respect to $\e$, 
we use the following formula (see, e.g., \cite{Henrot.Pierre}). 

\begin{lemma}[Hadamard formula, \emph{or} Reynolds transport theorem] \label{lem:hadamard}
Let $\Omega\subset \R^3$ be as in \eqref{def:omega}, 
and assume that $\beta_\e : \R^3 \to \R^3$ is a one-parameter family of diffeomorphisms, 
differentiable with respect to the parameter $\e$, 
such that $\beta_0(x) = x$, 
and let $\frac{d}{d \e}\big|_{\e = 0} \beta_\e(x) = X(x)$.
Assume that a family of functions $u(\cdot, \e):\Omega_\e\to \R $ is differentiable with respect to $\e$, 
and denote $\dot{u} = \pa_\e u(\cdot, 0)$. Then
\[
\frac{d}{d\e}\Big|_{\e = 0}\int_{\Omega_\e} u(x, \e)\,dx 
= \int_{\pa \Omega} u \langle X, \nu_{\Omega} \rangle\, d\sigma 
+ \int_{\Omega} \dot{u} \, dx.
\]
\end{lemma}

To apply Lemma \ref{lem:hadamard},  
we define the family of  diffeomorphisms $\beta_\e :\R^3 \to \R^3$ by 
\[
\beta_\e(x) = \Big(1 + \e \frac{\mE_0 \eta(x)}{1+ \mE_0 h (x)} \Big) x
\] 
for $x \in \R^3 \setminus \{ 0 \}$, and $\beta_\e (0) = 0$. 
Hence $\Om_\e =  \beta_\e(\Om)$ 
and $(\mE_1 \gamma) \circ \beta_\e = \beta_\e \circ (\mE_1 \gamma) = \mE_1 \gamma_\e$.
For all $x \in \R^3 \setminus \{ 0 \}$, it holds
\begin{equation*}
X(x) = \frac{d}{d\e} \Big|_{\e = 0} \beta_\e(x) = \frac{\mE_0 \eta(x)}{1+ \mE_0 h (x)}  \, x,
\end{equation*}
and, at $y = \gamma(x) \in \pa \Om$, with $x \in \S^2$, one has 
\begin{equation} \label{X.at.gamma.S2}
X(\gamma(x)) =\frac{\mathcal E_0 \eta(\gamma(x))}{1+ \mE_0 h(\gamma(x))} \gamma(x) 
= \mathcal E_0 \eta(x) x, 
\end{equation}
because from $y = (1 + h(x)) x$ it follows that $\mE_0 h(y) = \mE_0 h(x) = h(x)$. 
By \eqref{def.K.h.epsilon.eta} and Lemma \ref{lem:hadamard}, we have
\begin{align}
\pa_h K(h, \psi)[\eta] 
&= \frac{d}{d\e} \Big|_{\e = 0} \, \frac12\int_{\Om_\e } |\nabla \Phi_\e |^2\,dx
\notag \\
&= \frac12 \int_{\pa \Om} |\nabla \Phi|^2  \la X , \nu_\Om \ra \, d\sigma
+ \int_{\Om} \langle \nabla \Phi , \nabla \dot{\Phi} \rangle\, dx
\notag \\
&= \frac12 \int_{\pa \Om} |\nabla \Phi|^2  \la X , \nu_\Om \ra \, d\sigma
+ \int_{\pa \Om} \dot{\Phi}\, \langle \nabla \Phi, \nu_{\Om}  \rangle  d \sigma,
\label{der.K.h}
\end{align}
where we have integrated by parts and used that $\Phi$ is harmonic in $\Om$.
By \eqref{X.at.gamma.S2}, \eqref{nu.h} and the area formula, 
the change of variable $y = \g(x)$, $d \sigma(y) = (1+h) J d\sigma(x)$ gives 
\[
\int_{\pa \Om} |\nabla \Phi|^2 \la X, \nu_\Om \ra \, d\sigma 
=\int_{\S^2}|\nabla \Phi (\gamma(x))|^2  (1+h(x))^2 \eta(x)\, d\sigma. 
\]
From the identity $\Phi_\e(\gamma_\e(x))=\psi(x)$, $x \in \S^2$, 
it follows that 
$\dot \Phi(\gamma(x)) + \la (\grad \Phi)(\gamma(x)) , \eta(x) x \ra = 0$ 
for $x \in \S^2$, namely, by \eqref{X.at.gamma.S2}, 
\[
\dot \Phi (y) 
= - \langle \nabla \Phi (y), X(y) \rangle 
\]
at $y = \gamma(x) \in \pa \Om$, with $x \in \S^2$.
Hence, by \eqref{def:Dirichlet-Neumann}, \eqref{X.at.gamma.S2} and the area formula, 
the last integral in \eqref{der.K.h} is
\[\begin{split}
\int_{\pa \Om} \dot{\Phi}\, \langle \nabla \Phi, \nu_{\Om}  \rangle \, d \sigma 
& = - \int_{\pa \Om} \langle \nabla \Phi , X \rangle \langle \nabla \Phi, \nu_{\Om}  \rangle \, d \sigma
= - \int_{\S^2}\langle (\nabla \Phi)(\gamma(x)), x \rangle \eta (1 + h) J G(h)\psi \, d\sigma.
\end{split}
\]
Thus, by \eqref{der.K.h}, 
\begin{equation}\label{der.K.h.bis}
\pa_h K(h, \psi)  = \frac{|(\nabla \Phi) \circ \gamma|^2 (1+ h)^2 }{2} 
- \langle (\nabla \Phi) \circ \gamma, x\rangle       (1+ h) J   \,   G(h)\psi. 
\end{equation}
By \eqref{eq:normal-part-Phi}, \eqref{nu.h}, \eqref{eq:tangential-part-Phi}, we have 
\begin{align*} 
\langle (\nabla \Phi) \circ \gamma, x \rangle 
& = \la (\nabla_{\pa \Om} \Phi) \circ \gamma, x \ra + \la (\nabla_{\nu} \Phi) \circ \gamma, x \ra  
= \frac{ \la \grad_{\S^2} \psi , \grad_{\S^2} h \ra }{ J^2 } + \frac{1+h}{J} G(h)\psi
\end{align*}
on $\S^2$, 
and, by \eqref{eq:normal-part-Phi} and \eqref{eq:2.2-1}, 
\begin{align*}
|(\grad \Phi) \circ \gamma |^2 
& = |(\grad_{\pa \Om} \Phi) \circ \gamma |^2 + | (\grad_\nu \Phi) \circ \gamma |^2 
= \big( G(h)\psi \big)^2 + \frac{ |\grad_{\S^2} \psi|^2 }{ (1+h)^2 }
- \frac{ \la \grad_{\S^2} \psi , \grad_{\S^2} h \ra^2 }{ (1+h)^2 J^2 } 
\end{align*}
on $\S^2$.
Using also \eqref{der.U.h}, 
it follows that \eqref{dyn.eq.08-1} is the second equation in \eqref{quasi.Ham.syst}. 
\end{proof}

Now we show that, with a simple change of variable, 
the factor $(1+h)^{-2}$ can be removed from \eqref{quasi.Ham.syst-2}, 
so that we obtain a Hamiltonian system written in Darboux coordinates.

\begin{lemma} \label{lemma:Darboux}
Consider a change of variable of the form
\begin{equation} \label{change.f.g}
h = f(\eta), \quad 
\psi = g(\eta) \varpi,
\end{equation}
where $f,g$ are real-valued functions of one real variable, 
$f$ invertible, $g$ never vanishing, with
\begin{equation} \label{cond.Ham}
(1 + f(\eta))^2 f'(\eta) g(\eta) = 1. 
\end{equation}
Then system \eqref{quasi.Ham.syst-2} is transformed into the Hamiltonian system
\begin{equation} \label{pure.Ham.syst}
\begin{cases} \pa_t \eta = \pa_\varpi \tilde \mH(\eta, \varpi), \\
\pa_t \varpi = - \pa_\eta \tilde \mH (\eta, \varpi), 
\end{cases}
\end{equation}
where 
\begin{equation} \label{def.mH.1}
\tilde \mH(\eta, \varpi) = \mH_{2\s_0} ( f(\eta), g(\eta) \varpi ).
\end{equation}
\end{lemma}

\begin{proof} 
Differentiating \eqref{def.mH.1} gives
\[
\pa_\eta \tilde \mH(\eta, \varpi) = \pa_h \mH_{2\s_0}  (h, \psi) f'(\eta) + \pa_\psi \mH_{2\s_0}(h,\psi) g'(\eta) \varpi, 
\quad 
\pa_\varpi \tilde \mH(\eta, \varpi) = \pa_\psi \mH_{2\s_0}  (h, \psi) g(\eta),
\]
where $(h,\psi) = (f(\eta), g(\eta) \varpi)$, whence 
\[
\pa_h \mH_{2\s_0}  (h, \psi)  
= \frac{ \pa_\eta \tilde \mH(\eta, \varpi) }{ f'(\eta) }
- \frac{ \pa_\varpi \tilde \mH(\eta, \varpi) }{ g(\eta) f'(\eta) } g'(\eta) \varpi,
\quad 
\pa_\psi \mH_{2\s_0} (h, \psi) = \frac{ \pa_\varpi \tilde \mH(\eta, \varpi) }{ g(\eta) }\,.
\]
Also, 
\[
\pa_t h = f'(\eta) \pa_t \eta, 
\quad
\pa_t \psi = g'(\eta) \varpi \pa_t \eta + g(\eta) \pa_t \varpi.
\]
Hence \eqref{quasi.Ham.syst-2} becomes 
\[
\pa_t \eta = \frac{ \pa_\varpi \tilde \mH(\eta, \varpi) }{ a(\eta) }\,, 
\quad 
\pa_t \varpi = - \frac{ \pa_\eta\tilde \mH(\eta, \varpi) }{ a(\eta) }\,, 
\quad \text{where} \ 
a(\eta) = (1 + f(\eta))^2 f'(\eta) g(\eta),
\]
and this is \eqref{pure.Ham.syst} if $f,g$ satisfy \eqref{cond.Ham}.
\end{proof}

Special cases of transformations \eqref{change.f.g} satisfying \eqref{cond.Ham} are 
\begin{itemize}
\item[($i$)] 
$f(\eta) = \eta$, $g(\eta) = (1 + \eta)^{-2}$, 
which is the change of variable 
$\psi = \varpi / (1+h)^2$ with $h$ unchanged;
\item[($ii$)]
$f(\eta) = (1 + 3 \eta)^{\frac13} - 1$, $g(\eta) = 1$, 
which is $(1 + h)^3 = (1 + 3 \eta)$, with $\psi$ unchanged.
\end{itemize}
The transformation $(i)$ offers the convenience of not having to change $h$ 
in the Dirichlet-Neumann operator $G(h)$ and in the mean curvature $H(h)$. 
The transformation $(ii)$ also has some advantage, 
because, in that case, the conservation of the total fluid mass 
becomes a zero average condition for the new elevation function $\eta$. 
This nice feature of $(ii)$, however, only concerns the mass conservation, 
because the conservation of the barycenter velocity becomes 
a condition involving $(1 + 3 \eta)^{\frac43}$, 
which does not seem to be better than $(1+h)^4$.

\section{Shape derivative of the Dirichlet-Neumann operator} \label{sec:shapederivative}

The Dirichlet-Neumann operator $G(h)\psi$ 
defined in \eqref{def:Dirichlet-Neumann} is linear in $\psi$ 
and, as we prove below (Theorem \ref{thm:tame.est.DN.op}),
it depends analytically on $h$ in suitable Sobolev spaces. 
Hence $G(h)\psi$ is differentiable with respect to $h$;
in this section we prove the following formula for its derivative.

\begin{theorem}  \label{thm:shape.der.DN.op}
There exists $\delta > 0$ such that, 
for $h, \eta \in H^3(\S^2)$, $\psi \in H^{\frac52}(\S^2)$, 
$\| h \|_{H^3(\S^2)} < \delta$, 
the Fr\'echet derivative of $G(h)\psi$ with respect to $h$ in direction $\eta$ is 
\begin{equation} \label{G.der.02}
G'(h)[\eta]\psi = b \eta + \la B , \grad_{\S^2} \eta \ra - G(h) (W \eta),
\end{equation}
where 
\begin{equation} \label{def.W.B}
W = \frac{ \la \grad_{\S^2} \psi , \grad_{\S^2} h \ra }{ J^2 } 
+ \frac{ (1 + h) G(h)\psi }{ J }\,,
\quad \ 
B = \frac{\la \grad_{\S^2} \psi,  \grad_{\S^2} h \ra \grad_{\S^2} h }{ (1+h) J^3 }  
- \frac{ \grad_{\S^2} \psi }{ (1+h) J }\,,
\end{equation}
\begin{equation} \label{def.b}
b = \frac{ \la \grad_{\S^2} \psi , \grad_{\S^2} h \ra }{ J^3 } 
- \frac{ 2 G(h) \psi }{ 1+h } 
- \frac{ \div_{\S^2} \{ (1+h) ( \grad_{\S^2} \psi - W \grad_{\S^2} h ) \} }{ (1+h)^2 J }, 
\quad \ 
\end{equation}
and $J$ is in \eqref{def:J}. 
\end{theorem}

We give two independent proofs of Theorem \ref{thm:shape.der.DN.op}. 
The first proof, in subsection \ref{subsec:proof.via.good.unknown}, 
follows the method of the ``good unknown of Alinhac'';
the second proof, in subsection \ref{subsec:proof.via.geometric.argument}, 
is based on a geometric argument.

\subsection{Proof by the method of the good unknown of Alinhac}
\label{subsec:proof.via.good.unknown}

In this section we prove Theorem \ref{thm:shape.der.DN.op} 
by adapting the approach of Alazard, M\'etivier, and Lannes   
to the nearly spherical geometry. 
We follow, as long as possible, the proof in Lannes' book \cite{Lannes.book}.

The Dirichlet-Neumann $G(h)\psi$ is defined in \eqref{def:Dirichlet-Neumann}, 
and it can be written as \eqref{formula.G.P}, where $\tilde \Phi$ is the solution of \eqref{punct.modif.Lap.06}, with $P$ defined  in \eqref{punct.P.0} and $J$ defined  in \eqref{def:J}.
To compute the derivative of $G(h)\psi$ with respect to $h$ in direction $\eta$ 
using formula \eqref{formula.G.P}, 
we have to study the derivative of $\tilde \Phi$ with respect to $h$ in direction $\eta$. 
We note that we will prove later, in Theorem \ref{thm:tame.est.DN.op}, that  $\tilde \Phi$ is analytical, and therefore differentiable, 
with respect to $h$. 

Let $\Phi_\e$ be the weak solution of problem \eqref{def:Dirichlet-Neumann2} with $h$ replaced by $h+\e \eta$,
that is, 
\begin{equation} \label{tavolo.01}
\Phi_\e \in H^1(\Om_\e),
\quad  
\Delta \Phi_\e = 0 \ \ \text{in } \Om_\e,
\quad  
\Phi_\e \circ \gamma_{\e,B} = \psi \ \ \text{on } \S^2. 
\end{equation}
Let 
\begin{equation} \label{tavolo.02}
\tilde \Phi_\e := \Phi_\e \circ \gamma_{\e,B}.
\end{equation}
Hence $\tilde \Phi_\e$ is the weak solution of problem 
\begin{equation} \label{punct.modif.Lap.pb.tilde.Phi.eps}
\tilde \Phi_\e \in H^1(B_1),
\quad  
\div ( P_\e \grad \tilde \Phi_\e) = 0 \ \ \text{in } B_1,
\quad  
\tilde \Phi_\e = \psi \ \ \text{on } \S^2,
\end{equation}
where $P_\e$ is the matrix we obtain by replacing $h$ with $h + \e \eta$ in \eqref{punct.P.0}. 
Differentiating problem \eqref{punct.modif.Lap.pb.tilde.Phi.eps} with respect to $\e$ at $\e=0$, we obtain  
\begin{equation} \label{lin.ell.pb}
f_1 \in H^1(B_1), 
\quad  
\div(P_1 \grad f_0) + \div(P_0 \grad f_1) = 0 \ \ \text{in } B_1, 
\quad 
f_1 = 0 \ \ \text{on } \S^2,
\end{equation}
where, to shorten the notation, we denote 
\begin{equation} \label{def.P0.P1.f0.f1}
P_0 := P, \quad \ 
P_1 := \pa_\e P_\e |_{\e=0} = P'(h)[\eta], \quad \ 
f_0 := \tilde \Phi, \quad \ 
f_1 := \pa_\e \tilde \Phi_\e |_{\e=0} = \tilde \Phi'(h)[\eta].
\end{equation}
The matrix $P_1$ can be directly obtained by differentiating \eqref{punct.P.0} 
with respect to $h$ in direction $\eta$, and it is 
\begin{equation} \label{punct.P.1}
P_1 = (\chi_1 \eta_0) I - \grad (\chi \eta_0) \otimes x 
- x \otimes \grad (\chi \eta_0)
+ \Big( \frac{ 2 \la \grad (\chi h_0) , \grad (\chi \eta_0) \ra }{ 1 + \chi_1 h_0 } 
- \frac{ |\grad (\chi h_0)|^2 \chi_1 \eta_0 }{ (1 + \chi_1 h_0)^2 } \Big) x \otimes x
\end{equation}
in $\R^3$, where $h_0 := \mE_0 h$ and $\eta_0 := \mE_0 \eta$. 
To adapt the method of the ``good unknown of Alinhac'' in \cite{Lannes.book} to the unit ball $B_1$, 
we replace the vertical partial derivative $\pa_z$ of the flat case 
with the radial derivative operator 
\begin{equation} \label{def.D.x}
D_x := \la x, \grad \ra, 
\end{equation}
and we look for a scalar function $\a$ defined in $B_1$ such that 
\begin{equation} \label{eq.for.alpha}
\div (P_1 \grad f_0) = \div(P_0 \grad (\a D_x f_0)) \quad \text{in } B_1.
\end{equation}
Identity \eqref{eq.for.alpha} implies that 
also the term $\div (P_1 \grad f_0)$ appearing in \eqref{lin.ell.pb} 
can be expressed as the operator $\div (P_0 \grad \cdot )$ applied to a scalar function;
as a consequence, the second item of \eqref{lin.ell.pb} becomes an identity of the form 
$\div (P_0 \grad v) = 0$ where $v$ is a scalar function, 
and this is the identity one has 
in the definition of the Dirichlet-Neumann operator. 
We prove that equation \eqref{eq.for.alpha} has an explicit solution, 
given by the following lemma. 

\begin{lemma} \label{lemma:alpha.immediately} 
Let $h, \psi, \eta, h_0, \eta_0, P_0, P_1, f_0, f_1$ be as above. Then the function  
\begin{equation} \label{alpha.candidate.immediately}
\a = - \frac{\chi \eta_0}{1 + \chi_1 h_0} 
\end{equation} 
solves equation \eqref{eq.for.alpha}.
\end{lemma}

\begin{proof}	
Let 
\begin{equation} \label{tavolo.04}
\dot{\Phi} := \pa_\e \Phi_\e |_{\e = 0}.
\end{equation}
The function $\dot \Phi$ is defined in the open set $\Om$, and it satisfies 
\begin{equation} \label{tavolo.03}
\dot{\Phi} \in H^1(\Om), \quad 
\Delta \dot{\Phi} = 0 \ \ \text{in } \Om,
\end{equation}
because the difference $\Phi_\e - \Phi$ is harmonic away from the boundary of $\pa \Omega$. 
Since $\dot{\Phi}$ is harmonic in $\Om$, the composition $\dot{\Phi} \circ \gamma_B$ satisfies 
\begin{equation} \label{tavolo.05}
\dot{\Phi} \circ \gamma_B \in H^1(B_1), \quad 
\div (P \grad (\dot{\Phi} \circ \gamma_B)) = 0 \ \ \text{in } B_1. 
\end{equation}
Differentiating \eqref{tavolo.02} with respect to $\e$ at $\e = 0$, we get 
\begin{equation} \label{tavolo.06}
f_1 = \dot{\Phi} \circ \gamma_B + \la (\grad \Phi) \circ \gamma_B , \dot{\gamma}_B \ra.
\end{equation}
Hence, by \eqref{tavolo.05}, one has 
\begin{align}
\div(P_0 \grad f_1) 
& = \div [ P_0 \grad ( \dot{\Phi} \circ \gamma_B ) ] 
+ \div [ P_0 \grad ( \la (\grad \Phi) \circ \gamma_B , \dot{\gamma}_B \ra ) ]
\notag \\ 
& = \div [ P_0 \grad ( \la (\grad \Phi) \circ \gamma_B , \dot{\gamma}_B \ra ) ].
\label{tavolo.07}
\end{align}
Identities \eqref{lin.ell.pb} and \eqref{tavolo.07} imply that 
\begin{equation} \label{tavolo.08}
\div (P_1 \grad f_0) = - \div (P_0 \grad f_1) 
= - \div [ P_0 \grad ( \la (\grad \Phi) \circ \gamma_B , \dot{\gamma}_B \ra ) ]
\ \ \text{in } B_1.
\end{equation}
By \eqref{grad.rule}, 
\begin{equation} \label{tavolo.10} 
(\grad \Phi)(\gamma_B(x)) = [D \gamma_B (x)]^{-T} \grad \tilde \Phi(x)  
\quad \text{in } B_1.
\end{equation}
Differentiating $\gamma_{\e,B}$ (see \eqref{def.gamma.B}) with respect to $\e$ at $\e = 0$ one has 
\begin{equation} \label{tavolo.09}
\dot{\gamma}_B(x) = \chi(x) \eta_0(x) x \ \ \text{in } \R^3,
\end{equation}
where $\eta_0 = \mE_0 \eta$. 
By \eqref{tavolo.10} and \eqref{tavolo.09}, one has 
\begin{align}
\la (\grad \Phi)(\gamma_B(x)) , \dot{\gamma}_B(x) \ra 
& = \la [D \gamma_B (x)]^{-T} \grad \tilde \Phi(x) , \chi(x) \eta_0(x) x \ra
\notag \\ 
& = \chi(x) \eta_0(x) \la \grad \tilde \Phi(x) , [D \gamma_B (x)]^{-1} x \ra.
\label{tavolo.11}
\end{align}
By \eqref{def.chi.1}, recalling the notation \eqref{def.D.x}, 
we calculate 
\begin{align}
\la \grad \tilde \Phi(x) , [D \gamma_B (x)]^{-1} x \ra
& = \la \grad \tilde \Phi(x) , 
\frac{x}{1 + \chi h_0} \Big( 1 - \frac{h_0 D_x \chi}{1+\chi_1 h_0} \Big) \ra
\notag \\
& = \frac{\la \grad \tilde \Phi(x) , x \ra }{1 + \chi h_0} 
\Big( 1 - \frac{h_0 D_x \chi}{1+\chi_1 h_0} \Big)
\notag \\
& = \frac{D_x \tilde \Phi}{1 + \chi_1 h_0}.
\label{tavolo.12}
\end{align}
By \eqref{tavolo.08}, \eqref{tavolo.11}, \eqref{tavolo.12}, we obtain the thesis. 
\end{proof}

Now we use Lemma \ref{lemma:alpha.immediately} to calculate 
the shape derivative $G'(h)[\eta] \psi$. 
Using \eqref{eq.for.alpha} to replace the term $\div (P_1 \grad f_0)$ 
with $\div(P_0 \grad(\a D_x f_0))$ in \eqref{lin.ell.pb}, one obtains 
\begin{equation} \label{def.w}
\div(P_0 \grad w) = 0 \ \ \text{in } B_1, 
\qquad 
w := f_1 + \a D_x f_0 \ \ \text{in } B_1,
\end{equation}
where $\alpha$ is the function in \eqref{alpha.candidate.immediately}.
By \eqref{lin.ell.pb}, $f_1 \in H^1(B_1)$,
and, by Lemma \ref{lemma:sol.u}, $\a D_x f_0 \in H^1(B_1)$. 
Hence $w$ in \eqref{def.w} is in $H^1(B_1)$. 
By \eqref{lin.ell.pb}, $f_1 = 0$ on $\S^2$, 
and, by \eqref{alpha.candidate.immediately}, $\a = - \eta / (1 + h)$ on $\S^2$, 
while, by \eqref{formula.G.P}, $D_x f_0 = \la x, \grad \tilde \Phi \ra$ on $\S^2$ is 
\begin{equation} \label{Dx.f0.on.S2}
D_x f_0 
= \frac{ (1+h)^2 G(h)\psi }{ J }
+ \frac{ (1+h) \la \grad_{\S^2} \psi , \grad_{\S^2} h \ra }{ J^2 }
= (1+h) W 
\ \text{on } \S^2,
\end{equation}
where $W$ is defined in \eqref{def.W.B}. 
Therefore $w$ in \eqref{def.w} is $w = - \eta W$ on $\S^2$. 
Thus $w$ is the weak solution of the problem 
\begin{equation} \label{lin.ell.pb.w}
w \in H^1(B_1), \quad 
\div(P_0 \grad w) = 0 \ \ \text{in } B_1, \quad 
w = - \eta W \ \ \text{on } \S^2.
\end{equation}
By \eqref{lin.ell.pb.w}, 
formula \eqref{formula.G.P} with $(- \eta W)$ in the role of $\psi$
and $w$ in that of $\tilde \Phi$ gives  
\begin{align}
- G(h)(\eta W) 
& = Z \la P_0 \grad w , x \ra 
\ \  \text{on } \S^2,
\quad Z := Z(h) = \frac{1}{(1+h)J},
\label{formula.G.w}
\end{align}
with $J$ in \eqref{def:J}.

Before we proceed, we remark that we may use the results from Section \ref{sec:analytic.DN.op} to make sure that  the  formula in Theorem \ref{thm:shape.der.DN.op}  is well-defined. 

\begin{lemma} \label{lemma:trace.ok.for.the.good.unknown}
Let $h, \eta \in H^3(\S^2)$, $\psi \in H^{\frac52}(\S^2)$, 
and let $\| h \|_{H^3(\S^2)} < \delta$, 
where $\delta$ is in Theorem \ref{thm:tame.est.DN.op}.
Then $\eta W \in H^{\frac32}(\S^2)$, 
and the trace at $\S^2$ of the gradients 
$(\grad w)|_{\S^2}, 
(\grad f_1)|_{\S^2}
(\grad D_x f_0)|_{\S^2}$ 
$\in H^{\frac12}(\S^2)$
is well-defined. 
\end{lemma}

\begin{proof}
For $\| h \|_{H^3(\S^2)} < \delta_0$, 
$\psi \in H^{\frac52}(\S^2)$ and $\eta \in H^{\frac32}(\S^2)$,
one has $J \in H^{\frac32}(\S^2)$ by Lemma \ref{lemma:other.terms}, 
$G(h)\psi \in H^{\frac32}(\S^2)$ by Theorem \ref{thm:tame.est.DN.op}, 
and $W, \eta W \in H^{\frac32}(\S^2)$ 
by \eqref{def.W.B}, \eqref{embedding.Sd}, and Lemma \ref{lemma:est.tools.S.n-1}. 
By Lemma \ref{lemma:sol.u}, $\tilde \Phi_\e \in H^3(B_1)$, 
and therefore $f_0, f_1 \in H^3(B_1)$. Hence $w \in H^2(B_1)$. 
As a consequence, $\grad w$, $\grad f_1$, $\grad D_x f_0$ all belong to $H^1(B_1)$, 
and their trace at $\S^2$ is well-defined. 
\end{proof}

Lemma \ref{lemma:trace.ok.for.the.good.unknown} implies that 
any identity in $B_1$ involving $\grad D_x f_0, \grad w, \grad f_1$ 
also holds on $\S^2$ by taking the trace of the involved functions.
To calculate the derivative of $G(h)\psi$ with respect to $h$ in direction $\eta$, 
we differentiate the first identity in \eqref{formula.G.P}. 
Recalling the definition of $P_0, P_1, f_0, f_1, Z$ in \eqref{def.P0.P1.f0.f1}, \eqref{formula.G.w}, 
one has 
\begin{align}
G'(h)[\eta]\psi 
& = 
\underbrace{Z'(h)[\eta] \la P_0 \grad f_0 , x \ra}_{E_1}  
+ \underbrace{Z \la P_1 \grad f_0 , x \ra }_{E_2} 
+ \underbrace{Z \la P_0 \grad f_1 , x \ra}_{E_3}
\quad \text{on } \S^2.
\label{formula.der.G.01}
\end{align}

\emph{Calculation of $E_1$.} 
By the definition \eqref{formula.G.w} of $Z$, we have
\[ 
Z'(h)[\eta] = - \, \frac{ \{ 2 (1 + h)^2 + |\grad_{\S^2} h|^2 \} \eta 
+ (1 + h) \la \grad_{\S^2} h , \grad_{\S^2} \eta \ra  }{ (1 + h)^2 J^3 } 
\]
on $\S^2$, and, by \eqref{formula.G.P}, 
$\la P_0 \grad f_0 , x \ra = (1+h) J G(h)\psi$. 
Hence $E_1$ in \eqref{formula.der.G.01} is 
\begin{equation} 
E_1 = - \frac{ \{ 2 (1 + h)^2 + |\grad_{\S^2} h|^2 \} G(h)\psi }{ (1 + h) J^2 } \, \eta
- \frac{ G(h)\psi }{ J^2 } \,\la \grad_{\S^2} h , \grad_{\S^2} \eta \ra . 
\label{gomma.1}
\end{equation} 

\emph{Calculation of $E_2$.} 
By formula \eqref{punct.P.1}, recalling notation \eqref{def.D.x}, 
we calculate 
\begin{equation} \label{P1.grad.f0.x}
\la P_1 \grad f_0 , x \ra 
= \eta D_x f_0  
- \la \grad_{\S^2} \eta, \grad f_0 \ra 
+ \Big( \frac{ 2 \la \grad_{\S^2} h, \grad_{\S^2} \eta \ra }{ 1 + h } 
- \frac{ |\grad_{\S^2} h|^2 \eta }{ (1 + h)^2 } \Big) D_x f_0
\end{equation}
on $\S^2$. 
Now $D_x f_0$ on $\S^2$ is given by \eqref{Dx.f0.on.S2}, and 
$\la \grad_{\S^2} \eta, \grad f_0 \ra 
= \la \grad_{\S^2} \eta, \grad_{\S^2} \psi \ra$  
on $\S^2$ because 
$f_0 = \psi$ on $\S^2$. 
Hence $E_2$ in \eqref{formula.der.G.01} is 
\begin{align}
E_2 & = \frac{ W }{ J } \Big( 1 - \frac{|\grad_{\S^2} h|^2}{(1+h)^2} \Big) \eta  
- \frac{ \la \grad_{\S^2} \psi, \grad_{\S^2} \eta \ra }{ (1+h) J }
+ \frac{ 2 W \la \grad_{\S^2} h , \grad_{\S^2} \eta \ra }{ (1+h)J }.
\label{gomma.2}
\end{align}

\emph{Calculation of $E_3$.}
To calculate the term $E_3$ in \eqref{formula.der.G.01}, 
we use the definition \eqref{def.w} of $w$ to write $f_1$ as the difference 
$f_1 = w - \a D_x f_0$ in $B_1$. Thus, 
\begin{equation} \label{E3.temp.11} 
\la P_0 \grad f_1 , x \ra = \la P_0 \grad w , x \ra - \la P_0 \grad (\a D_x f_0) , x \ra
\end{equation}
in $B_1$, and therefore, 
by the discussion following Lemma \ref{lemma:trace.ok.for.the.good.unknown}, 
also on $\S^2$. 
Hence, by \eqref{formula.G.w} and \eqref{E3.temp.11}, 
the term $E_3$ in \eqref{formula.der.G.01} is 
\begin{equation} \label{E3.temp.01}
E_3 = - G(h) (\eta W) - Z \la P_0 \grad (\a D_x f_0) , x \ra 
\end{equation}
on $\S^2$.
Since $\a$ and $D_x f_0$ are scalar functions, one has 
\begin{equation}
\la P_0 \grad (\a D_x f_0) , x \ra
= (D_x f_0) \la P_0 \grad \a , x \ra 
+ \a \la P_0 \grad D_x f_0 , x \ra 
\label{E3.temp.13}
\end{equation}
in $B_1$, and therefore on $\S^2$. 
Hence, by \eqref{E3.temp.01} and \eqref{E3.temp.13}, 
\begin{equation} \label{E3.temp.02}
E_3 = - G(h) (\eta W) 
- \underbrace{ Z (D_x f_0) \la P_0 \grad \a , x \ra }_{E_4}
- \underbrace{ Z \a \la P_0 \grad D_x f_0 , x \ra }_{E_5}
\end{equation}
on $\S^2$. 
To study $E_4$, by \eqref{punct.P.0}, \eqref{alpha.candidate.immediately}, we calculate
\begin{equation*} 
\la P_0 \grad \a , x \ra 
= \frac{ \la \grad_{\S^2} h , \grad_{\S^2} \eta \ra }{ 1+h } 
- \frac{ |\grad_{\S^2} h|^2 \eta }{ (1+h)^2 }  
\end{equation*}
on $\S^2$, while $D_x f_0$ and $Z$ on $\S^2$ are in \eqref{Dx.f0.on.S2}, \eqref{formula.G.w}.
Therefore 
$E_4$ in \eqref{E3.temp.02} is 
\begin{equation}
E_4 = \frac{ W \la \grad_{\S^2} h , \grad_{\S^2} \eta \ra }{ (1+h)J } \, 
- \frac{ |\grad_{\S^2} h|^2 W \eta }{ (1+h)^2 J }.
\label{E4.temp.04} 
\end{equation}
The term $E_5$ in \eqref{E3.temp.02} contains derivatives of $f_0$ of second order; 
we use the identity $\div (P_0 \grad f_0) = 0$ 
to express them in terms of $h,\psi$.
Denoting $g = P_0 \grad f_0$, and $g_j$ its $j$-th component, 
one has 
\begin{align*}
0 = \div g(y) 
& = \sum_j \la \grad g_j(y) , e_j \ra  
= \sum_j \la \Pi_{T_x(\S^2)} [ \grad g_j(y) ], e_j \ra  
+ \sum_j \la \grad g_j(y) , x \ra \la x , e_j \ra  
\end{align*}
for $y \in B_1$, $x \in \S^2$. 
Taking the trace at the sphere, i.e., $y \in \S^2$, 
and recalling \eqref{def:tangdiv}, we find
\begin{align} \label{0=div}
0 = \div (P_0 \grad f_0) 
= \div_{\S^2} (P_0 \grad f_0) + \la x , D_x (P_0 \grad f_0) \ra
\end{align}
on $\S^2$. 
We recall that the tangential divergence $\div_{\S^2} (P_0 \grad f_0)$ 
depends only on the restriction of $P_0 \grad f_0$ to $\S^2$, 
which now we calculate.
By \eqref{Dx.f0.on.S2}, and because $f_0 = \psi$ on $\S^2$, one has 
\begin{equation} \label{grad.f0.S2}
\grad f_0 = \grad_{\S^2} f_0 + \la \grad f_0, x \ra x 
= \grad_{\S^2} f_0 + (D_x f_0) x 
= \grad_{\S^2} \psi + W (1+h) x
\end{equation}
on $\S^2$. 
By \eqref{grad.f0.S2}, 
using formulas \eqref{punct.P.0}, \eqref{def.W.B}, \eqref{def:J} of $P_0, W, J$, 
we obtain 
\begin{align}
(P_0 \grad f_0)|_{\S^2}
& = (1+h) ( \grad_{\S^2} \psi - W \grad_{\S^2} h ) + (1+h) J (G(h) \psi)\, x
\label{P0.grad.f0.S2}
\end{align}
on $\S^2$
(note that the $x$ component in \eqref{P0.grad.f0.S2} is also given by \eqref{formula.G.P}). 
Now we consider the scalar product $\la x , D_x (P_0 \grad f_0) \ra$ in $B_1$.   
Since $P_0$ is a $0$-homogeneous function of $x$ in the exterior set $|x| > \frac12$, 
one has $D_x P_0 = 0$ in that set, and we calculate 
\begin{align*}
D_x (P_0 \grad f_0) 
& = (D_x P_0) \grad f_0 + P_0 D_x \grad f_0 
= P_0 D_x \grad f_0 
= P_0 \grad D_x f_0 - P_0 \grad f_0
\end{align*}
in the annnulus $\frac12 < |x| < 1$, and therefore also on $\S^2$. 
Hence, taking the scalar product with $x$, 
and using \eqref{formula.G.P} (or \eqref{P0.grad.f0.S2}) 
to write $\la x, P_0 \grad f_0 \ra$, 
we get 
\begin{equation}  \label{E3.temp.03}
\la x, D_x (P_0 \grad f_0) \ra = \la x, P_0 \grad D_x f_0 \ra - (1+h) J G(h)\psi
\end{equation}
on $\S^2$. 
By \eqref{E3.temp.03}, \eqref{0=div}, \eqref{P0.grad.f0.S2}, we get 
\begin{equation}  \label{E3.temp.04}
\la x, P_0 \grad D_x f_0 \ra 
= (1+h) J G(h) \psi 
- \div_{\S^2} \{ ( \ref{P0.grad.f0.S2} ) \}. 
\end{equation}
Moreover, by \eqref{def:tangdiff}, \eqref{def:tangdiv}, we have
$\div_{\S^2} \{ \ph(x) x \} = 2 \ph(x)$ for any scalar function $\ph$ on $\S^2$; 
we apply it to $\ph = (1+h) J G(h)\psi$.  
Therefore, using \eqref{E3.temp.04}, \eqref{P0.grad.f0.S2} 
and formulas \eqref{formula.G.w}, \eqref{alpha.candidate.immediately} of $Z, \alpha$, 
the term $E_5$ in \eqref{E3.temp.02} is  
\begin{equation}  \label{E3.temp.06}
E_5 = \Big( \frac{ \div_{\S^2} \{ (1+h) ( \grad_{\S^2} \psi - W \grad_{\S^2} h ) \} }{ (1+h)^2 J } 
+ \frac{ G(h)\psi }{ 1+h } \Big) \eta.
\end{equation}

By \eqref{formula.der.G.01}, 
\eqref{gomma.1}, 
\eqref{gomma.2}, 
\eqref{E3.temp.02},
\eqref{E4.temp.04}, 
\eqref{E3.temp.06}, 
we obtain \eqref{G.der.02} with $W$ in \eqref{def.W.B} and 
\begin{align*}
b & = - \frac{ \{ 2 (1 + h)^2 + |\grad_{\S^2} h|^2 \} G(h)\psi }{ (1 + h) J^2 } 
+ \frac{ W }{ J } - \frac{ \div_{\S^2} \{ (1+h) ( \grad_{\S^2} \psi - W \grad_{\S^2} h ) \} }{ (1+h)^2 J } 
- \frac{ G(h)\psi }{ 1+h }, 
\\ 
B & = \Big( \frac{ W }{ (1+h) J } - \frac{ G(h)\psi }{ J^2 } \Big) \grad_{\S^2} h 
- \frac{ \grad_{\S^2} \psi }{ (1+h) J }\,.
\end{align*}
Note that, in computing $b$, the terms $\frac{ W }{ J } \frac{|\grad_{\S^2} h|^2}{(1+h)^2}$ 
in $E_2$ in \eqref{gomma.2} and in $E_4$ in \eqref{E4.temp.04} cancel out,  
and also note that, in computing the coefficient of $\grad_{\S^2} h$ in $B$, 
the term $\frac{W}{(1+h)J}$ appears with coefficient 2 in $E_2$ and with coefficient 1 in $E_4$. 
Finally, using the definition of $W,J$ in \eqref{def.W.B}, \eqref{def:J},  
we obtain the formulas for $B,b$ in \eqref{def.W.B}, \eqref{def.b}.
The proof of Theorem \ref{thm:shape.der.DN.op} is complete.

\subsection{Proof via geometric argument}
\label{subsec:proof.via.geometric.argument}

In this section we give another, independent proof of formula \eqref{G.der.02}, 
using an argument relying on geometry. Here we assume that $\Omega$ is star-shaped.  
The proof is divided in different subsections. 

We consider $\gamma_\e, \Om_\e, \Phi_\e$ 
as in \eqref{def.gamma.epsilon}, \eqref{Lap.problem.epsilon}. 
By Theorem \ref{thm:tame.est.DN.op} (or arguing as in \cite[Proof of Proposition 8.1]{Cagnetti.Mora.Morini}), 
the map $\e \mapsto \Phi_\e$ is smooth. 
We recall that $\dot \Phi$ defined in \eqref{tavolo.04} is harmonic in $\Omega$. 
Similarly we denote $\dot \nu = \frac{d}{d \e} \big|_{\e = 0} \nu_{\Om_\e}(\gamma_\e(x))$ for $x \in \S^2$.  
We also define, like in Lemma \ref{lem:hadamard}, 
the vector field $X: \pa \Omega \to  \R^3$ associated with the change of the domain 
such that, for $x \in \S^2$, 
\begin{equation}\label{eq:velocityfield2}
X(\gamma(x)) =  \frac{d}{d \e} \Big|_{\e = 0} \gamma_\e(x) =  \eta(x) \, x,
\end{equation}
which is \eqref{X.at.gamma.S2}.
Using these notations we may write the Dirichlet-Neumann operator as 
\[
G(h+ \e \eta)\psi = \la (\nabla \Phi_\e)(\gamma_\e(x)) , \nu_{\Om_\e} \ra. 
\]
We may then write \eqref{eq:deri:diri-neumann}  by differentiating the above and have 
\begin{equation}\label{eq:diri-neu-lin2}
G'(h)[\eta] \psi =  \overbrace{\la \nabla \dot \Phi(\gamma(x)),  \nu_\Om \ra}^{= A_1} +   \overbrace{\la D^2 \Phi(\gamma(x)) \, X(\gamma(x))  ,   \nu_\Om \ra}^{= A_2}  + \overbrace{\la \nabla \Phi(\gamma(x)) ,  \dot \nu(\gamma(x)) \ra }^{= A_3}.
\end{equation}
The calculations for each term is rather cumbersome. We will treat them separately in different subsections. 

\subsubsection{Calculations of the term $A_1$} 

We begin by calculating the term $A_1= \la \nabla \dot \Phi(\gamma(x)),  \nu_\Om \ra$ in \eqref{eq:diri-neu-lin2} and show that it can be written as 
\begin{equation} \label{eq:A1}
A_1 = G(h)(-\eta W),
\end{equation}
where $W :\S^2 \to \R$ is in \eqref{def.W.B}. 
We begin by recalling that  $\dot \Phi$ is harmonic. 
Therefore, by the definition of the Dirichlet-Neumann operator in \eqref{def:Dirichlet-Neumann}, in order to identify the term  $A_1$ we need to show that $\dot \Phi(\gamma(x)) = -\eta  W(x)$ for $x \in \S^2$, 
where $W$ is given by \eqref{def.W.B}. 
To this aim we recall that $\Phi_\e$ has the boundary values $\Phi_\e(\gamma_\e(x)) = \psi(x)$ for all  $\e$ and $x \in \S^2$. We differentiate this with respect to $\e$ and obtain 
\[
0 = \frac{d}{d \e } \Big|_{\e = 0}  \Phi_\e(\gamma_\e(x))  = \dot \Phi(\gamma(x)) + \la \nabla \Phi(\gamma(x)) ,  x \ra \, \eta (x) .
\]
Let us split the gradient of $\Phi$ into the normal and the tangential components as in \eqref{def:normal-tangential-2} and by the above it holds 
\begin{equation} \label{eq:phidot1}
 \dot \Phi(\gamma(x)) = -\big( \la  \nabla \Phi(\gamma(x)) , \nu_\Om \ra   \,  \la  \nu_\Om(\gamma(x)), x \ra   + \la \nabla_{\pa \Om} \Phi(\gamma(x)) ,  x \ra \big) \, \eta(x). 
\end{equation}

We may simplify this by using the definition \eqref{def:Dirichlet-Neumann}, i.e., 
$\la  \nabla \Phi(\gamma(x)) , \nu_\Om \ra = G(h)\psi$ and the formula of the normal  \eqref{nu.h} which implies  $\la \nu_\Om(\gamma(x))   , x  \ra  = \frac{1+h}{J}$, 
where $J$ is defined in \eqref{def:J}. 
Therefore the first term on the RHS in \eqref{eq:phidot1} is 
\[
 \la  \nabla \Phi(\gamma(x)) , \nu_\Om \ra   \,  \la  \nu_\Om(\gamma(x)), x \ra 
= \frac{1+h}{J}  \, G(h)\psi.
\]
To deal with the last term in \eqref{eq:phidot1}, we use \eqref{eq:tangential-part-Phi} and get 
\[
\la  \nabla_{\pa \Omega} \Phi(\gamma(x)), x \ra  =  \frac{\la \nabla_{\S^2} \psi, \nabla_{\S^2} h \ra}{J^2}.
\]
The two above equalities and  \eqref{eq:phidot1} imply the formulas \eqref{eq:A1} and \eqref{def.W.B}.

\subsubsection{Calculations of the term $A_2$} 
\label{subsec:A2}

This term is the most cumbersome to calculate and we show that it has the form 
\begin{equation} \label{eq:A2}
\begin{split}
A_2 =& - \eta\,  \frac{\mathcal L_{h}(\psi)}{(1+h)J} + \frac{\eta}{J^2} \la \nabla_{\S^2}\big( G(h)\psi \big),\nabla_{\S^2} h \ra  \\
&+ \frac{\eta}{(1+h)J^{3}} \left( \Delta_{\S^2} h -  2\frac{\la (D_{\S^2}^2 h) \nabla_{\S^2} h ,\nabla_{\S^2} h \ra  }{J^2} - \frac{2(1+h)|\nabla_{\S^2} h|^2}{J^2}  -(1+h) \right) \la\nabla_{\S^2} \psi, \nabla_{\S^2} h \ra\\
&+ \frac{\eta}{(1+h)J^{3}}\la (D_{\S^2}^2 h) \nabla_{\S^2} h ,\nabla_{\S^2} \psi \ra - \eta\, \frac{1+h}{J}H(h) \, G(h)\psi.
\end{split}
\end{equation}

 We begin by splitting vector field $X$ in \eqref{eq:velocityfield2} into the normal and tangential components as in \eqref{def:normal-tangential-2}. Recalling that $X(\gamma(x)) = \eta(x) \, x$ and the formula for the normal \eqref{nu.h} we have 
\begin{equation}
\label{eq:velo-fields}
\begin{split}
 X_\nu(\gamma(x)) &=    \frac{\eta(x)}{J} \la ( (1+h)x - \nabla_{\S^2} h  x),  \nu_\Om(\gamma(x)) \ra= \eta(x) \, \left(\frac{1+h}{J}\right)\nu_\Om(\gamma(x)),\\
X_{\pa \Om}(\gamma(x)) &= (X - X_\nu)(\gamma(x)) = \eta(x) \, x - \eta(x) \, \left(\frac{1+h}{J}\right) \left( \frac{(1+h)x - \nabla_{\S^2} h}{J}\right)  \\
&= \eta(x) \frac{|\nabla_{\S^2} h|^2}{J^2} x + \eta(x) \left(\frac{1+h}{J^2}\right) \nabla_{\S^2} h.
\end{split}
\end{equation}
Then we write 
\begin{equation}
\label{eq:A2-split}
A_2 =  \la D^2 \Phi(\gamma(x)) \, X ,  \nu_\Om\ra   = \la X , \nu_\Om \ra \, \la D^2 \Phi(\gamma(x))  \nu_\Om ,  \nu_\Om \ra+  \la D^2 \Phi(\gamma(x)) \, X_{\pa \Om} ,  \nu_\Om\ra.
\end{equation}

Let us first calculate the first term on the RHS of  \eqref{eq:A2-split}. First, we have by \eqref{eq:velo-fields} that 
\[
\la X , \nu_\Om \ra = X_\nu =  \eta(x) \, \left(\frac{1+h}{J}\right). 
\]
We proceed by recalling that the function $\Phi$ is harmonic in $\Omega$ and have by   the formula \eqref{eq:laplace-extension} 
\[
 \la D^2 \Phi(\gamma(x))  \nu_\Om ,  \nu_\Om \ra = - \Delta_{\pa \Om} \Phi - H_{\Omega} \la \nabla \Phi, \nu_{\Om} \ra .
\]
The mean curvature is  calculated in Lemma \ref{lem:meancurvature}, the Laplace-Beltrami is calculated in Lemma \ref{lem:LapBel} (recall that $\Phi(\gamma(s)) = \psi(x)$) 
and by definition of the Dirichlet-Neumann operator \eqref{def:Dirichlet-Neumann} 
it holds $\la \nabla \Phi, \nu_{\Om}) \ra(\gamma(x))  = G(h)\psi$.
Therefore it holds by \eqref{eq:Lap-Bel-2}
\begin{equation}
\label{eq:A2-0}
\begin{split}
\la X , \nu_\Om \ra \, \la D^2 \Phi(\gamma(x))  \nu_\Om ,  \nu_\Om \ra &= - \eta \frac{1+h}{J}\big( \Delta_{\pa \Om} \Phi +  H_{\Omega}(\gamma(x)) G(h) \psi \big)\\
&= -\eta  \frac{\mathcal L_h(\psi)}{(1+h)J}  + \eta \left( \frac{\mathcal L_h( h)}{(1+h)J^{3}} - \frac{|\nabla_{\S^2} h|^2}{J^{5}} \right) \la\nabla_{\S^2} \psi, \nabla_{\S^2} h \ra \\
&\,\,\,\,\,\,\,- \eta \frac{1+h}{J} H(h) \, G(h)\psi,
\end{split}
\end{equation}
where the operator $\mathcal L_h$ is defined in \eqref{def:elliptic-operator}. 

Let us then calculate the last term in \eqref{eq:A2-split}.  We differentiate  $\la \nabla \Phi, \nu_{\Om} \ra $  in the direction of $X_{\pa \Om}$, which of course is on the tangent plane,  and have 
\[
\la \nabla_{\pa \Om} \la \nabla \Phi, \nu_{\Om} \ra , X_{\pa \Om}\ra   =   \la D^2 \Phi \, X_{\pa \Om} ,  \nu_\Om\ra + \la  D_{\pa \Om} \nu_{\Om} \, X_{\pa \Om}, \nabla \Phi \ra. 
\]
Therefore 
\begin{equation}
\label{eq:A2-1}
\la D^2 \Phi \, X_{\pa \Om} ,  \nu_\Om\ra = \la \nabla_{\pa \Om} \la \nabla \Phi, \nu_{\Om} \ra , X_{\pa \Om}\ra  - \la  D_{\pa \Om} \nu_{\Om} \, X_{\pa \Om}, \nabla \Phi \ra.
\end{equation}

Let us treat the first term on the RHS of \eqref{eq:A2-1}. First we write 
\[
 \la \nabla_{\pa \Om} \la \nabla \Phi, \nu_{\Om} \ra , X_{\pa \Om}\ra  = \la \nabla_{\pa \Om} \la \nabla \Phi, \nu_{\Om} \ra, X \ra.
\] 
We have by \eqref{eq.1.in.lem:formulas}, 
by $\la \nabla \Phi, \nu_{\Om} \ra(\gamma(x)) = \big( G(h)\psi \big)(x)$, 
and by $X(\gamma(x)) = \eta(x)\, x$ that 
\begin{equation}
\label{eq:A2-2}
\la \nabla_{\pa \Om} \la \nabla \Phi, \nu_{\Om} \ra, X_{\pa \Om}\ra = \la \nabla_{\pa \Om} \la \nabla \Phi, \nu_{\Om} \ra, X \ra = \frac{\eta(x)}{J^2} \la \nabla_{\S^2}\big( G(h)\psi \big), \nabla_{\S^2} h \ra .
\end{equation}

We begin by treating the second term on the RHS of \eqref{eq:A2-1} by recalling that 
the differential of the normal is the second fundamental form $B_{\pa \Om} =  D_{\pa \Om} \nu_{\Om}$. 
In particular, $D_{\pa \Om} \nu_{\Om}$ is symmetric, because  
$(D_{\pa \Om} \nu_\Om) \nu_\Om = 0$,
and $D \nu_\Om$ is the Hessian of the signed distance, therefore it is a symmetric matrix.
Hence for all vector fields $F, \Psi$ one has 
\[
\la (D_{\pa \Om} \nu_\Om) F , \Psi \ra 
= \la (D_{\pa \Om} \nu_\Om) F , \Psi_{\pa \Om} \ra.
\]
Moreover, by definition, $(D_{\pa \Om} \nu_\Om) F = (D_{\pa \Om} \nu_\Om) F_{\pa \Om}$, 
so that $D_{\pa \Om} \nu_{\Om}$ is symmetric.
Thus 
\[
\la  D_{\pa \Om} \nu_{\Om} \, X_{\pa \Om}, \nabla \Phi \ra = \la  D_{\pa \Om} \nu_{\Om} \nabla_{\pa \Om} \Phi , X_{\pa \Om} \ra .
\]
We write the normal by using the vector field $n$ defined in \eqref{def:fieldN} 
as $\nu_{\Om} = \frac{n}{|n|}$. 
Then it holds $ D_{\pa \Om} \nu_{\Om} = \frac{1}{|n|} D_{\pa \Om} n + n \otimes (\nabla_{\pa \Om} 1/|n|) $ 
and using $\la n, X_{\pa \Om}\ra = 0$ we have 
\[
\la  D_{\pa \Om} \nu_{\Om} \nabla \Phi , X_{\pa \Om} \ra  
= \frac{1}{|n|} \la  D_{\pa \Om} n  \, \nabla_{\pa \Om} \Phi , X_{\pa \Om} \ra.   
\]
We use \eqref{eq.1.in.lem:formulas} 
and $ \Phi(\gamma(x)) = \psi(x)$ to deduce that
\[\begin{split}
(D_{\pa \Om} n)  (\gamma(x)) \, (\nabla_{\pa \Om} \Phi) (\gamma(x)) &
=D_{\pa \Om} n(\gamma(x)) \left( \frac{\nabla_{\S^2}\psi(x)}{1+h} 
+ \frac{\langle \nabla_{\S^2} \psi , \nabla_{\S^2} h\rangle}{(1+h) J} \nu_{\Om}(\gamma(x))  \right)
\\
&= \frac{1}{1+h} D_{\pa \Om} n(\gamma(x)) \nabla_{\S^2}\psi(x)  
\end{split}
\]
because $ D_{\pa \Om} n (y) \nu_\Om (y) = 0$, 
and then \eqref{eq.2.in.lem:formulas} and have 
\[
D_{\pa \Om} n(\gamma(x)) \nabla_{\S^2}\psi  
= \frac{1}{1+h} D_{\S^2} \tilde n \nabla_{\S^2}\psi 
- \frac{\la  \nabla_{\S^2}\psi , \nabla_{\S^2} h\ra }{(1+h) J^2} D_{\S^2} \tilde n \, \nabla_{\S^2} h.  
\]
Recall that  by \eqref{eq:velo-fields}  it holds 
$X_{\pa \Om}(\gamma(x)) =  \eta \frac{|\nabla_{\S^2} h|^2}{J^2} x 
+ \eta \left(\frac{1+h}{J^2}\right) \nabla_{\S^2} h$. 
Then  by $|n|(\gamma(x)) = J$ we may write 
\begin{equation}
\label{eq:A2-3}
\begin{split}
\la  D_{\pa \Om} \nu_{\Om} \nabla \Phi , X_{\pa \Om} \ra (\gamma(x)) 
&= \frac{1}{(1+h) J}\la  D_{\pa \Om} n(\gamma(x)) \nabla_{\S^2} \psi , X_{\pa \Om}(\gamma(x)) \ra  \\
&= \eta(x) \frac{|\nabla_{\S^2} h|^2}{(1+h)^2J^{3}} \Big( \la D_{\S^2} \tilde n \nabla_{\S^2}\psi, x \ra 
- \frac{\la  \nabla_{\S^2}\psi , \nabla_{\S^2} h\ra}{J^2}  \la D_{\S^2} \tilde n \nabla_{\S^2}h, x  \ra  \Big)
\\
&\,\,\,\,\,\,+   \frac{\eta}{(1+h)J^{3}} \Big( \la D_{\S^2} \tilde n \nabla_{\S^2}\psi, \nabla_{\S^2} h \ra 
- \frac{\la  \nabla_{\S^2}\psi , \nabla_{\S^2} h\ra}{J^2}  \la D_{\S^2} \tilde n \nabla_{\S^2}h , \nabla_{\S^2} h \ra  \Big).
\end{split}
\end{equation}
By \eqref{plants.02} it holds $(D_{\S^2} \tilde n)^T x =  2 \nabla_{\S^2} h$. Using this and recalling $J = \sqrt{(1+h)^2 + |\nabla_{\S^2} h|^2}$ we may write the first term on the RHS of \eqref{eq:A2-3} as
\begin{equation}
\label{eq:A2-4}
\begin{split}
\eta \frac{|\nabla_{\S^2} h|^2}{(1+h)^2J^{3}}&\left( \la D_{\S^2} \tilde n \nabla_{\S^2}\psi, x \ra 
- \frac{\la  \nabla_{\S^2}\psi , \nabla_{\S^2} h\ra}{J^2}  \la D_{\S^2} \tilde n \nabla_{\S^2}h, x  \ra  \right) 
= 2\eta \frac{|\nabla_{\S^2} h|^2}{J^{5}} \la  \nabla_{\S^2}\psi , \nabla_{\S^2} h\ra .
\end{split}
\end{equation}
We use \eqref{eq:meancurv2} and have 
\[
\begin{split}
\la D_{\S^2} \tilde n \nabla_{\S^2}\psi, \nabla_{\S^2} h \ra &= - \la D_{\S^2}^2 h \nabla_{\S^2} h ,\nabla_{\S^2}\psi \ra + (1+h)\la \nabla_{\S^2}\psi, \nabla_{\S^2} h \ra, \\
\la D_{\S^2} \tilde n \nabla_{\S^2}h , \nabla_{\S^2} h \ra &= - \la D_{\S^2}^2 h \nabla_{\S^2} h ,\nabla_{\S^2} h \ra + (1+h)|\nabla_{\S^2} h |^2.  
\end{split}
\]
We may then write the last term in  \eqref{eq:A2-3} as 
\begin{equation}
\label{eq:A2-5}
\begin{split}
\frac{\eta(x)}{(1+h)J^{3}} &\left( \la D_{\S^2} \tilde n \nabla_{\S^2}\psi, \nabla_{\S^2} h \ra 
- \frac{\la  \nabla_{\S^2}\psi , \nabla_{\S^2} h\ra}{J^2}  \la D_{\S^2} \tilde n \nabla_{\S^2}h , \nabla_{\S^2} h \ra  \right)\\
&= -  \frac{\eta(x)}{(1+h)J^{3}} \la D_{\S^2}^2 h \nabla_{\S^2} h ,\nabla_{\S^2}\psi \ra + \frac{ \eta(x)}{J^{3}}\la \nabla_{\S^2}\psi, \nabla_{\S^2} h \ra\\
&\,\,\,\,\,\,\,\,\,+ \frac{\eta(x)}{(1+h)J^{5}} \la D_{\S^2}^2 h \nabla_{\S^2} h ,\nabla_{\S^2} h \ra \la \nabla_{\S^2}\psi, \nabla_{\S^2} h \ra - \eta(x)\frac{|\nabla_{\S^2} h |^2}{J^{5}} \la \nabla_{\S^2}\psi, \nabla_{\S^2} h \ra .
\end{split}
\end{equation}
The formula \eqref{eq:A2} then follows by combining  \eqref{eq:A2-split}, \eqref{eq:A2-0},  
\eqref{eq:A2-1}, \eqref{eq:A2-2}, \eqref{eq:A2-3}, \eqref{eq:A2-4} and \eqref{eq:A2-5}.

\subsubsection{Calculations of the term $A_3$} 

We show that the term $A_3 = \la \nabla \Phi(\gamma(x)) ,  \dot \nu(\gamma(x)) \ra $ in \eqref{eq:diri-neu-lin2}, 
where  $\dot \nu = \frac{d}{d \e} \big|_{\e = 0} \nu_{\Om_\e}(\gamma_\e(x))$ can be written as 
\begin{equation}
\label{eq:A3_0}
A_3 = \eta  \frac{\la \nabla_{\S^2} \psi, \nabla_{\S^2} h \ra }{J^{3}}  - \frac{\la \nabla_{\S^2} \psi, \nabla_{\S^2} \eta \ra}{(1+h)J} + \frac{\la \nabla_{\S^2} \psi, \nabla_{\S^2} h \ra\la \nabla_{\S^2} \eta, \nabla_{\S^2} h \ra}{(1+h)J^{3}}.
\end{equation}
To this aim we use the result in \cite{Cagnetti.Mora.Morini}, equation (8.6), 
where a general formula for calculation of $\dot \nu$ is derived, and we have 
\[
\la \nabla \Phi(\gamma(x)) , \dot \nu(\gamma(x))  \ra = -\la D_{\pa \Omega} X(\gamma(x)) \nabla_{\pa \Omega}  \Phi(\gamma(x)) ,  \nu_{\pa \Omega}(\gamma(x))\ra .\\
\] 
Recall that $X(\gamma(x)) =  \eta(x) \, x$. We then have $D_{\S^2} \tilde X(x) = \eta I_{\S^2} + x \otimes \nabla_{\S^2} \eta$, where $I_{\S^2} = I - x \otimes x$, and therefore by applying the second equality in Lemma \ref{lem:formulas} we deduce 
\[
\begin{split}
D_{\pa \Omega} X(\gamma(x)) = &\frac{\eta}{1+h} I_{\S^2} + \frac{1}{1+h}  x \otimes \nabla_{\S^2} \eta  - \frac{\eta(x)}{(1+h)J^2} \nabla_{\S^2}  h \otimes \nabla_{\S^2}  h \\
&-  \frac{\la \nabla_{\S^2}  \eta, \nabla_{\S^2} h\ra }{(1+h)J^2}x \otimes \nabla_{\S^2}  h +\frac{\eta(x)}{J^2} \nabla_{\S^2}  h \otimes x +\frac{\la \nabla_{\S^2}  \eta, \nabla_{\S^2} h\ra }{J^2}x \otimes x.
\end{split}
\]
Therefore it holds by the first equality in  Lemma \ref{lem:formulas}, by $\Phi(\gamma(x)) = \psi(x)$, and by the formula of the normal  \eqref{nu.h} that 
\[
\begin{split}
 -&\la D_{\pa \Omega} X(\gamma(x)) \nabla_{\pa \Omega}  \Phi(\gamma(x)) ,  \nu_{\pa \Omega}(\gamma(x))\ra = -\frac{1}{1+h}\la D_{\pa \Omega} X(\gamma(x)) \nabla_{\S^2}  \psi ,  \nu_{\pa \Omega}(\gamma(x))\ra\\
&= -\frac{1}{J}\la D_{\pa \Omega} X(\gamma(x)) \nabla_{\S^2}  \psi , x\ra + \frac{1}{(1+h)J}\la D_{\pa \Omega} X(\gamma(x)) \nabla_{\S^2}  \psi ,  \nabla_{\S^2}  h \ra\\
&= - \frac{\la \nabla_{\S^2} \psi, \nabla_{\S^2} \eta \ra}{(1+h)J} + \frac{\la \nabla_{\S^2} \psi, \nabla_{\S^2} h \ra\la \nabla_{\S^2} \eta, \nabla_{\S^2} h \ra}{(1+h)J^{3}} \\
& \,\,\,\,\,\,\,\,\,+\frac{\eta}{(1+h)^2J} \la \nabla_{\S^2} \psi, \nabla_{\S^2} h \ra -\frac{\eta|\nabla_{\S^2} h|^2}{(1+h)^2 J^{3}} \la \nabla_{\S^2} \psi, \nabla_{\S^2} h \ra .
\end{split}
\]
The formula \eqref{eq:A3_0} then follows from above and by recalling that $J = \sqrt{(1+h)^2 +|\nabla_{\S^2} h|^2}$ which then simplifies the last terms as
\[
\frac{\eta}{(1+h)^2} \la \nabla_{\S^2} \psi, \nabla_{\S^2} h \ra 
-\frac{\eta|\nabla_{\S^2} h|^2}{(1+h)^2 J^{3}} \la \nabla_{\S^2} \psi, \nabla_{\S^2} h \ra 
= \frac{\eta}{J^{3}} \la \nabla_{\S^2} \psi, \nabla_{\S^2} h \ra.
\]

\subsubsection{Conclusion}

By \eqref{eq:diri-neu-lin2}, 
\eqref{eq:A1}, 
\eqref{eq:A2}, 
\eqref{eq:A3_0}, 
we obtain \eqref{G.der.02} with $W,B$ in \eqref{def.W.B} 
and 
\begin{equation}\label{eq:shape-coeff}
\begin{split}
b & = - \frac{\mathcal L_{h}(\psi)}{(1+h)J} 
+ \frac{1}{J^2} \la \nabla_{\S^2}\big( G(h)(\psi)\big),\nabla_{\S^2} h \ra  
\\
& \quad \ 
+ \frac{ \la\nabla_{\S^2} \psi, \nabla_{\S^2} h \ra }{(1+h)J^{3}} \Big( \Delta_{\S^2} h -  2\frac{\la D_{\S^2}^2 h \nabla_{\S^2} h ,\nabla_{\S^2} h \ra  }{J^2} - \frac{2(1+h)|\nabla_{\S^2} h|^2}{J^2}  \Big) 
\\
& \quad \ 
+ \frac{ \la (D_{\S^2}^2 h) \nabla_{\S^2} h ,\nabla_{\S^2} \psi \ra  }{(1+h)J^{3}} 
- \frac{1+h}{J}H(h) \, G(h)\psi ,
\end{split}
\end{equation}
where the operator $\mathcal L_{h}$ is defined in \eqref{def:elliptic-operator}, 
and the mean curvature $H(h)$ is calculated in Lemma \ref{lem:meancurvature}.
Finally, a long, but straightforward, calculation 
shows that the function $b$ in \eqref{eq:shape-coeff} 
coincides with $b$ in \eqref{def.b}.
This completes the proof of Theorem \ref{thm:shape.der.DN.op}.

\subsection{Another proof of the Hamiltonian structure}
\label{subsec:Ham.by.der.G}

In this short subsection we provide an alternative proof of Proposition \ref{prop:quasi.Ham}, 
which makes use of formula \eqref{G.der.02} 
instead of relying on Hadamard's formula (Lemma \ref{lem:hadamard}) 
to compute $\pa_h K(h,\psi)$.

\begin{proof}[Second proof of Proposition \ref{prop:quasi.Ham}] 
The only difference with respect to the first proof concerns the calculation of $\pa_h K$.
Recalling the definition \eqref{kin.en.G} of $K(h,\psi)$, we have 
\begin{align}
\pa_h K(h,\psi)[\eta] 
& = \frac12 \int_{\S^2} \eta J(h) \psi G(h)\psi \, d\sigma 
+ \frac12 \int_{\S^2} (1+h) J'(h)[\eta] \psi G(h)\psi \, d\sigma 
\notag \\ 
& \quad \ 
+ \frac12 \int_{\S^2} (1+h) J(h) \psi G'(h)[\eta] \psi \, d\sigma,
\label{hopla.01}
\end{align}
where $J(h) = J$ is in \eqref{def:J}. 
The second term on the RHS of \eqref{hopla.01} is equal to 
\begin{align*}
& \frac12 \int_{\S^2} (1+h) \, 
\frac{ (1+h) \eta + \la \grad_{\S^2} h, \grad_{\S^2} \eta \ra }{ J } \, \psi G(h)\psi \, d\sigma 
\\
& \quad = \frac12 \int_{\S^2} \eta \, \frac{ (1+h)^2 }{ J } \psi G(h)\psi \, d\sigma 
- \frac12 \int_{\S^2} \eta \, \div_{\S^2} 
\Big( \frac{ (1+h) \grad_{\S^2} h }{ J } \, \psi G(h)\psi \Big) \, d\sigma, 
\end{align*}
where we have used the divergence theorem \eqref{div.thm.manifold} on $\S^2$.
For the last term of \eqref{hopla.01}, 
we use \eqref{G.der.02} to replace $G'(h)[\eta] \psi$
with $b \eta + \la B , \grad_{\S^2} \eta \ra - G(h)(W\eta)$, 
and then we apply the divergence theorem \eqref{div.thm.manifold} on $\S^2$
to the integral containing $\la B, \grad_{\S^2} \eta \ra$, 
and \eqref{G.self.adj} to that containing $G(h)(W \eta)$.
Hence the last term of \eqref{hopla.01} is 
\[
\frac12 \int_{\S^2} \eta \, \Big\{ (1+h) J \big( b \psi - W G(h)\psi \big) 
- \div_{\S^2} \Big( (1+h) J \psi B \Big) \Big\} \, d\sigma.
\]
Thus we have proved that 
\begin{align}
\pa_h K(h,\psi) & = \frac12 \Big( J + \frac{(1+h)^2}{J} \Big) \psi G(h)\psi 
+ \frac12 (1+h) J \big( b \psi - W G(h)\psi \big) 
\notag \\ 
& \quad \ 
- \frac12 \div_{\S^2} \Big( \frac{1+h}{J} (\grad_{\S^2} h) \psi G(h)\psi 
+ (1+h) J \psi B \Big).
\label{hopla.02}
\end{align}
By definition \eqref{def.W.B} of $W$ and $B$, 
the term in the tangential divergence in \eqref{hopla.02} is 
\[
\frac{1+h}{J} (\grad_{\S^2} h) \psi G(h)\psi + (1+h) J \psi B
= - \psi (\grad_{\S^2} \psi - W \grad_{\S^2} h).
\]
Inserting the last identity in \eqref{hopla.02}, and using definitions 
\eqref{def.W.B}, \eqref{def.b}, after some cancellations it remains
\begin{equation} \label{formula.pa.h.K}
\pa_h K(h,\psi) = \frac{ | \grad_{\S^2} \psi |^2 }{2} 
- \frac12 \Big( (1+h) G(h)\psi + \frac{ \la \grad_{\S^2} \psi , \grad_{\S^2} h \ra }{J} \Big)^2,
\end{equation}
and the proof is complete. 
\end{proof}

\section{Analyticity of the Dirichlet-Neumann operator}
\label{sec:analytic.DN.op}

In this section we show that the Dirichlet-Neumann operator depends analytically 
on the elevation function $h$ in Sobolev class.  
We will use Sobolev spaces $H^k(B_1)$ on the open unit ball $B_1$ with integer exponent $k \geq 0$, 
the space $H^1_0(B_1)$, 
the space $W^{1,\infty}(\S^2)$, 
and spaces $H^s(\S^2)$ on the unit sphere $\S^2$ with half-integer exponents $s = \frac{m}{2}$, $m \in \N$. 
We define the $H^k(B_1)$ norm of a function $u$ as the sum of the $L^2(B_1)$ norm 
of all weak derivatives $\partial^\alpha u$ of order $|\a| \leq k$. 
We define the $W^{1,\infty}(\S^2)$ norm of a function $u$ as the sum of the $L^\infty(\S^2)$ norm 
of $u$ and that of $\grad_{\S^2} u$. 
The $H^s(\S^2)$ norm of a function can be defined by localization, rectification and extension, 
using a partition of unity of the sphere $\S^2$,  
or by means of its orthogonal decomposition into spherical harmonics of $\S^2$, 
see \eqref{def.spec.norm}. These two definitions give equivalent norms 
(\cite{Lions.Magenes.volume.1}, Remark 7.6 in Section 7.3). 
For the theory of Sobolev spaces on open bounded domains and on smooth manifolds 
we refer, e.g., to Lions-Magenes \cite{Lions.Magenes.volume.1}, 
Taylor \cite{Taylor.volume.1}, Triebel \cite{Triebel.1983}. 
For any two function spaces $X,Y$, $\mL(X,Y)$ is the space of all bounded linear operators 
of $X$ into $Y$, endowed with the operator norm. In short, $\mL(X) = \mL(X,X)$.  

\emph{Notation}. In this section, $C$ denotes a constant, possibly different from line to line, 
depending on the regularity exponents $k,s$, and independent of the functions involved in the inequality. 

\medskip

We begin with recalling some classical results, not stated in a general version, 
but just in the form they are needed below. 

\begin{lemma}[Trace operator] 
\label{lemma:trace.op}
The restriction map $u \to {\mathtt T} u = u |_{\S^2}$ extends uniquely to a continuous linear map 
\[
\mathtt T : H^1(B_1) \to H^{\frac12}(\S^2),
\]
which maps continuously  $H^k(B_1)$ into $H^{k-\frac12}(\S^2)$ for all integer $k \geq 1$.
\end{lemma}

\begin{proof}
See, e.g., \cite{Taylor.volume.1}, Proposition 4.5 in Section 4.4
\end{proof}

\begin{lemma}[Harmonic extension operator, or Poisson integral map] 
\label{lemma:def.harmonic.ext.op} 
Given $f \in C^\infty(\S^2)$, there exists a unique 
$u \in C^\infty(B_1) \cap C^0(\overline{B_1}) \cap H^1(B_1)$ 
such that 
\[
\Delta u = 0 \ \  \text{in } B_1, \quad \ 
u = f \ \  \text{on } \S^2.
\]
The linear map $f \to u$ has a unique continuous extension 
\[
\mathtt{PI} : H^{\frac12}(\S^2) \to H^1(B_1),
\]
which maps continuously $H^{k+\frac12}(\S^2)$ into $H^{k+1}(B_1)$ for all integer $k \geq 0$ and  satisfies $\|\mathtt{PI} f\|_{L^\infty(B_1)} \leq \|f\|_{L^\infty(\S^2)} $.   
\end{lemma}

\begin{proof} 
See, e.g., \cite{Taylor.volume.1}, Proposition 1.7 in Section 5.1. 
The $L^\infty$ bound follows from the maximum principle. 
\end{proof}

\begin{lemma}[Solution map for the Poisson problem in divergence form] 
\label{lemma:op.S}
For any $g \in L^2(B_1)$ there exists a unique $u \in H^1_0(B_1)$ such that $\Delta u = \div g$
in $B_1$ (in the sense of distributions). 
The map $g \to u$ is a linear continuous operator 
\[
\mathtt S : L^2(B_1) \to H^1_0(B_1),
\]
which maps continuously $H^{k}(B_1)$ into $H^{k+1}(B_1) \cap H^1_0(B_1)$ 
for all integer $k \geq 0$. 
\end{lemma}

\begin{proof} 
See, e.g., \cite{Taylor.volume.1}, 
Propositions 1.1 and 1.2, and Theorem 1.3 in Section 5.1. 
\end{proof}

\begin{lemma}[Sobolev extension operator] 
\label{lemma:def.sobolev.ext.op} 
For every $k \in \mathbb{N}$ there exists a linear continuous extension operator $\mathtt{E} : H^k(B_1) \to H^k(\R^3)$ such that 
\[
\| \mathtt{E}  u\|_{H^l(\R^3)} \leq C_k \| u \|_{H^l(B_1)} \quad \text{for all }\, l \leq k \quad \text{and} \quad \|\mathtt{E} u\|_{L^\infty(\R^3)} \leq C_k \| u \|_{L^\infty(B_1)}.  
\]
\end{lemma}

\begin{proof} 
The proof can be found in Section 4.4 in \cite{Taylor.volume.1}.  
It follows from the existence of the Sobolev extension map in the half-space 
\cite{Taylor.volume.1}, Lemma 4.1 in Section 4.4,  
and a partition of unity argument (see, e.g., \cite[Proposition 2.1]{Julin.La.Manna}). 
The bound for the $L^\infty$-norm follows from the explicit construction of the extension map for the half-space. 
\end{proof}

Using the extension operator from Lemma \ref{lemma:def.sobolev.ext.op} 
and the classical Sobolev embedding for $\R^3$ 
(see e.g.{} \cite{Taylor.volume.1}, Proposition 1.3 in Section 4.1), 
we deduce the Sobolev embedding for the unit ball  
\begin{equation}
\label{eq:Sobolev-inequality}
\| f \|_{L^\infty(B_1)} \leq C \| f \|_{H^2(B_1)}.
\end{equation}

We also have the following classical product estimate. 

\begin{lemma}[Product estimate for the unit ball] 
\label{lemma:prod.emb.est}
For all $k \in \N_0$, all $f,g \in H^k(B_1) \cap L^\infty(B_1)$, 
the product $fg$ belongs to $H^k(B_1)$, with 
\[
\| f g \|_{H^k(B_1)} 
\leq C (\| f \|_{L^\infty(B_1)} \| g \|_{H^k(B_1)} + \| f \|_{H^k(B_1)} \| g \|_{L^\infty(B_1)}).
\]
\end{lemma}

\begin{proof} 
We give the proof for the reader's convenience. Extend $f$ and $g$ by the extension operator given by Lemma \ref{lemma:def.sobolev.ext.op} and use the product estimate in $\R^3$ (see e.g. \cite{Taylor.volume.3}, Proposition 3.7  in Section 13.3) to deduce 
\[
\begin{split}
\| f g \|_{H^k(B_1)} \leq \| (\mathtt{E} f) (\mathtt{E} g) \|_{H^k(\R^3)} 
& \leq  C (\|\mathtt{E}  f \|_{L^\infty(\R^3)} \| \mathtt{E} g \|_{H^k(\R^3)} + \| \mathtt{E} f \|_{H^k(\R^3)} \| \mathtt{E}g \|_{L^\infty(\R^3)})\\
& \leq C (\| f \|_{L^\infty(B_1)} \|g \|_{H^k(B_1)} + \|  f \|_{H^k(B_1)} \| g \|_{L^\infty(B_1)}).
\qedhere 
\end{split}
\]
\end{proof}
We also recall the Sobolev embedding on the sphere 
(see, e.g., \cite{Taylor.volume.1}, Proposition 3.3 in Section 4.3.), which implies 
\begin{equation}  \label{embedding.Sd}
\| f \|_{L^\infty(\S^2)} \leq C \| f \|_{H^{\frac32}(\S^2)}, 
\quad 
\| f \|_{W^{1,\infty}(\S^2)} \leq C \| f \|_{H^{\frac52}(\S^2)}.
\end{equation}

Now we study problem \eqref{punct.modif.Lap.06}. 
Let $\psi \in H^{\frac12}(\S^2)$, $h \in W^{1, \infty}(\S^2)$. 
Suppose that $u = \tilde \Phi$ is the unique solution in $H^1(B_1)$ of problem \eqref{punct.modif.Lap.06}. 
Let $v$ be the difference $v := u - \mathtt{PI} \psi$. 
Then $v \in H^1_0(B_1)$ and  
\[ 
0 = \div (P \grad (v + \mathtt{PI} \psi) ) 
= \Delta v + \div ((P-I) \grad (v + \mathtt{PI} \psi)) 
\quad \text{in } B_1, 
\]
because $\mathtt{PI} \psi$ is a harmonic function. 
Hence $\Delta v = \div g$ 
where $g = (I-P) \grad (v+ \mathtt{PI} \psi) \in L^2(B_1)$, 
whence $v = \mathtt{S} g$, that is, 
\begin{equation} \label{eq.for.v.03}
v - A v = A (\mathtt{PI} \psi), \quad 
A := \mathtt{S}[ (I-P) \grad \, \cdot \,].
\end{equation}
Vice versa, if a function $v \in H^1_0(B_1)$ satisfies identity \eqref{eq.for.v.03}, 
then the sum $u = v + \mathtt{PI} \psi$ is the unique solution of \eqref{punct.modif.Lap.06} 
belonging to $H^1(B_1)$. 

We want to prove the invertibility of the operator $I-A$ on the LHS of \eqref{eq.for.v.03},
and the analytic dependence of $A$ on $h$. 

\begin{lemma} 
\label{lemma:est.A}
Let $A$ be the linear operator defined in \eqref{eq.for.v.03}.
For every integer $k \geq 2$ there exists $\d > 0$ such that 
the map $h \to A$ from $\{ h \in H^{k+1}(\S^2) : \| h \|_{W^{1,\infty}(\S^2)} < \delta \}$ 
to $\mL(H^{k+1}(B_1))$ is analytic, and 
\[
\| Av \|_{H^{k+1}(B_1)} 
\leq C ( \| h \|_{H^{k+1}(\S^2)} \| v \|_{H^3(B_1)} 
+ \| h \|_{W^{1, \infty}(\S^2)} \| v \|_{H^{k+1}(B_1)} ).
\]
\end{lemma}

\begin{proof}
By Lemma \ref{lemma:op.S}, 
\begin{align*}
\| A v \|_{H^{k+1}(B_1)} \leq C \| (I-P) \grad v \|_{H^k(B_1)}. 
\end{align*}
By Lemma \ref{lemma:prod.emb.est}, 
\[
\| (I-P) \grad v \|_{H^k(B_1)} 
\leq C ( \| I-P \|_{H^k(B_1)} \| \grad v \|_{L^\infty(B_1)} 
+ \| I-P \|_{L^\infty(B_1)} \| \grad v \|_{H^k(B_1)} )
\]
and 
\[
\| \grad v \|_{L^\infty(B_1)} 
\leq C \| \grad v \|_{H^2(B_1)} 
\leq C \| v \|_{H^3(B_1)}, 
\quad 
\| \grad v \|_{H^k(B_1)} 
\leq C \| v \|_{H^{k+1}(B_1)}.
\]
The matrix $P$ appearing in \eqref{punct.modif.Lap.06} is defined in \eqref{def.P},
and it is written in \eqref{punct.P.0} in terms of $h$. 
By \eqref{est.rho.cut.off}, one has $0 \leq \rho(r) + r \rho'(r) \leq 1 + \frac12 8 = 5$ 
for all $r \in \R$, whence 
\begin{equation*} 
0 \leq \chi \leq 1, \quad \ 
0 \leq \chi_1 \leq 5 \quad \forall x \in \R^3,
\end{equation*}
where $\chi$ is defined in \eqref{def.chi.rho.cut.off}
and $\chi_1$ in \eqref{def.chi.1}.
Thus, for $\| h \|_{L^\infty(\S^2)} < 1/5$, 
$P$ is given by the series $P = \sum_{n=0}^\infty P_n$, 
where $P_0 := I$,
\begin{align} 
\label{def.P.n.1}
P_1 & := \chi_1 h_0 I - \grad (\chi h_0) \otimes x - x \otimes \grad (\chi h_0), 
\\
P_n & := (- \chi_1 h_0)^{n-2} |\grad (\chi h_0)|^2 x \otimes x,  \quad n \geq 2, 
\label{def.P.n}
\end{align}
and $h_0 := \mE_0 h$ is the 0-homogeneous extension of $h$. 
Using bound \eqref{est.rho.cut.off} 
and the orthogonality property $\la \grad \chi , \grad h_0 \ra = 0$, 
one has 
\[
| (\chi_1 h_0)(x) | \leq 5 \| h \|_{L^\infty(\S^2)}, 
\quad 
| \grad(\chi h_0)(x) | |x| \leq 2 \| h \|_{W^{1,\infty}(\S^2)} 
\quad \forall x \in \R^3.
\]
Hence 
\begin{align*}
\| P_n \|_{L^\infty(\R^3)} 
& \leq (9 \| h \|_{W^{1,\infty}(\S^2)} )^n 
\quad \forall n \geq 1.
\end{align*}
As a consequence, for any fixed $\delta_0 \in (0,1/9)$, 
the series $\sum P_n$ is totally convergent in the $L^\infty(B_1)$ norm, 
uniformly for $h$ in the ball $\| h \|_{W^{1,\infty}(\S^2)} \leq \delta_0$, 
with
\begin{equation*} 
\| I-P \|_{L^\infty(B_1)} 
\leq \sum_{n=1}^\infty \| P_n \|_{L^\infty(B_1)} 
\leq C \| h \|_{W^{1,\infty}(\S^2)}.
\end{equation*}
By using a partition of unity and local coordinate systems flattening the boundary $\S^2$ 
of the ball $B_1$, one proves that 
\[
\| \chi h_0 \|_{H^k(B_1)} + \| \chi_1 h_0 \|_{H^k(B_1)} 
\leq C \| h \|_{H^{k}(\S^2)}.
\]
Hence 
\[
\| P_n \|_{H^k(B_1)} \leq (C \| h \|_{W^{1,\infty}(\S^2)})^{n-1} \| h \|_{H^{k+1}(\S^2)}, 
\quad n \geq 1,
\]
and 
\begin{equation*} 
\| I-P \|_{H^k(B_1)} 
\leq \sum_{n=1}^\infty \| P_n \|_{H^k(B_1)} 
\leq C \| h \|_{H^{k+1}(\S^2)}
\end{equation*}
for $C \| h \|_{W^{1,\infty}(\S^2)} < 1$. 
The thesis follows from the estimates above. 
\end{proof}

\begin{lemma} 
\label{lemma:inv.I-A}
For every integer $k \geq 2$ there exists $\d > 0$ such that, 
for all $h \in H^{k+1}(\S^2)$ with $\| h \|_{H^3(\S^2)} < \delta$,  
the map $I-A$ is an invertible linear operator of $H^{k+1}(B_1)$ onto itself, 
and 
\[
\| (I-A) v \|_{H^{k+1}(B_1)} + \| (I-A)^{-1} v \|_{H^{k+1}(B_1)} 
\leq C ( \| v \|_{H^{k+1}(B_1)} + \| h \|_{H^{k+1}(\S^2)} \| v \|_{H^3(B_1)} ). 
\]
The map $h \to (I-A)^{-1}$ from $\{ h \in H^{k+1}(\S^2) : \| h \|_{H^3(\S^2)} < \delta \}$ 
to $\mL( H^{k+1}(B_1) )$ is analytic.
\end{lemma}

\begin{proof}
Using Lemma \ref{lemma:est.A}, one proves, by induction on $n$, 
that there exists $C$, depending on $k$ and independent of $n$, 
such that 
\begin{align*}
\| A^n v \|_{H^3(B_1)} 
& \leq ( C \| h \|_{H^3(\S^2)} )^n \| v \|_{H^3(B_1)}, \quad 
\\
\| A^n v \|_{H^{k+1}(B_1)} 
& \leq C^n ( \| h \|_{H^3(\S^2)}^n \| v \|_{H^{k+1}(B_1)} 
+ n \| h \|_{H^3(\S^2)}^{n-1} \| h \|_{H^{k+1}(\S^2)} \| v \|_{H^3(B_1)} )
\end{align*}
for all $n \in \N$. The thesis follows by Neumann series. 
\end{proof}

From the analyticity and invertibility of $I-A$ 
we deduce the following result for the solution of equation \eqref{eq.for.v.03}.

\begin{lemma}
\label{lemma:sol.v}
Fix $k \geq 2$, and let $\d > 0$ be as in Lemma \ref{lemma:inv.I-A}. 
Let $h \in H^{k+1}(\S^2)$, $\| h \|_{H^3(\S^2)} < \delta$, $\psi \in H^{k+\frac12}(\S^2)$. 
Then equation \eqref{eq.for.v.03} has a unique solution 
$v = (I-A)^{-1} A (\mathtt{PI} \psi) \in H^{k+1}(B_1)$, with 
\[
\| v \|_{H^{k+1}(B_1)} 
\leq C ( \| h \|_{H^3(\S^2)} \| \psi \|_{H^{k+\frac12}(\S^2)} 
+ \| h \|_{H^{k+1}(\S^2)} \| \psi \|_{H^{\frac52}(\S^2)} ).
\]
The map $h \to v$ from $\{ h \in H^{k+1}(\S^2) : \| h \|_{H^3(\S^2)} < \delta \}$ 
to $\mL ( H^{k+\frac12}(\S^2) , H^{k+1}(B_1) )$ is analytic.
\end{lemma}

\begin{proof}
Apply Lemmas \ref{lemma:def.harmonic.ext.op}, \ref{lemma:est.A}, \ref{lemma:inv.I-A}. 
\end{proof}

We recall that  $v := u - \mathtt{PI} \psi$, 
and therefore we obtain the following lemma for $u$.

\begin{lemma}
\label{lemma:sol.u}
Let $k, \delta, h, \psi$ be like in Lemma \ref{lemma:sol.v}.
The solution $u$ of \eqref{punct.modif.Lap.06} satisfies 
\begin{align*}
\| u \|_{H^{k+1}(B_1)} 
& \leq C ( \| \psi \|_{H^{k+\frac12}(\S^2)} + \| h \|_{H^{k+1}(\S^2)} \| \psi \|_{H^{\frac52}(\S^2)} ),
\\
\| \la \grad u , x \ra \|_{H^{k-\frac12}(\S^2)} 
& \leq C \| u \|_{H^{k+1}(B_1)}. 
\end{align*}
The map $h \to u$ from $\{ h \in H^{k+1}(\S^2) : \| h \|_{H^3(\S^2)} < \delta \}$ 
to $\mL ( H^{k+\frac12}(\S^2) , H^{k+1}(B_1) )$ is analytic. 
\end{lemma}

\begin{proof}
The analyticity and the first inequality follow from $u  = v + \mathtt{PI} \psi$ 
and Lemmas \ref{lemma:def.harmonic.ext.op} and \ref{lemma:sol.v}. 
The second inequality follows from Lemma \ref{lemma:trace.op}, Lemma \ref{lemma:prod.emb.est} 
and \eqref{eq:Sobolev-inequality}.
\end{proof}

Let us focus on the Dirichlet-Neumann operator defined in \eqref{formula.G.P}. 
The previous lemma is crucial, as it enables us to estimate the most difficult term in \eqref{formula.G.P}, 
which is $\la \nabla \tilde \Phi, x\ra$. In order to estimate the other terms, 
we first state the product estimate on the sphere and an estimate on the tangential gradient.  
The following lemma holds for all real $s>0$, but for simplicity we state it only for half-integers. 

\begin{lemma} \label{lemma:est.tools.S.n-1} 
For all $k \in \N_0$, all $f,g \in H^{k+\frac12}(\S^2) \cap L^\infty(\S^2)$, 
the product $fg$ belongs to $H^{k+\frac12}(\S^2) $, with  
\[
\begin{split}
\| fg \|_{H^{k+\frac12}(\S^2)} 
&\leq C ( \| f \|_{L^{\infty}(\S^{2})} \| g \|_{H^{k+\frac12}(\S^2)} +\| f \|_{H^{k+\frac12}(\S^2)} \| g \|_{L^{\infty}(\S^{2})}   )\\
&\leq  C ( \| f \|_{H^{\frac32}(\S^2))} \| g \|_{H^{k+\frac12}(\S^2)} +\| f \|_{H^{k+\frac12}(\S^2)} \| g \|_{H^{\frac32}(\S^{2})}   ). 
\end{split}
\]
If $f \in H^{k+\frac12}(\S^2)$ and $k \geq 1$ then
\[
\| \grad_{\S^2} f \|_{H^{k-\frac12}(\S^2)} 
\leq C \| f \|_{H^{k+\frac12}(\S^2)}.  
\]
\end{lemma}

\begin{proof}
For the first inequality we extend $f$ and $g$ with the harmonic extension 
from Lemma \ref{lemma:def.harmonic.ext.op} and have, 
by the trace estimate from Lemma \ref{lemma:trace.op} 
and the product estimate from Lemma \ref{lemma:prod.emb.est}, that 
\[
\begin{split}
\| fg \|_{H^{k+\frac12}(\S^2)}  &\leq C \| (\mathtt{PI} f)(\mathtt{PI} g) \|_{H^{k+1}(B_1)}\\
&\leq C ( \| \mathtt{PI} f \|_{L^{\infty}(B_1)} \|  \mathtt{PI} g \|_{H^{k+1}(B_1)} +\|  \mathtt{PI} f \|_{H^{k+1}(B_1)} \|  \mathtt{PI} g\|_{L^{\infty}(B_1)}  ). 
\end{split}
\]
The first inequality then follows from Lemma \ref{lemma:def.harmonic.ext.op} 
and the Sobolev embedding \eqref{embedding.Sd}.

For the second inequality we again extend $f$ with the harmonic extension $\mathtt{PI} f$, 
and we recall that $\grad_{\S^2} f$ is the trace on $\S^2$ of the difference 
$\nabla (\mathtt{PI} f) - \la \nabla (\mathtt{PI} f), x \ra x$. 
By the trace estimate in Lemma \ref{lemma:trace.op}, 
the triangular inequality, 
the product estimate in Lemma \ref{lemma:prod.emb.est}, 
and Lemma \ref{lemma:def.harmonic.ext.op}, 
one has 
\[
\begin{split}
\| \grad_{\S^2} f \|_{H^{k+\frac12}(\S^2)} 
& = \| \nabla (\mathtt{PI} f) - \la \nabla (\mathtt{PI} f), x \ra x \|_{H^{k+\frac12}(\S^2)} 
\\
& \leq C \| \nabla (\mathtt{PI} f) - \la \nabla (\mathtt{PI} f), x \ra x \|_{H^{k+1}(B_1)} 
\\
& \leq C ( \| \nabla (\mathtt{PI} f) \|_{H^{k+1}(B_1)} 
+ \| \la \nabla (\mathtt{PI} f), x \ra x \|_{H^{k+1}(B_1)} )
\\
& \leq C \| \nabla (\mathtt{PI} f) \|_{H^{k+1}(B_1)} 
\\
& \leq C \| \mathtt{PI} f \|_{H^{k+2}(B_1)} 
\\
& \leq C \| f \|_{H^{k+\frac32}(\S^2)}. 
\qedhere
\end{split}
\]
\end{proof}

\begin{lemma} \label{lemma:other.terms}
Let $k \geq 1$. There exists $\delta > 0$ such that, for $\| h \|_{H^{\frac52}(\S^2)} \leq \delta$, 
one has  
\begin{align*} 
\| (1+h)^\a \|_{H^{k+\frac12}(\S^2)}  
& \leq C (1 + \| h \|_{H^{k+\frac12}(\S^2)}), 
\quad \a \in \{ -1, -2 \}, 
\\
\| \{ (1+h)^2 + | \grad_{\S^2} h |^2 \}^{\pm \frac12} \|_{H^{k+\frac12}(\S^2)}  
& \leq C (1 + \| h \|_{H^{k + \frac32}(\S^2)}) .
\end{align*}
\end{lemma}

\begin{proof}
Using the product estimate from Lemma \ref{lemma:est.tools.S.n-1} inductively, 
we have for every $m \in \N$ that 
\[
\|  h^m \|_{H^{k+\frac12}(\S^2)}  \leq \big( 2C \| h \|_{H^{\frac32}(\S^2)}\big)^{m-1} \| h \|_{H^{k+\frac12}(\S^2)},
\]
where $C\geq 2$.  To estimate $(1+h)^\a$, we write it as a power series around $h=0$ 
and apply the above  estimate  to each term of the series. 
We estimate $(1+g)^p$, $p = \pm \frac12$, in the same way, 
with $g = 2h + h^2 + |\grad_{\S^2} h|^2$. Finally  
 we estimate $g$ by Lemma \ref{lemma:est.tools.S.n-1}.
\end{proof}

We may now prove the main result of this section.

\begin{theorem}[Tame estimate and analyticity of the Dirichlet-Neumann operator]
\label{thm:tame.est.DN.op}
Let $k \geq 2$ be an integer. There exists $\delta > 0$ such that, 
for $h \in H^{k+1}(\S^2)$, $\psi \in H^{k+\frac12}(\S^2)$, $\| h \|_{H^{3}(\S^2)} \leq \delta$, 
the Dirichlet-Neumann operator $G(h)\psi$ satisfies 
\[
\| G(h)\psi \|_{H^{k-\frac12}(\S^2)} 
\leq C ( \| \psi \|_{H^{k+\frac12}(\S^2)} + \| h \|_{H^{k+1}(\S^2)} \| \psi \|_{H^{\frac52}(\S^2)} ).  
\]
The map $h \to G(h)$ is analytic from $\{ h \in H^{k+1}(\S^2) : \| h \|_{ H^{3}(\S^2) } < \delta \}$ 
to $\mL (H^{k + \frac12}(\S^2) , H^{k-\frac12}(\S^2) )$.
\end{theorem}

\begin{proof}
Use formula \eqref{formula.G.P} 
and Lemmas \ref{lemma:sol.u}, \ref{lemma:est.tools.S.n-1}, \ref{lemma:other.terms}.
\end{proof}

\begin{remark} \label{rem:weaker.because}
The tame estimate in Theorem \ref{thm:tame.est.DN.op} can be improved in several respects: 
$(i)$ the $H^{k+1}(\S^2)$ norm of $h$ on the RHS of the inequality 
can be replaced by its $H^{k+\frac12}(\S^2)$ norm;
$(ii)$ the radius $\delta$ can be proved to be independent of the high regularity $k$;
$(iii)$ the regularity parameter $k$ can be real, not only integer; 
$(iv)$ the constant in front of the $H^{k+\frac12}(\S^2)$ norm of $\psi$ 
can be taken independent of $k$;
$(v)$ the regularity threshold $k=2$ can be made lower.
These technical improvements require a longer proof, which we defer to a forthcoming paper. 
\end{remark}

\section{Traveling waves}

In this section we construct traveling waves. 
Let $h, \psi$ be functions of the form 
\begin{equation} \label{ansatz.h.psi}
h(t,x) = \eta(R(\om t)x), \quad 
\psi(t,x) = \beta(R(\om t)x), \quad 
t \in \R, \ \  
x \in \S^2,
\end{equation}
where $\eta, \beta : \S^2 \to \R$ are scalar functions defined on $\S^2$, 
independent of time, 
$\om \in \R$ is the angular velocity parameter, 
and $R(\th)$ is the rotation matrix 
\begin{equation} \label{def.R(t)} 
R(\th) = \begin{pmatrix} 
\cos \th & - \sin \th & 0 \\
\sin \th & \cos \th & 0 \\
0 & 0 & 1 
\end{pmatrix}.
\end{equation}
The transformation law for the time derivative $\pa_t h, \pa_t \psi$ is the following one.

\begin{lemma}
\label{lemma:pat.transf.law}
The matrix $R(\th)$ in \eqref{def.R(t)} satisfies 
\begin{equation}  \label{def.mJ}
\pa_\th R(\th) R^T(\th) 
= \begin{pmatrix} 
0 	& - 1 	& 0 \\
1 	&  0 		& 0 \\
0 	&  0 		& 0 
\end{pmatrix}
=: \mJ
\end{equation}
for all $\th \in \R$. 
The function $h(t,x)$ defined in \eqref{ansatz.h.psi} satisfies 
\begin{equation}  \label{pat.transf.law}
\pa_t h(t,x) 
= \om \la \mJ y, \grad_{\S^2} \eta(y) \ra 
\end{equation}
where $y = R(\om t) x$, 
for all $x \in \S^2$, 
all $t, \om \in \R$.  
\end{lemma}

\begin{proof}
The proof is a straightforward calculation.
\end{proof}

For the other terms of the equations, 
the time variable plays the role of a parameter. 
We have the following transformation laws. 

\begin{lemma}
\label{lemma:pax.transf.law}
Let $M \in \mathrm{Mat}_{3 \times 3}(\R)$ be an orthogonal matrix, 
i.e., $M M^T = M^T M = I$. 
Let $h(x) = \eta(Mx)$ and $\psi(x) = \beta(Mx)$ 
for all $x \in \S^2$, 
where $h, \psi, \eta, \beta : \S^2 \to \R$.
Then 
\begin{align}  
(\mE_0 h)(x) 
& = (\mE_0 \eta)(Mx) \quad \forall x \in \R^3 \setminus \{ 0 \},
\label{tool.M.01} 
\\
\grad_{\S^2} h(x) 
& = M^T (\grad_{\S^2} \eta)(Mx) 
\quad \forall x \in \S^2,
\label{tool.M.02} 
\\
|\grad_{\S^2} h(x)| 
& = |(\grad_{\S^2} \eta)(Mx)|
\quad \forall x \in \S^2,
\label{tool.M.03} 
\\
\Delta_{\S^2} h(x) 
& = (\Delta_{\S^2} \eta)(Mx)
\quad \forall x \in \S^2,
\label{tool.M.04} 
\\
\la (D^2_{\S^2} h)(x) \grad_{\S^2} h(x) , \grad_{\S^2} h(x) \ra 
& = \la (D^2_{\S^2} \eta)(Mx) (\grad_{\S^2} \eta)(Mx) , (\grad_{\S^2} \eta)(Mx) \ra 
\quad \forall x \in \S^2,
\label{tool.M.05} 
\\
H(h)(x) 
& = H(\eta)(Mx)
\quad \forall x \in \S^2,
\label{tool.M.06} 
\\
\la \grad_{\S^2} h(x), \grad_{\S^2} \psi(x) \ra 
& = \la (\grad_{\S^2} \eta)(Mx), (\grad_{\S^2} \beta)(Mx) \ra 
\quad \forall x \in \S^2,
\label{tool.M.07} 
\\
P_h(x) 
& = M^T P_\eta(Mx) M
\quad \forall x \in \R^3 \setminus \{ 0 \},
\label{tool.M.08} 
\\
(G(h)\psi)(x)
& = (G(\eta)\beta)(Mx)
\quad \forall x \in \S^2,
\label{tool.M.09} 
\end{align}
where $\mE_0 h$ is the $0$-homogeneous extension of $h$, 
and $P_h(x) = (1 + h_0(x)) I - \grad h_0(x) \otimes x - x \otimes \grad h_0(x) 
+ (1 + h_0(x))^{-1} |\grad h_0(x)|^2 x \otimes x$, with $h_0 = \mE_0 h$, 
is the matrix in \eqref{def.P} and \eqref{punct.P.0}. 
\end{lemma}

\begin{proof} 
All the properties can be easily proved using the identities $M M^T = M^T M = I$, 
which imply that the map $x \mapsto Mx$ is an isometry; 
to prove \eqref{tool.M.09}, use also the uniqueness of the solution 
of the elliptic problem \eqref{punct.modif.Lap.06}.
\end{proof}

By Lemma \ref{lemma:pat.transf.law}
and Lemma \ref{lemma:pax.transf.law}, 
system \eqref{kin.eq.08}, \eqref{dyn.eq.08} 
for the unknowns $h, \psi$ 
satisfying ansatz \eqref{ansatz.h.psi}
becomes the equation
\begin{equation} \label{eq.mF}
\mF(\om,u) = 0
\end{equation}
for the unknown $u = (\eta, \beta)$ on $\S^2$, 
where 
\begin{align} 
\mF(\om, u) & := (\mF_1(\om, u) , \mF_2(\om, u)),
\label{def.mF}
\\
\mF_1(\om,u) & := 
\om \pois \eta 
- \frac{\sqrt{(1+\eta)^2 + |\grad_{\S^2} \eta|^2}}{1 + \eta} \, G(\eta) \beta,
\label{def.mF.1} 
\\
\mF_2(\om,u) & := 
\om \pois \beta 
- \frac12 \Big( G(\eta) \beta
+ \frac{\la \grad_{\S^2} \beta , \grad_{\S^2} \eta \ra}{(1+\eta) 
\sqrt{(1+\eta)^2 + |\grad_{\S^2} \eta|^2}} \Big)^2
+ \frac{|\grad_{\S^2} \beta|^2}{2(1+\eta)^2} 
+ \s_0 \big(H(\eta) - 2\big), 
\label{def.mF.2} 
\end{align} 
and $\pois$ is the linear operator 
\begin{equation}  \label{def.mL}
\pois f(x) := \la \mJ x , \grad_{\S^2} f(x) \ra, \quad x \in \S^2,
\end{equation}
for any $f : \S^2 \to \R$.
Note that, since $\mJ x$ belongs to the tangent space $T_x(\S^2)$, 
one also has 
\begin{equation} \label{Poisson.brackets}
\pois f(x) = \la \mJ x, \grad \tilde f(x) \ra
= (x_1 \pa_{x_2} - x_2 \pa_{x_1}) \tilde f(x) 
\quad \forall x \in \S^2,
\end{equation}
for any extension $\tilde f$ of $f$ to an open neighborhood of $\S^2$. 
We also observe that $\mF(\om,0) = 0$ for all $\om \in \R$.

\begin{lemma} 
\label{lemma:mF.Hs}
Let $k \geq 2$ and let 
\[
U := \{ u = (\eta, \beta) : 
\eta \in H^{k+1}(\S^2,\R), \ \ 
\beta \in H^{k+\frac12}(\S^2,\R), \ \ 
\| \eta \|_{H^{3}(\S^2)} < \delta \},
\]
where $\delta$ is the constant in Theorem \ref{thm:tame.est.DN.op}.
Then $\mF_1(\om,u) \in H^{k-\frac12}(\S^2,\R)$, 
$\mF_2(\om,u) \in H^{k-1}(\S^2,\R)$ 
for all $u \in U$, $\om \in \R$, 
and the map 
\[
\mF : \R \times U \to H^{k -\frac12}(\S^2,\R) \times H^{k-1}(\S^2,\R)
\]
is analytic.
\end{lemma}

\begin{proof}
It is a straightforward consequence of the properties 
and estimates proved in Section \ref{sec:analytic.DN.op} and of Lemma \ref{lem:meancurvature}.
\end{proof}

\subsection{The linearized operator at zero}

We calculate the linearized operator 
$L := \pa_u \mF(\om,0)$ at $u=0$, which is the linear operator
\[
L : H^{k+1}(\S^2) \times H^{k+\frac12}(\S^2) 
\to H^{k - \frac12}(\S^2) \times H^{k-1}(\S^2),
\]
\begin{equation} \label{def.L}
L (\eta, \beta) = \begin{pmatrix} 
\om \pois \eta 
- G(0) \beta 
\\
\om \pois \beta 
- \s_0 (2 \eta + \Delta_{\S^2} \eta)
\end{pmatrix}
= \begin{pmatrix} \om \pois & - G(0) \\ 
- \s_0 (2 + \Delta_{\S^2}) & \om \pois 
\end{pmatrix}
\begin{pmatrix}
\eta \\ \beta 
\end{pmatrix}.
\end{equation}
The operators $G(0)$ and $2 + \Delta_{\S^2}$ are diagonalized by the real spherical harmonics, 
with 
\[
G(0) \ph_\ell = \ell \ph_\ell, \quad \ 
- (2 + \Delta_{\S^2}) \ph_\ell = (\ell + 2)(\ell - 1) \ph_\ell \quad \
\forall \ph_\ell \in \mH_\ell(\S^2,\R), \quad 
\ell \in \N_0,
\]
where $\mH_\ell(\S^2,\R)$ is the space of the real spherical harmonics of degree $\ell$; 
as is well known, it is a vector space of dimension $(2\ell+1)$ on $\R$. 
The operator $\pois$ can also be block-diagonalized by real spherical harmonics; 
in particular, the restriction of $\pois$ to $\mH_\ell(\S^2,\R)$ 
can be represented by a block-diagonal matrix 
with $\ell$ 2-blocks $( \begin{smallmatrix} 0 & - m \\ m & 0 \end{smallmatrix} )$, 
$m = 1, \ldots, \ell$, and one 1-block $0$ 
(using complex spherical harmonics, $\pois$ becomes diagonal with complex eigenvalues $im$, 
$m = -\ell, \ldots, \ell$).

As is well known, an $L^2(\S^2,\R)$ orthonormal basis of $\mH_\ell(\S^2,\R)$ 
is given by the classical real spherical harmonics
\begin{align}
Y_{\ell,m}^{(\cos)}(\theta,\phi) 
& = c_\ell^{(m)} (\sin \theta)^m P_\ell^{(m)}(\cos \theta) \cos(m \phi), 
\quad m = 0, \ldots, \ell,
\notag \\ 
Y_{\ell,m}^{(\sin)}(\theta,\phi) 
& = c_\ell^{(m)} (\sin \theta)^m P_\ell^{(m)}(\cos \theta) \sin(m \phi), 
\quad m = 1, \ldots, \ell,
\label{real.spherical.harmonics.Y}
\end{align}
commonly written as functions of the angles 
$\theta \in [0, \pi]$,
$\phi \in [0, 2 \pi]$
expressing any point $x \in \S^2$ in spherical coordinates 
$x_1 = \sin \theta \cos \phi$, 
$x_2 = \sin \theta \sin \phi$, 
$x_3 = \cos \theta$. 
Here $P_\ell^{(m)}(t)$ is the $m$-th derivative 
of the ordinary Legendre polynomial $P_\ell(t)$, 
which is a polynomial of degree $\ell$ with real coefficients, 
with parity $P_\ell(-t) = (-1)^\ell P_\ell(t)$, 
and $c_\ell^{(m)} \in \R$ is a normalizing coefficient;
see, e.g., \cite{Atkinson.Han}, Example 2.48 in Section 2.11. 
For $a_m, b_m \in \R$, the sum 
$a_m \cos(m \phi) + b_m \sin(m \phi)$ 
is the real part of the complex number $(a_m - i b_m) e^{im\phi}$. 
Hence any linear combination of \eqref{real.spherical.harmonics.Y}
with real coefficients $a_m, b_m$ can be written as 
\begin{align*}
& \sum_{m=0}^\ell a_m Y_{\ell,m}^{(\cos)}(\theta,\phi) 
+ \sum_{m=1}^\ell b_m Y_{\ell,m}^{(\sin)}(\theta,\phi) 
\\ 
& \quad 
= a_0 c_\ell^{(0)} P_\ell(\cos \theta) 
+ \sum_{m=1}^\ell c_\ell^{(m)} P_\ell^{(m)}(\cos \theta) 
\Re \{ (a_m - i b_m) (\sin \theta)^m e^{im\phi} \},
\end{align*}
which, in Cartesian coordinates, becomes
\[
= a_0 c_\ell^{(0)} P_\ell(x_3) 
+ \sum_{m=1}^\ell c_\ell^{(m)} P_\ell^{(m)}( x_3 ) 
\Re \{ (a_m - i b_m) (x_1 + i x_2)^m \}.
\]
Hence the functions
\begin{align}
\ph_{\ell,0}(x) := c_{\ell}^{(0)} P_\ell(x_3), 
\quad \ 
\ph_{\ell,m}^{(\Re)}(x) & := c_{\ell}^{(m)} P_\ell^{(m)}(x_3) \Re[ (x_1 + i x_2)^m ], 
\quad m = 1, \ldots, \ell,
\notag \\
\ph_{\ell,m}^{(\Im)}(x) & := c_{\ell}^{(m)} P_\ell^{(m)}(x_3) \Im[ (x_1 + i x_2)^{m} ], 
\quad m = 1, \ldots, \ell, 
\label{def.ph.ell.m.Re.Im}
\end{align}
with $x \in \S^2$, 
form an $L^2(\S^2,\R)$ orthonormal basis 
of the real vector space $\mH_\ell(\S^2,\R)$. 
For notation convenience, we denote 
\begin{equation} \label{def.ph.ell.m}
\ph_{\ell,m} := \ph_{\ell,m}^{(\Re)}, \quad  
\ph_{\ell,-m} := \ph_{\ell,m}^{(\Im)}, \quad 
m = 1, \ldots, \ell.
\end{equation} 
Thus, $\{ \ph_{\ell,m} : m = - \ell, \ldots, \ell \}$ 
is an $L^2(\S^2,\R)$ orthonormal basis of $\mH_\ell(\S^2,\R)$.
This is the basis of Legendre real spherical harmonics in Cartesian coordinates.

\begin{lemma}
\label{lemma:block.diag.mS}
One has 
$\pois \ph_{\ell,m} = - m \ph_{\ell,-m}$ 
for all $m = -\ell, \ldots, \ell$,
all $\ell \in \N_0$.
\end{lemma}

\begin{proof} 
To apply \eqref{Poisson.brackets}, 
we observe that the functions in \eqref{def.ph.ell.m.Re.Im} 
have a natural extension, which we write without changing the notation, 
to a neighborhood of the sphere; 
such extensions are simply obtained by extending the validity of the formulae 
in \eqref{def.ph.ell.m.Re.Im}. 
In general, these extensions are neither harmonic nor homogeneous, 
but \eqref{Poisson.brackets} holds without requiring those properties. 
One has 
\[
(x_1 \pa_{x_2} - x_2 \pa_{x_1}) \{ (x_1 + i x_2)^m \} 
= i m (x_1 + i x_2)^m 
\]
for all $m \in \N$, all $(x_1, x_2) \in \R^2$, 
and   
\begin{align*}
(x_1 \pa_{x_2} - x_2 \pa_{x_1}) \Re \{ (x_1 + i x_2)^m \} 
& = - m \, \Im \{ (x_1 + i x_2)^m \}, 
\\
(x_1 \pa_{x_2} - x_2 \pa_{x_1}) \Im \{ (x_1 + i x_2)^m \} 
& = m \, \Re \{ (x_1 + i x_2)^m \}. 
\end{align*}
Therefore, by \eqref{Poisson.brackets}, we obtain
\begin{equation} \label{Poisson.brackets.phi.ell.m.Re.Im}
\pois \ph_{\ell,0}(x) = 0, \quad  
\pois \ph_{\ell,m}^{(\Re)}(x) = - m \ph_{\ell,m}^{(\Im)}, \quad 
\pois \ph_{\ell,m}^{(\Im)}(x) = m \ph_{\ell,m}^{(\Re)}, \quad 
m = 1, \ldots, \ell.
\end{equation}
Recalling the notation in \eqref{def.ph.ell.m}, 
this completes the proof.
\end{proof}

Given $(f,g) \in H^{s}(\S^2) \times H^{s-\frac12}(\S^2)$, 
we study the equation $L(\eta, \beta) = (f,g)$.
We use the real spherical harmonics $(\ph_{\ell,m})$ of Lemma \ref{lemma:block.diag.mS}
to decompose 
\begin{equation} \label{def.mT}
\eta = \sum_{(\ell,m) \in \mT} \hat \eta_{\ell,m} \ph_{\ell,m}, 
\quad \mT = \bigcup_{\ell \in \N_0} \mT_\ell, 
\quad \mT_\ell = \{ (\ell,m) : m = - \ell, \ldots, \ell \},
\end{equation}
with $\hat \eta_{\ell,m} \in \R$, 
and similarly for $\beta, f, g$. 
Hence 
\[
\pois \eta 
= \sum_{(\ell,m) \in \mT} \hat \eta_{\ell,m} (-m) \ph_{\ell,-m} 
= \sum_{(\ell,m) \in \mT} m \hat \eta_{\ell,-m} \ph_{\ell,m}.
\] 
One has $L(\eta,\beta) = (f,g)$ if and only if 
\begin{equation} \label{syst.coeff.lin.eq}
- \om m \hat \eta_{\ell,m} - \ell \hat \beta_{\ell,-m} 
= \hat f_{\ell,-m}, 
\quad \ 
\sigma_0 (\ell+2)(\ell-1) \hat \eta_{\ell,m} + \om m \hat \beta_{\ell,-m} 
= \hat g_{\ell,m}
\quad \ 
\forall (\ell,m) \in \mT,
\end{equation}
that is,
\begin{equation} \label{syst.coeff.matrix}
L_{\ell,m}
\begin{pmatrix} 
\hat \eta_{\ell,m} \\ 
\hat \beta_{\ell,-m} 
\end{pmatrix}
=
\begin{pmatrix} 
\hat f_{\ell,-m} \\ 
\hat g_{\ell,m} 
\end{pmatrix}
\quad \ 
\forall (\ell,m) \in \mT, 
\quad \ 
L_{\ell,m} := \begin{pmatrix} 
- \om m 										& - \ell \\
\sigma_0 (\ell+2)(\ell-1) & \om m
\end{pmatrix}.
\end{equation}
One has 
\begin{equation} \label{det.L.ell.j}
\det L_{\ell,m} = - \om^2 m^2 + \sigma_0 (\ell+2)(\ell-1)\ell,
\end{equation}
and bifurcation can only occur at values of $\om$ 
such that \eqref{det.L.ell.j} vanishes at some $(\ell,m)$. 
Thus, we assume that 
\begin{equation} \label{om.fix}
\om = \sqrt{\sigma_0} \frac{ \sqrt{(\ell_0+2)(\ell_0-1)\ell_0} }{m_0}
\end{equation}
for some fixed integers $\ell_0, m_0$, 
with $\ell_0 \geq 2$ and $1 \leq m_0 \leq \ell_0$.
With $\om$ in \eqref{om.fix}, a pair $(\ell,m)$ gives 
$\det L_{\ell m} = 0$ if and only if
\begin{equation} \label{dioph.eq.travelling}
(\ell+2)(\ell-1)\ell = c_0 m^2, \quad \ 
c_0 := (\ell_0+2)(\ell_0-1) \ell_0 m_0^{-2}.
\end{equation}

\begin{lemma}  
\label{lemma:bound.det.L.ell.j}
Let $\ell_0, m_0 \in \N$, with $\ell_0 \geq 2$, $1 \leq m_0 \leq \ell_0$. 
Let $S \subset \mT$ be the set of the pairs $(\ell,m) \in \mT$ satisfying 
\eqref{dioph.eq.travelling}. 
Then $S$ has a finite number of elements, 
which are $(\ell,m) = (\ell_0, m_0)$, $(\ell_0, -m_0)$, $(1,0)$, $(0,0)$, 
and possibly finitely many other pairs, all of which with $\ell \leq c_0$. 
Moreover, assuming \eqref{om.fix}, 
there exists a constant $C > 0$, depending on $\ell_0, \sigma_0$, such that 
\begin{equation} \label{bound.det.L.ell.j}
|\det L_{\ell, m}| \geq C \ell^3 
\quad \forall (\ell,m) \in \mT \setminus S. 
\end{equation}
\end{lemma}

\begin{proof} 
For any $(\ell,m) \in \mT$ one has $m^2 \leq \ell^2$. 
If $(\ell,m) \in S$, then 
$(\ell+2)(\ell-1)\ell = c_0 m^2 \leq c_0 \ell^2$.  
For $\ell \geq 2$ one has $\ell^2 \leq (\ell+2)(\ell-1)$, 
and we deduce that $\ell \leq c_0$ for all $(\ell,m) \in S$ with $\ell \geq 2$. 
The bound $\ell \leq c_0$ holds also for $\ell=0,1$ because 
$c_0 \geq (\ell_0+2)(\ell_0-1) \ell_0^{-1} \geq \ell_0 \geq 2$.

For $(\ell,m) \in \mT \setminus S$, one has $\det L_{\ell,m} \neq 0$. 
Moreover, $\det L_{\ell,m} \geq \det L_{\ell \ell}$ because $|m| \leq \ell$, 
and   for $\ell \geq 2 c_0$ it holds  $\det L_{\ell \ell}$ 
$= \sigma_0 (\ell+2)(\ell-1)\ell - \sigma_0 c_0 \ell^2 
\geq \sigma_0 \ell^3 - \sigma_0 c_0 \ell^2 \geq \frac12 \sigma_0 \ell^3$. 
On the other hand, 
$\min \{ |\det L_{\ell m}| \ell^{-3} : (\ell,m) \in \mT \setminus S, \ \ell < 2 c_0 \}$ 
is also positive because it is the minimum of a finite set of positive numbers. 
\end{proof}

For $(\ell,m) \in \mT \setminus S$, one has $\det L_{\ell,m} \neq 0$, 
and system \eqref{syst.coeff.matrix} with $(\hat f_{\ell,-m}, \hat g_{\ell,m}) = (0,0)$ 
has only the trivial solution $(\hat \eta_{\ell,m}, \hat \beta_{\ell,-m}) = (0,0)$. 
For $(\ell,m) \in S$, we distinguish $\ell=0$ from $\ell > 0$. 
For $(\ell,m) = (0,0)$, system \eqref{syst.coeff.matrix} 
with $(\hat f_{0,0}, \hat g_{0,0}) = (0,0)$ has solutions 
$( \hat \eta_{0,0} , \hat \beta_{0,0} ) = ( 0 , \lm )$, $\lm \in \R$. 
For $(\ell,m) \in S$ with $\ell \geq 1$, 
\eqref{syst.coeff.matrix} with $(\hat f_{\ell,-m}, \hat g_{\ell,m}) = (0,0)$ 
has solutions $(\hat \eta_{\ell,m}, \hat \beta_{\ell,-m}) = \lm (\ell, - \om m)$, $\lm \in \R$. 
Hence the kernel of the linear operator $L$ is the finite dimensional space 
\begin{equation} \label{def.V}
V := \ker L = \bigg\{ 
\begin{pmatrix} \eta \\ \beta \end{pmatrix} 
= \lm_{0,0} \begin{pmatrix} 0 \\ \ph_{0,0} \end{pmatrix} 
+ \sum_{ \begin{subarray}{c} (\ell,m) \in S \\ \ell \geq 1 \end{subarray}} 
\lm_{\ell,m} \begin{pmatrix} \ell \ph_{\ell,m} \\  - \om m \ph_{\ell,-m} \end{pmatrix}
: \lm_{\ell,m} \in \R \bigg\}.
\end{equation}
Its orthogonal complement in $L^2(\S^2) \times L^2(\S^2)$ 
(we denote $L^2(\S^2) = L^2(\S^2,\R)$) 
is the vector space
\begin{multline*}
W := \bigg\{ 
\begin{pmatrix} \eta \\ \beta \end{pmatrix} 
= \lm_{0,0} \begin{pmatrix} \ph_{0,0} \\ 0 \end{pmatrix} 
+ \sum_{ \begin{subarray}{c} (\ell,m) \in S \\ \ell \geq 1 \end{subarray}}
\lm_{\ell,m} \begin{pmatrix} \om m \ph_{\ell,m} \\ \ell \ph_{\ell,-m} \end{pmatrix}
+ \sum_{(\ell,m) \in \mT \setminus S} 
\begin{pmatrix} \hat \eta_{\ell,m} \ph_{\ell,m} \\  \hat \beta_{\ell,m} \ph_{\ell,m} \end{pmatrix}
\\ 
: \lm_{\ell,m}, \hat \eta_{\ell,m}, \hat \beta_{\ell,m}  \in \R, \  
(\eta, \beta) \in L^2(\S^2) \times L^2(\S^2) \bigg\}.
\end{multline*} 
Thus $L^2(\S^2) \times L^2(\S^2) = V \oplus W$, 
and $V$ and $W$ are orthogonal 
with respect to the scalar product of $L^2(\S^2) \times L^2(\S^2)$.  
Moreover 
\[
H^{s+\frac32}(\S^2) \times H^{s+1}(\S^2) 
= V \oplus W^s, \quad 
W^s := W \cap (H^{s+\frac32}(\S^2) \times H^{s+1}(\S^2)). 
\]

For $(\ell,m) \in \mT \setminus S$, 
given any $\hat f_{\ell,-m} , \hat g_{\ell,m}$, 
there exists a unique solution of system \eqref{syst.coeff.matrix}, 
which is 
\begin{equation} \label{inv.formula.outside.S}
\hat \eta_{\ell,m} 
= \frac{\om m \hat f_{\ell,-m} + \ell \hat g_{\ell, m} }{ \det L_{\ell,m} }, 
\quad \ 
\hat \beta_{\ell,-m} 
= \frac{ - \sigma_0 (\ell+2)(\ell-1) \hat f_{\ell,-m} - \om m \hat g_{\ell, m} }{ \det L_{\ell,m} }.
\end{equation}
For $(\ell,m) \in S$, we distinguish the cases $\ell = 0$ and $\ell > 0$. 
For $(\ell,m) = (0,0)$, system \eqref{syst.coeff.matrix} has a solution only if $\hat f_{0,0} = 0$, 
and, in that case, the solutions are the pairs $(\hat \eta_{0,0} , \hat \beta_{0,0})$ 
with 
\begin{equation} \label{inv.formula.S.0}
\hat \eta_{0,0} = - (2 \sigma_0)^{-1} \hat g_{0,0}, \quad 
\hat \beta_{0,0} \in \R.
\end{equation} 
For $(\ell,m) \in S$ with $\ell \geq 1$, 
system \eqref{syst.coeff.matrix} has a solution only if 
$\om m \hat f_{\ell,-m} + \ell \hat g_{\ell,m} = 0$,
and, in that case, the solutions are the pairs $(\hat \eta_{\ell,m} , \hat \beta_{\ell,-m})$ 
with 
\begin{equation} \label{inv.formula.S.geq.1}
\hat \beta_{\ell,m} = - \ell^{-1} (\om m \hat \eta_{\ell,m} + \hat f_{\ell,-m}), \quad
\hat \eta_{\ell,m} \in \R. 
\end{equation}
Hence the range of $L$ is contained in the space 
\begin{multline} 
R := \bigg\{ 
\begin{pmatrix} f \\ g \end{pmatrix} 
= \hat g_{0,0} \begin{pmatrix} 0 \\ \ph_{0,0} \end{pmatrix} 
+ \sum_{ \begin{subarray}{c} (\ell,m) \in S \\ \ell \geq 1 \end{subarray}}
\hat f_{\ell,-m} \begin{pmatrix} \ph_{\ell,-m} \\ -\om m \ell^{-1} \ph_{\ell,m} \end{pmatrix}
+ \sum_{(\ell,m) \in \mT \setminus S} 
\begin{pmatrix} \hat f_{\ell,m} \ph_{\ell,m} \\ \hat g_{\ell,m} \ph_{\ell,m} \end{pmatrix}  
\\ 
: \hat f_{\ell,m}, \hat g_{\ell,m} \in \R, \ 
(f,g) \in L^2(\S^2) \times L^2(\S^2) \bigg\}.
\label{def.R}
\end{multline}
The orthogonal complement of $R$ with respect to the scalar product of 
$L^2(\S^2) \times L^2(\S^2)$ is the finite-dimensional space 
\begin{equation} \label{def.Z} 
Z := \bigg\{ 
\begin{pmatrix} f \\ g \end{pmatrix} 
= \lm_{0,0} \begin{pmatrix} \ph_{0,0} \\ 0 \end{pmatrix} 
+ \sum_{ \begin{subarray}{c} (\ell,m) \in S \\ \ell \geq 1 \end{subarray}} 
\lm_{\ell,m} \begin{pmatrix} \om m \ph_{\ell,-m} \\ \ell \ph_{\ell,m} \end{pmatrix}
: \lm_{\ell,m} \in \R \bigg\}.
\end{equation}
Thus, $L^2(\S^2) \times L^2(\S^2) = R \oplus Z$, 
and $R$ and $Z$ are orthogonal 
with respect to the scalar product of $L^2(\S^2) \times L^2(\S^2)$.  
Moreover 
\[
H^s(\S^2) \times H^{s-\frac12}(\S^2) 
= R^s \oplus Z, \quad 
R^s := R \cap (H^s(\S^2) \times H^{s-\frac12}(\S^2)).
\]

Let $L|_{W^s} : W^s \to R^s$ be the restriction of $L$ to $W^s$.  

\begin{lemma} 
\label{lemma:L.inv}
The linear map $L|_{W^s} : W^s \to R^s$ is bijective. 
Its inverse $(L|_{W^s})^{-1} : R^s \to W^s$ is bounded, with
\begin{equation} \label{inv.est.L}
\| (L|_{W^s})^{-1} (f,g) \|_{H^{s+\frac32}(\S^2) \times H^{s+1}(\S^2)} 
\leq C \| (f,g) \|_{H^{s}(\S^2) \times H^{s-\frac12}(\S^2)} 
\end{equation}
for all $(f,g) \in R^s$. 
The constant $C$ depends on $\sigma_0, \ell_0, s$.  
\end{lemma}

To prove the lemma, we use spectral norms for the spaces $H^s(\S^2)$.
Given any $\eta \in H^s(\S^2)$ written as the series of spherical harmonics \eqref{def.mT}, 
we define 
\begin{equation} \label{def.spec.norm}
\| \eta \|_{*,s} := 
\Big( |\hat \eta_{0,0}|^2 
+ \sum_{(\ell,m) \in \mT \setminus \mT_0} |\hat \eta_{\ell,m}|^2 \ell^{2s} \Big)^{\frac12}.
\end{equation}

\begin{proof}[Proof of Lemma \ref{lemma:L.inv}]
The map $L|_{W^s}$ is injective on $W^s$ by construction; 
we prove that it is surjective onto $R^s$.
Let $(f,g) \in R^s$, with coefficients $(\hat f_{\ell,m}, \hat g_{\ell,m})$.
For $(\ell,m) \in \mT \setminus S$, the solution $(\hat \eta_{\ell,m}, \hat \beta_{\ell,m})$ 
of system \eqref{syst.coeff.matrix} is uniquely determined by \eqref{inv.formula.outside.S}. 
For $(\ell,m) = (0,0)$, the infinitely many solutions $(\hat \eta_{0,0}, \hat \beta_{0,0})$  
of system \eqref{syst.coeff.matrix} are given by \eqref{inv.formula.S.0},  
and the condition $(\eta, \beta) \in W$ selects just one of these, 
which is $\hat \beta_{0,0} = 0$.  
For $(\ell,m) \in S$ with $\ell \geq 1$, the infinitely many solutions 
$(\hat \eta_{\ell,m}, \hat \beta_{\ell,-m})$ of \eqref{syst.coeff.matrix} 
are given by \eqref{inv.formula.S.geq.1},  
and the condition $(\eta, \beta) \in W$ selects just one of these, which is
\begin{equation} \label{inv.formula.S.geq.1.selected.W}
\hat \eta_{\ell,m} 
= \frac{- \om m \hat f_{\ell,-m}}{\ell^2 + \om^2 m^2}
= \frac{\ell \hat g_{\ell,m}}{\ell^2 + \om^2 m^2}, 
\quad 
\hat \beta_{\ell,-m} 
= \frac{- \ell \hat f_{\ell,-m}}{\ell^2 + \om^2 m^2}.
\end{equation}
Hence the inversion problem $L(\eta,\beta) = (f,g)$, $(\eta,\beta) \in W^s$ 
has a unique candidate solution $(\eta,\beta)$ determined by its coefficients 
$(\hat \eta_{\ell,m}, \hat \beta_{\ell,m})$. 
We have to prove that this candidate is an element of $W^s$. 
For $(\ell,m) \in \mT \setminus S$, formula \eqref{inv.formula.outside.S}
and bound \eqref{bound.det.L.ell.j} imply that 
\begin{equation} \label{bound.coeff.outside.S}
|\hat \eta_{\ell,m}| \leq C \ell^{-2} ( |\hat f_{\ell,-m}| + |\hat g_{\ell,m}|), 
\quad 
|\hat \beta_{\ell,-m}| \leq C ( \ell^{-1} |\hat f_{\ell,-m}| + \ell^{-2} |\hat g_{\ell,m}|),
\end{equation}
for some constant $C>0$ depending on $\sigma_0, \ell_0$. 
From \eqref{bound.coeff.outside.S} it follows that 
\begin{equation} \label{bound.coeff.outside.S.with.s}
\ell^{s+\frac32} |\hat \eta_{\ell,m}| + \ell^{s+1} |\hat \beta_{\ell,-m}| 
\leq 2C ( \ell^s |\hat f_{\ell,-m}| + \ell^{s-\frac12} |\hat g_{\ell,m}| ) 
\end{equation}
for all $(\ell,m) \in \mT \setminus S$. 
For $(\ell,m) = (0,0)$ one has $|\hat \eta_{0,0}| = (2 \sigma_0)^{-1} |\hat g_{0,0}|$, 
$\hat \beta_{0,0} = 0$. 
For $(\ell,m) \in S$ with $\ell \geq 1$, \eqref{inv.formula.S.geq.1} implies that 
$|\hat \eta_{\ell,m}| \leq \ell^{-1} |\hat g_{\ell,m}|$  
and $|\hat \beta_{\ell,m}| \leq \ell^{-1} |\hat f_{\ell,m}|$,
but $\ell$ is in the bounded interval $1 \leq \ell \leq c_0$, 
see Lemma \ref{lemma:bound.det.L.ell.j}, 
and therefore $\ell^{s+\frac12} \leq c_0 \, \ell^{s-\frac12}$. 
This implies that inequality \eqref{bound.coeff.outside.S.with.s} also holds 
for $(\ell,m) \in S$ with $\ell \geq 1$, 
with a (possibly different) constant $C$ depending on $\s_0,\ell_0$.
As a consequence, one has 
\[
\| \eta \|_{*,s+\frac32}^2 + \| \beta \|_{*,s+1}^2 
\leq C ( \| f \|_{*,s}^2 + \| g \|_{*,s-\frac12}^2 ).
\]
Hence $(\eta, \beta) \in W^s$, 
the inverse map $(L|_{W^s})^{-1} : R^s \to W^s$ is well-defined and bounded, 
and the proof is complete. 
\end{proof}

\subsection{Symmetries and bifurcation from a simple eigenvalue}

The set $S$ in Lemma \ref{lemma:bound.det.L.ell.j} 
has at least the 4 elements $(0,0)$, $(1,0)$, $(\ell_0, \pm m_0)$,
and consequently the kernel of $L$ has dimension $\dim V \geq 4$. 
In this subsection we use the symmetries of equation \eqref{eq.mF} 
to reduce the problem to the case of bifurcation from a simple eigenvalue. 

First, we observe that the space of functions that are even in $x_3$ is an invariant set
for the map $\mF(\om, \cdot)$. 

\begin{lemma} 
\label{lemma:mF.x3}
If $(\eta, \beta)$ is even in $x_3$, 
then also $\mF(\om, \eta, \beta)$ is an even function of $x_3$. 
\end{lemma}

\begin{proof}
Let $M = \mathrm{diag}(1, 1, -1)$ be the $3 \times 3$ matrix 
that maps $x = (x_1, x_2, x_3) \mapsto Mx = (x_1, x_2, - x_3)$ 
for all $x \in \R^3$. Then $M$ is an orthogonal matrix, i.e., 
$M M^T = M^T M = I$, and Lemma \ref{lemma:pax.transf.law} applies to $M$.
Now a function $\eta$ defined on $\S^2$ is even in $x_3$ 
if $\eta(Mx) = \eta(x)$ for all $x \in \S^2$; 
in the notation of Lemma \ref{lemma:pax.transf.law}, 
this means that $\eta(x) = \eta(Mx) = h(x)$, i.e., $\eta = h$.
Hence, if $\eta, \beta$ are even in $x_3$, then all the properties 
of Lemma \ref{lemma:pax.transf.law} hold with $h = \eta$ and $\psi = \beta$. 
In particular, $G(\eta)\beta$, $\la \grad_{\S^2} \eta, \grad_{\S^2} \beta \ra$, 
$|\grad_{\S^2} \eta|^2$, $H_\eta$ are all even in $x_3$. 
By \eqref{def.mL}, $\mM \eta$, $\mM \beta$ are also even in $x_3$ 
because $M^T \mJ M = \mJ$. 
Recalling the definition \eqref{def.mF.1}, \eqref{def.mF.2} of $\mF$, 
the proof is complete. 
\end{proof}

\begin{lemma}
\label{lemma:parity.x3.ph.ell.m}
The spherical harmonics $\ph_{\ell,m}$ of Lemma \ref{lemma:block.diag.mS} satisfy
\[
\ph_{\ell,m} (x_1, x_2, - x_3) = (-1)^{\ell-m} \ph_{\ell,m}(x) 
\quad \forall x = (x_1, x_2, x_3) \in \S^2, 
\quad \forall (\ell, m) \in \mT.
\]
\end{lemma}

\begin{proof}
The ordinary Legendre polynomial $P_\ell(t)$ of degree $\ell$ 
has parity $P_\ell(-t) = (-1)^\ell P_\ell(t)$. 
Hence its $m$-th derivative has parity $(-1)^{\ell-m}$, i.e., 
\[
P_\ell^{(m)}(-t) = (-1)^{\ell - m} P_\ell^{(m)}(t),
\quad m = 1, \ldots, \ell.
\]
Then, by \eqref{def.ph.ell.m.Re.Im}, $\ph_{\ell,0}$ has parity $(-1)^\ell$ 
as a function of $x_3$, 
and $\ph_{\ell,m}^{(\Re)}, \ph_{\ell,m}^{(\Im)}$, 
$m = 1, \ldots, \ell$,  
have parity $(-1)^{\ell-m}$ in $x_3$.
Since $(-1)^{\ell+m} = (-1)^{\ell-m}$, by \eqref{def.ph.ell.m} we obtain the thesis.
\end{proof}

If a function $\eta$ is even in $x_3$, 
then, by Lemma \ref{lemma:parity.x3.ph.ell.m}, 
only the spherical harmonics $\ph_{\ell,m}$ that are even in $x_3$ 
appear in its expansion, i.e., 
$\hat \eta_{\ell,m} = 0$ for all $(\ell,m) \in \mT$ such that $\ell - m$ is odd, 
and only coefficients $\hat \eta_{\ell,m}$ with $\ell - m$ even  
can be nonzero.

\medskip

Now we consider the parity with respect to $x_2$, 
and prove that $\mF(\om, \cdot)$ changes that parity, 
as it maps any (even, odd) pair into an (odd, even) one. 

\begin{lemma}
\label{lemma:mF.x2}
If $\eta$ is even in $x_2$ and $\beta$ is odd in $x_2$, 
then $\mF_1(\om,\eta,\beta)$ is odd in $x_2$ 
and $\mF_2(\om,\eta,\beta)$ is even in $x_2$.
\end{lemma}

\begin{proof}
Consider the matrix $M = \mathrm{diag}\,(1, -1, 1)$ 
that maps $x = (x_1, x_2, x_3)$ into $Mx = (x_1, - x_2, x_3)$ 
for all $x \in \R^3$. Then $M$ is an orthogonal matrix, i.e., 
$M M^T = M^T M = I$, and Lemma \ref{lemma:pax.transf.law} applies to $M$.
Let $\eta, \beta$ be defined on $\S^2$, and let $h(x) := \eta(Mx)$, 
$\psi(x) := \beta(Mx)$ for all $x \in \S^2$. 
Assume that $\eta$ is even in $x_2$ and that $\beta$ is odd in $x_2$. 
Then $h = \eta$ and $\psi = - \beta$. 
By Lemma \ref{lemma:pax.transf.law}, we deduce that 
$(1+\eta)$, $(1+\eta)^2$, $|\grad_{\S^2} \eta|^2$, 
$H_\eta$, $|\grad_{\S^2} \beta|^2$ are even in $x_2$, 
while $G(\eta)\beta$ and $\la \grad_{\S^2} \beta, \grad_{\S^2} \eta \ra$ are odd in $x_2$. 
Also, $\mM \eta$ is odd in $x_2$ and $\mM \beta$ is even in $x_2$, because 
$(\mM \eta)(Mx) = \la \mJ M x , (\grad_{\S^2} \eta)(Mx) \ra 
= \la M^T \mJ M x, \grad_{\S^2} h(x) \ra$, 
similarly $(\mM \beta)(Mx) = \la M^T \mJ M x, \grad_{\S^2} \psi(x) \ra$, 
and $M^T \mJ M = - \mJ$. 
By the definition \eqref{def.mF.1}, \eqref{def.mF.2} of $\mF_1$, $\mF_2$, 
the proof is complete. 
\end{proof}

\begin{lemma}
\label{lemma:parity.x2.ph.ell.m}
The spherical harmonics $\ph_{\ell,m}$ of Lemma \ref{lemma:block.diag.mS} satisfy
\begin{alignat*}{2}
\ph_{\ell,m} (x_1, - x_2, x_3) 
& = \ph_{\ell,m}(x),  
\quad & m & = 0, \ldots, \ell,
\\
\ph_{\ell,m} (x_1, - x_2, x_3) 
& = - \ph_{\ell,m}(x), 
\quad & m & = -\ell, \ldots, -1,
\end{alignat*}
for all $x = (x_1, x_2, x_3) \in \S^2$, 
all $\ell \in \N_0$.
\end{lemma}

\begin{proof}
For $m=0$, the spherical harmonic $\ph_{\ell,0}$ in \eqref{def.ph.ell.m.Re.Im} 
does not depend on $x_2$, therefore it is even in $x_2$. 
For $m = 1, \ldots, \ell$, one has ${\overline w}^{\,m} = \overline{w^m}$ for all $w \in \C$,
that is, complex conjugation and $m$th power commute in $\C$, 
whence 
\begin{align*}
\Re \{ (x_1 - i x_2)^m \} 
= \Re \{ (x_1 + i x_2)^m \}, 
\quad \ 
\Im \{ (x_1 - i x_2)^m \} 
= - \Im \{ (x_1 + i x_2)^m \} 
\end{align*}
for all $(x_1, x_2) \in \R^2$, 
and, recalling \eqref{def.ph.ell.m.Re.Im},  
\[
\ph_{\ell,m}^{(\Re)}(x_1, -x_2, x_3) 
= \ph_{\ell,m}^{(\Re)}(x), 
\quad \ 
\ph_{\ell,m}^{(\Im)}(x_1, -x_2, x_3)
= - \ph_{\ell,m}^{(\Im)}(x)
\]
for all $x \in \S^2$. By \eqref{def.ph.ell.m} we obtain the thesis.
\end{proof}

If $\eta$ is even in $x_2$ and $\beta$ is odd in $x_2$,  
then, by Lemma \ref{lemma:parity.x2.ph.ell.m}, 
their coefficients satisfy
$\hat \eta_{\ell,m} = 0$ for $m < 0$ 
and $\hat \beta_{\ell,m} = 0$ for $m \geq 0$, 
and only coefficients $\hat \eta_{\ell,m}$ with $m \geq 0$ 
and $\hat \beta_{\ell,m}$ with $m < 0$ 
can be nonzero.

We put together the properties of parity with respect to $x_2$ and $x_3$, 
and define the two subspaces
\begin{align} 
X & := \{ f \in L^2(\S^2,\R) : f = \text{even}(x_2), \ \text{even}(x_3) \}, 
\notag \\
Y & := \{ f \in L^2(\S^2,\R) : f = \text{odd}(x_2), \ \text{even}(x_3) \}.
\label{def.X.Y}
\end{align}
Hence any $\eta \in X$, $\beta \in Y$ have expansion 
\begin{align}
& \eta = \sum_{(\ell,m) \in \mT_X} \hat \eta_{\ell,m} \ph_{\ell,m}, 
\quad 
\beta = \sum_{(\ell,m) \in \mT_Y} \hat \beta_{\ell,m} \ph_{\ell,m},
\label{stilo.01}
\\
& \mT_X 
:= \{ (\ell,m) \in \N_0^2 : 0 \leq m \leq \ell, 
\  \ell - m = \text{even} \}, 
\notag \\
& \mT_Y 
:= \{ (\ell,m) \in \N_0 \times \Z : - \ell \leq m \leq -1, 
\  \ell - m = \text{even} \}.
\notag
\end{align}
By Lemmas \ref{lemma:mF.Hs}, \ref{lemma:mF.x3} and \ref{lemma:mF.x2}, 
the domain and codomain of the map $\mF$ can be restricted to the subspaces 
$X \times Y$ and $Y \times X$ respectively, namely 
\begin{equation}  \label{mF.restricted}
\mF_{\res} : \R \times (U \cap (X \times Y))  
\to (H^{s}(\S^2) \times H^{s-\frac12}(\S^2)) \cap (Y \times X),
\end{equation} 
where the index ``$\res$'' indicates this restriction. 
The linearized operator $L_\res = \pa_u \mF_\res(\om,0)$ 
is $L$ restricted to $X \times Y$. 
The kernel of $L_\res$ is $V_\res := V \cap (X \times Y)$, 
its complement in $X \times Y$ is $W_\res := W \cap (X \times Y)$, 
the range of $L_\res$ is contained in $R_\res := R \cap (Y \times X)$, 
whose complement in $Y \times X$ is $Z_\res := Z \cap (Y \times X)$. 
We calculate 
\begin{align} 
V_\res 
& = V \cap (X \times Y) = \bigg\{ 
\begin{pmatrix} \eta \\ \beta \end{pmatrix} 
= \sum_{ \begin{subarray}{c} 
(\ell,m) \in S_\res 
\end{subarray}} 
\lm_{\ell,m} \begin{pmatrix} \ell \ph_{\ell,m} \\  - \om m \ph_{\ell,-m} \end{pmatrix}
: \lm_{\ell,m} \in \R \bigg\},
\label{V.res}
\\ 
S_\res 
& := \{ (\ell,m) \in S : \ell \geq 1, \, m \geq 0, \, \ell - m = \text{even} \},
\label{S.res}
\end{align}
and we note that among the 4 elements $(0,0), (1,0), (\ell_0, \pm m_0)$ of $S$ 
listed in Lemma \ref{lemma:bound.det.L.ell.j}, 
only $(\ell_0, m_0)$ belongs to $S_\res$. 
Hence, if $S$ contains only those 4 elements, 
then $V_\res$ is a 1-dimensional space,
and $\om$ in \eqref{om.fix} is a simple eigenvalue of $L_\res$. 
Now we check that Lemma \ref{lemma:L.inv} also holds on the restricted spaces 
$W^s_\res := W^s \cap (X \times Y)$, 
$R^s_\res := R^s \cap (Y \times X)$.

\begin{lemma} 
\label{lemma:L.inv.res}
The map $L_{\res}|_{W_\res^s} : W^s_\res \to R^s_\res$
is invertible, with bounded inverse satisfying estimate \eqref{inv.est.L} 
for all $(f,g) \in W^s_\res$. 
\end{lemma}

\begin{proof}
Let $(f,g) \in R^s_\res = R^s \cap (Y \times X)$. 
By Lemma \ref{lemma:L.inv}, we already know that 
there exists a unique $(\eta, \beta) = (L|_{W^s})^{-1}(f,g) \in W^s$ 
such that $L(\eta, \beta) = (f,g)$, with inequality \eqref{inv.est.L}. 
We only have to prove that $(\eta, \beta) \in X \times Y$. 
The coefficients of $(\eta, \beta)$ are determined by those of $(f,g)$ 
by explicit formulas: they are given 
by \eqref{inv.formula.outside.S} 
for $(\ell,m) \in \mT \setminus S$, 
by \eqref{inv.formula.S.geq.1.selected.W}
for $(\ell,m) \in S$ with $\ell \geq 1$, 
and by \eqref{inv.formula.S.0} with $\hat \beta_{0,0} = 0$ 
for $(\ell,m) = (0,0)$. 
If $(\ell,m) \in \mT$ and $\ell - m$ is an odd integer, 
then $\hat f_{\ell,-m} = \hat g_{\ell,m} = 0$ because both $f$ and $g$ are even in $x_3$, 
and hence $\hat \eta_{\ell,m} = \hat \beta_{\ell,-m} = 0$ 
from the explicit formulas just mentioned. 
This implies that both $\eta$ and $\beta$ are even in $x_3$.
If $(\ell,m) \in \mT$ with $m < 0$, 
then $\hat g_{\ell,m} = 0$ because $g$ is even in $x_2$ 
and $\hat f_{\ell,-m} = 0$ because $-m > 0$ and $f$ is odd in $x_2$. 
Then, again from the explicit formulas, 
$\hat \eta_{\ell,m} = \hat \beta_{\ell,-m} = 0$ for $m < 0$.   
This implies that $\eta$ is even in $x_2$. 
Moreover, for $m=0$, one has $\hat f_{\ell,0} = 0$ because $f$ is odd in $x_2$, 
and therefore, from the explicit formulas, $\hat \beta_{\ell,0} = 0$. 
Hence $\beta$ is odd in $x_2$. 
\end{proof}

To obtain the bifurcation from a simple eigenvalue, 
it only remains to check the following transversality property. 
Recall that $\om$ is given by \eqref{om.fix}, and it is nonzero. 

\begin{lemma} 
Let $(\ell, m) \in S$, with $\ell \geq 1$ and $1 \leq |m| \leq \ell$. 
Let $\eta = \ell \ph_{\ell,m}$ and $\beta = - \om m \ph_{\ell, -m}$. 
Then the pair $(\pois \eta, \pois \beta)$ does not belong to $R$. 
\end{lemma}

\begin{proof} 
By Lemma \ref{lemma:block.diag.mS}, one has $\pois \eta = - \ell m \ph_{\ell, -m}$
and $\pois \beta = - \om m^2 \ph_{\ell, m}$. 
Hence, by \eqref{def.R}, the pair $(\pois \eta, \pois \beta)$ 
belongs to $R$ if and only if 
\[
\begin{pmatrix} 
- \ell m \ph_{\ell, -m} \\
- \om m^2 \ph_{\ell, m} 
\end{pmatrix}
= \lm 
\begin{pmatrix} 
\ph_{\ell, -m}  \\ 
- \om m \ell^{-1} \ph_{\ell,m} 
\end{pmatrix}
\quad \exists \lm \in \R,
\]
and, since $\om \neq 0$, this is possible only for $m=0$. 
\end{proof}

By the classical results of bifurcation from a simple eigenvalue, 
we obtain the following result.

\begin{theorem} \label{thm:bif}
Let $\sigma_0 > 0$. 
Let $\ell_0, m_0$ be integers with $\ell_0 \geq 2$, $1 \leq m_0 \leq \ell_0$
and $\ell_0 - m_0$ even. 
Assume that the Diophantine equation in the unknowns $(\ell,m)$
\begin{equation} \label{dioph.eq.travelling.in.the.thm}
m_0^2 (\ell+2)(\ell-1)\ell = (\ell_0+2)(\ell_0-1)\ell_0 m^2
\end{equation} 
has only the solution $(\ell,m) = (\ell_0, m_0)$ in the finite set 
\begin{align}
T & = \big\{ (\ell,m) \in \Z^2 : 
1 \leq \ell \leq c_0, \ \ 
0 \leq m \leq \ell, \ \ 
\ell - m = \text{even} \big\}, 
\notag \\ 
c_0 & = (\ell_0+2)(\ell_0-1)\ell_0 m_0^{-2}.
\label{c0.in.the.thm}
\end{align}
Then the value 
\begin{equation} \label{om.fix.in.the.thm}
\om_* = \sqrt{\sigma_0} \frac{ \sqrt{(\ell_0+2)(\ell_0-1)\ell_0} }{m_0}
\end{equation}
for the angular velocity parameter $\om$ is a bifurcation point. 
For $k \geq 2$ integer, the set of nontrivial solutions of equation $\mF(\om,u) = 0$ 
near $(\om_*,0)$ in $\R \times (V_\res \oplus W^{k - \frac12}_\res)$ 
is a unique analytic curve with parametric representation 
on the 1-dimensional space $V_\res$.
\end{theorem}

\begin{proof} 
Equation \eqref{dioph.eq.travelling.in.the.thm} is \eqref{dioph.eq.travelling},  
and \eqref{om.fix.in.the.thm} is \eqref{om.fix}. 
As observed in Lemma \ref{lemma:bound.det.L.ell.j}, 
equation \eqref{dioph.eq.travelling.in.the.thm} has no solution with $\ell > c_0$, 
and therefore the set $S_\res$ in \eqref{S.res} 
has only one element, the pair $(\ell_0, m_0)$.  
Recalling \eqref{def.L}, the mixed second derivative $\pa^2_{\om u} \mF(\om,0)$ 
is the operator $(\eta, \beta) \mapsto (\pois \eta, \pois \beta)$. 
Thus, by the analysis above, the thesis follows from 
a direct application of the classical theory of bifurcation 
from a simple eigenvalue; 
see, e.g., \cite{Ambrosetti.Prodi}, Theorem 4.1 in Section 5.4, 
and \cite{Buffoni.Toland}.  
The use of symmetries to obtain a simple eigenvalue is also contained, e.g.,  
in \cite{Ambrosetti.Prodi}, Sections 6.3 and 6.4.
\end{proof}

\subsection{Arithmetics of simple eigenvalues}

Using the prime factor decomposition of integers, 
it is not difficult to see that there exist
both pairs $(\ell_0, m_0)$ that satisfy the assumptions of Theorem \ref{thm:bif}
and pairs that do not satisfy them. 
By direct check, we have studied the following few cases of small integers. 

\begin{lemma} 
$(i)$ For $(\ell_0, m_0) =$ 
$(2,2)$, 
$(3,3)$, 
$(4,2)$, $(4,4)$, 
$(5,5)$, 
$(6,4)$, $(6,6)$,
$(7,5)$, $(7,7)$,
the assumptions of Theorem \ref{thm:bif} are satisfied, 
and hence the set $S_\res$ in \eqref{S.res} has only one element, 
the pair $(\ell_0, m_0)$ itself. 

$(ii)$ For $(\ell_0, m_0) =$ 
$(3,1)$, $(5,3)$, $(5,1)$, $(16, 16)$, 
the assumptions of Theorem \ref{thm:bif} are not satisfied,
and the corresponding set $S_\res$ is, respectively, 
\begin{alignat*}{2}
S_\res & = \{ (3,1), \, (10, 6), \, (16,12) \}, 
\quad & 
S_\res & = \{ (5,3), \, (8, 6) \}, 
\\ 
S_\res & = \{ (5,1), \, (8, 2), \, (126, 120) \},
\quad & 
S_\res & = \{ (16,16), \, (10, 8) \}. 
\end{alignat*}
\end{lemma} 

\begin{proof} 
Using the prime factor decomposition, the proof is a bit long but completely elementary. 
\end{proof}

As the previous lemma shows, it seems hard to guess a simple criterion 
that determines whether a given pair $(\ell_0, m_0)$ satisfies the assumptions of Theorem \ref{thm:bif}. 
Nonetheless, again using prime numbers, we can prove that 
there are infinitely many pairs $(\ell_0, m_0)$ that satisfy those assumptions.

\begin{proposition} \label{prop:infinity-family}
$(i)$ For every prime integers
$2 < p_1 < p_2 < \dots < p_n$, $n \geq 1$, 
given $\ell_0 = p_1 p_2 \cdots p_n$, 
the pair $(\ell_0, m_0) = (\ell_0, \ell_0)$  
satisfies the assumptions of Theorem \ref{thm:bif}. 

$(ii)$ For every prime integer $\ell_0 \geq 11$, 
the pair $(\ell_0, m_0) = (\ell_0, \ell_0 - 2)$  
satisfies the assumptions of Theorem \ref{thm:bif}. 

$(iii)$ For every prime integer $p>3$, given $\ell_0 = 2p$, 
the pair $(\ell_0, m_0) = (\ell_0, \ell_0)$  
satisfies the assumptions of Theorem \ref{thm:bif}. 

In particular, there are infinitely many pairs $(\ell_0, m_0)$ satisfying 
the assumptions of Theorem \ref{thm:bif}.
\end{proposition}

\begin{proof}
$(i)$ For $m_0 = \ell_0 > 0$, equation \eqref{dioph.eq.travelling.in.the.thm} becomes 
\begin{equation} \label{dioph.3.0}
\ell_0 (\ell+2)(\ell-1)\ell = (\ell_0+2)(\ell_0-1) m^2,
\end{equation} 
and $c_0$ in \eqref{c0.in.the.thm} becomes 
$c_0 = (\ell_0^2 + \ell_0 - 2) \ell_0^{-1} < \ell_0 + 1$. 
Assume that $(\ell,m) \in T$ solves \eqref{dioph.3.0}. 
Recall that $\ell_0$ is the product $p_1 \cdots p_n$, 
and observe that, for any $i \in \{1,\dots, n\}$, 
$p_i$ divides neither $\ell_0+2$ nor $\ell_0-1$, since $p_1>2$. 
Hence, for all $i$, the prime $p_i$ must divide $m^2$, and therefore $m$. 
Thus $m = \kappa \ell_0$ for some integer $\kappa \geq 1$. 
Since $m \leq \ell \leq c_0 < \ell_0+1$, we immediately have $\kappa=1$. 
Hence $\ell_0 = m \leq \ell < \ell_0 + 1$, whence $\ell = \ell_0$. 

$(ii)$ For $m_0 = \ell_0 - 2 > 0$, equation \eqref{dioph.eq.travelling.in.the.thm} becomes 
\begin{equation}\label{dioph.eq.cor}
(\ell_0-2)^2 (\ell +2)(\ell-1)\ell = (\ell_0 +2)(\ell_0-1)\ell_0 m^2,
\end{equation}
and $c_0$ in \eqref{c0.in.the.thm} becomes 
$c_0 = \ph(\ell_0) \ell_0$, where $\ph$ is the function 
$\ph(x) = (x+2)(x-1) (x-2)^{-2} 
= 1 + 5 (x-2)^{-1} + 4 (x-2)^{-2}$, decreasing in $x \in (2, \infty)$. 
For $\ell_0 \geq 11$, one has $c_0 = \ph(\ell_0) \ell_0 \leq \ph(11) \ell_0$, 
and $\ph(11) = 130/81 < 2$.  
Assume that $(\ell,m) \in T$ solves \eqref{dioph.eq.cor},  
and let $\ell_0 \geq 11$ be prime. Then, by \eqref{dioph.eq.cor}, 
$\ell_0$ must divide one of the factors $\ell-1$, $\ell$, $\ell+2$. 
We consider the three cases.

{\it Case one.} Assume that $\ell_0$ divide $\ell$. 
Hence $\ell= b \ell_0$ for some integer $b \geq 1$. 
Since $(\ell,m) \in T$, one has 
$b \ell_0 = \ell \leq c_0 \leq \ph(11) \ell_0 < 2 \ell_0$, 
whence $b=1$. Then $\ell = \ell_0$, which is the trivial solution.

{\it Case two}. Assume that $\ell_0$ divides $\ell-1$. 
Hence $\ell-1 = b \ell_0$ for some integer $b \geq 1$. 
Since $(\ell,m) \in T$, one has 
$b \ell_0 < b \ell_0 + 1 = \ell \leq c_0 \leq \ph(11) \ell_0 < 2 \ell_0$, 
whence $b=1$. Then $\ell = \ell_0 + 1$, and \eqref{dioph.eq.cor} becomes
\begin{equation} \label{dioph.case.two}
(\ell_0-2)^2 (\ell_0 + 3) (\ell_0 + 1) = (\ell_0 +2)(\ell_0-1) m^2.
\end{equation}
Now $\gcd(\ell_0 + 2, \ell_0 + 1) = 1$ 
and $\gcd( \ell_0 + 2, \ell_0 + 3) = 1$ 
(consecutive integers). 
Also, $\gcd(\ell_0 + 2, \ell_0 - 2) \in \{1, 2, 4\}$ 
(their difference is 4), but $\ell_0 \pm 2$, like $\ell_0$, is odd. 
Hence $\gcd(\ell_0 + 2, \ell_0 - 2) = 1$. 
As a consequence, any divisor of $\ell_0 + 2$ divides the RHS of \eqref{dioph.case.two}
and it does not divide the LHS of \eqref{dioph.case.two}, a contradiction.

{\it Case three}. Assume that $\ell_0$ divides $\ell+2$. 
Hence $\ell+2 = b \ell_0$ for some integer $b \geq 1$. 
Since $(\ell,m) \in T$, one has 
$b \ell_0 - 2 = \ell \leq c_0 \leq \ph(11) \ell_0$, 
whence $b \ell_0 \leq \ph(11) \ell_0 + 2$ and 
$b \leq \ph(11) + (2/\ell_0) \leq \ph(11) + (2/11) < 2$.
Therefore $b=1$. 
Then $\ell = \ell_0 - 2$, and \eqref{dioph.eq.cor} becomes
\begin{equation} \label{dioph.case.three}
(\ell_0-2)^3 (\ell_0 - 3) = (\ell_0 +2)(\ell_0-1) m^2.
\end{equation}
Now since  $\ell_0$ is odd we have by arguing as above $\gcd(\ell_0-2, \ell_0+2)=1$. 
Moreover $\gcd(\ell_0-2, \ell_0-1)=1$ (consecutive integers).   
Hence $(\ell_0-2)^3$ divides $m^2$, 
namely $m^2 = b (\ell_0-2)^3$ for some integer $b \geq 1$. 
Thus $(\ell_0 - 2)^3 \leq b (\ell_0 - 2)^3 = m^2 \leq \ell^2 = (\ell_0 - 2)^2$, 
but this is impossible for $\ell_0 > 3$.

$(iii)$ As observed above, 
for $m_0 = \ell_0 > 0$, equation \eqref{dioph.eq.travelling.in.the.thm} becomes \eqref{dioph.3.0},
and $c_0$ in \eqref{c0.in.the.thm} becomes 
$c_0 = (\ell_0^2 + \ell_0 - 2) \ell_0^{-1} < \ell_0 + 1$. 
Let $m_0 = \ell_0 = 2 p$, with $p$ prime. 
Then \eqref{dioph.3.0} becomes 
\begin{equation} \label{dioph.3.3}
p (\ell+2)(\ell-1)\ell = (p+1)(2p-1) m^2.
\end{equation} 
Assume that $(\ell,m) \in T$ solves \eqref{dioph.3.3}. 
Since $p$ is prime, $\gcd(p, 2p-1) = 1$, and $\gcd(p, p+1)=1$, 
we have that $p$ divides $m^2$ and therefore $m$. 
Hence $m = \kappa p$ for some integer $\kappa \geq 1$, 
and, since $m \leq \ell \leq c_0 < \ell_0 + 1 = 2p+1$, 
one has $\kappa \in \{ 1, 2 \}$. 
If $\kappa = 2$, then $m = 2p = \ell_0$, 
and $\ell_0 = m \leq \ell \leq c_0 < \ell_0 + 1$, 
whence $\ell = \ell_0$. Therefore $(\ell,m) = (\ell_0, m_0)$, 
which is the trivial solution. 
It remains to study the case $\kappa = 1$, i.e., $m = p$. 
For $m=p$, \eqref{dioph.3.3} gives
\begin{equation} \label{dioph.3.4}
(\ell+2) (\ell-1) \ell = (p+1) (2p-1) p.
\end{equation}
Since $p$ is prime and $p > 3$, we have that $p$ divides exactly one of the three factors 
on the LHS of \eqref{dioph.3.4}. We consider the three cases. 

\textit{Case one.} Assume that $p$ divides $\ell$. 
Hence $\ell \in \{ p, 2p \}$ because $\ell \leq c_0 < 2p+1$. 
Identity \eqref{dioph.3.4} with $\ell = p$ gives $p^2 + 1 = 0$, a contradiction; 
\eqref{dioph.3.4} with $\ell = 2p$ gives $4 = 1$, a contradiction. 

\textit{Case two.} Assume that $p$ divides $\ell-1$. 
Hence $\ell - 1$ is an integer multiple of $p$, 
and, since $\ell - 1 \leq c_0 - 1 < \ell_0 = 2p$, 
one has $\ell - 1 = p$. 
Plugging $\ell = p+1$ into \eqref{dioph.3.4} gives $p=4$, which is not a prime number, 
a contradiction. 

\textit{Case three.} Assume that $p$ divides $\ell+2$. 
Hence $\ell+2 \in \{ p, 2p \}$ because $\ell + 2 \leq c_0 + 2 < 2p+3$ $< 3p$. 
For $\ell + 2 = p$, \eqref{dioph.3.4} gives $(p+7)(p-1) = 0$, a contradiction. 
For $\ell + 2 = 2p$, \eqref{dioph.3.4} gives $6 p^2 - 21 p + 13 = 0$,
but this polynomial has no integer root, a contradiction.
\end{proof}

\bigskip

\emph{Statements and Declarations.} 
The authors state that there is no conflict of interest. 
No data was used for the research described in the article.

\begin{flushright}

\textbf{Pietro Baldi}

Dipartimento di Matematica e Applicazioni ``R. Caccioppoli''

University of Naples Federico II

Via Cintia, Monte Sant'Angelo, 80126 Naples, Italy

pietro.baldi@unina.it

ORCID 0000-0002-9644-3935

\medskip

\textbf{Vesa Julin} 

Department of Mathematics and Statistics

University of Jyv\"askyl\"a

P.O. Box 35, 40014 Jyv\"askyl\"a, Finland

vesa.julin@jyu.fi

ORCID 0000-0002-1310-4904

\medskip

\textbf{Domenico Angelo La Manna}

Dipartimento di Matematica e Applicazioni ``R. Caccioppoli''

University of Naples Federico II

Via Cintia, Monte Sant'Angelo, 80126 Naples, Italy

domenicoangelo.lamanna@unina.it

ORCID 0000-0003-1900-2025

\end{flushright}
\end{document}